\documentclass[leqno,11pt]{amsart}
\usepackage{amsmath}
\usepackage{amsfonts}
\usepackage{amssymb}
\usepackage{amsthm}

\usepackage{mathrsfs}
\usepackage{dsfont}
\usepackage{txfonts}
\usepackage{mathabx}

\usepackage{graphicx}
\usepackage[all]{xy}
\usepackage{pgf,tikz}
\usepackage{subfigure}
\usepackage{tikz-3dplot}
\usetikzlibrary{patterns,arrows}
\tikzset{
    partial ellipse/.style args={#1:#2:#3}{
        insert path={+ (#1:#3) arc (#1:#2:#3)}
    }
}

\usepackage{hyperref}

\usepackage{color}
\usepackage[shortlabels]{enumitem}

\newtheorem{theorem}{Theorem}[section]
\newtheorem{lemma}[theorem]{Lemma}
\newtheorem{corollary}[theorem]{Corollary}

\newtheorem{proposition}[theorem]{Proposition}

\newtheorem{claim}[theorem]{Claim}

\newtheorem{remark}[theorem]{Remark}
\newtheorem{challenge}[theorem]{Challenge}

\theoremstyle{definition}

\newtheorem{convention}[theorem]{Convention}

\newcommand{\eps}{\varepsilon}

\title[{Ancient noncollapsed flows in $\mathbb{R}^4$}]{Classification of ancient noncollapsed flows in $\mathbb{R}^4$}

\author{Kyeongsu Choi, Robert Haslhofer}

\begin{document}

\begin{abstract}
In this paper, we classify all noncollapsed singularities of the mean curvature flow in $\mathbb{R}^4$. Specifically, we prove that any ancient noncollapsed solution either is one of the classical historical examples (namely $\mathbb{R}^j\times S^{3-j}$, $\mathbb{R}\times $2d-bowl, $\mathbb{R}\times $2d-oval, the rotationally symmetric 3d-bowl, or a cohomogeneity-one 3d-oval), or belongs to the 1-parameter family of $\mathbb{Z}_2\times \mathrm{O}_2$-symmetric 3d-translators constructed by Hoffman-Ilmanen-Martin-White, or belongs to the 1-parameter family of $\mathbb{Z}_2^2\times \mathrm{O}_2$-symmetric ancient 3d-ovals constructed by Du-Haslhofer.

In light of the five prior papers on the classification program in $\mathbb{R}^4$ from our collaborations with Du, Hershkovits, and Choi-Daskalopoulos-Sesum, the major remaining challenge is the case of mixed behaviour, where the convergence to the round bubble-sheet is fast in $x_1$-direction, but logarithmically slow in $x_2$-direction. To address this, we prove a differential neck theorem, which allows us to capture the (dauntingly small) slope in $x_1$-direction. To establish the differential neck theorem, we introduce a slew of new ideas of independent interest, including switch and differential Merle-Zaag dynamics, anisotropic barriers, and propagation of smallness estimates. Applying our differential neck theorem, we show that every noncompact strictly convex solution is selfsimilarly translating, and also rule out exotic ovals.

\end{abstract}

\maketitle

\tableofcontents

\section{Introduction}

In the study of mean curvature flow, ever since the pioneering work of Brakke \cite{Brakke_book} and Huisken \cite{Huisken_convex}, the main challenge has been to analyze its singularities. To study singularities, like for any geometric PDE, one rescales by a sequence of factors going to infinity and passes to a limit. Any such blowup limit is an ancient solution, i.e. a solution in entire  space that is defined for all sufficiently negative times.\\

The most important ancient solutions are the noncollapsed ones \cite{Andrews_noncollapsing,ShengWang,HaslhoferKleiner_meanconvex}, i.e. ancient mean-convex flows $M_t$, such that at any $p\in M_t$ the inscribed and exterior radius is at least $\alpha/H(p,t)$ for some uniform $\alpha>0$. It is known since the work of White \cite{White_size,White_nature} (see also \cite{White_subsequent,HaslhoferHershkovits_subsequent}) that all blowup limits of mean-convex mean curvature flow are noncollapsed. In fact, it is expected that any blowup limit near any generic singularity is noncollapsed \cite{Ilmanen_problems,CM_generic}. In particular, this is known in $\mathbb{R}^3$ thanks to the recent proofs of the mean-convex neighborhood conjecture \cite{CHH}, the genericity conjecture \cite{CCMS,CCS,CCMS_rev} and the multiplicity-one conjecture \cite{BK_mult_one}.\\

Understanding all noncollapsed singularities thus amounts to classifying all ancient noncollapsed flows. In $\mathbb{R}^3$ this classification problem has been studied intensively over the last 15 years, initially under additional assumptions such as selfsimilarity \cite{Wang_convex,MSHS,Haslhofer_bowl,BW}, and a definite answer has been obtained in \cite{ADS1,ADS2,BC}:

\begin{theorem}[Angenent-Daskalopoulos-Sesum, Brendle-Choi]\label{thm_ADS_BC}
Any ancient noncollapsed flow in $\mathbb{R}^3$ is, up to scaling and rigid motion, either the static plane $\mathbb{R}^2$, the round shrinking $S^2$, the round shrinking $\mathbb{R}\times S^1$, the rotationally symmetric 2d-bowl, or the rotationally symmetric 2d-oval from \cite{White_nature}.
\end{theorem}

Theorem \ref{thm_ADS_BC} confirms the intuitive picture that for the flow of surfaces all noncollapsed singularities look like neck-pinches, degenerate neck-pinches or potentially an accumulation of neck-pinches \cite{CHH_note}. The result and the techniques introduced in its proof had far reaching consequences, including the above mentioned resolution of the mean-convex neighborhood conjecture and genericity conjecture, and the well-posedness for the flow of embedded 2-spheres \cite{Brendle_sphere,HershkovitsWhite,CHH,BK_mult_one}. Moreover, it also led to a corresponding classification of $\kappa$-solutions in 3d Ricci flow \cite{Brendle_3dRicci,ABDS,BDS}, which in turn played a key role in the uniqueness of 3d Ricci flow through singularities \cite{BK_uniqueness}, and its topological and geometric applications \cite{BK_top1,BK_top2,BK_contr}.\\

In stark contrast, a classification of ancient noncollapsed flows in $\mathbb{R}^4$ until recently seemed out of reach. Fundamentally, this is because while all solutions in $\mathbb{R}^3$ turned out to be rotationally symmetric, it is known since the pioneering work of Wang \cite{Wang_convex} that the Bernstein (i.e. symmetry) theorem fails for flows in $\mathbb{R}^4$. More precisely, Hoffman-Ilmanen-Martin-White constructed a 1-parameter family of 3d-translators that interpolate between the round 3d-bowl and $\mathbb{R}\times$2d-bowl, and are only $\mathbb{Z}_2\times \mathrm{O}_2$-symmetric \cite{HIMW}. Similarly, Du and the second author constructed a 1-parameter family of 3d-ovals that interpolate between the $\mathrm{O}_2\times \mathrm{O}_2$-symmetric 3d-oval and $\mathbb{R}\times$2d-oval, and are only $\mathbb{Z}_2^2\times \mathrm{O}_2$-symmetric \cite{DH_ovals}.\\

Our main theorem provides a complete classification of all ancient noncollapsed flows in $\mathbb{R}^4$:

\begin{theorem}[classification]\label{thm_main}
Any ancient noncollapsed flow in $\mathbb{R}^4$ is, up to scaling and rigid motion,
\begin{itemize}
\item either one of the standard shrinkers $S^3$, $\mathbb{R}\times S^2$, $\mathbb{R}^2\times S^1$ or $\mathbb{R}^3$,
\item or the 3d-bowl, or $\mathbb{R}\times$2d-bowl, or belongs to the 1-parameter family of $\mathbb{Z}_2\times \mathrm{O}_2$-symmetric translators from \cite{HIMW},
\item or the $\mathbb{Z}_2\times \mathrm{O}_3$-symmetric 3d-oval, or the $\mathrm{O}_2\times \mathrm{O}_2$-symmetric 3d-oval, or $\mathbb{R}\times$2d-oval,
or belongs to the one-parameter family of $\mathbb{Z}_2^2\times \mathrm{O}_2$-symmetric 3d-ovals from \cite{DH_ovals}.
\end{itemize}
\end{theorem}

Theorem \ref{thm_main} gives a complete classification of all noncollapsed singularities of mean curvature flow in $\mathbb{R}^4$. We note that in addition to the 1-parameter families of translators and ovals discussed above, the list of course also contains all classical historical examples, in particular the rotationally symmetric 3d-bowl \cite{CSS}, and the two examples of cohomogeneity-one 3d-ovals from \cite{White_nature} and \cite{HaslhoferHershkovits_ancient}, respectively.\\

As an immediate consequence we obtain a classification of all potential blowup limits (and thus a canonical neighborhood theorem) for the mean curvature flow of mean-convex hypersurfaces in $\mathbb{R}^4$:

\begin{corollary}[blowup limits and canonical neighborhoods]\label{cor_main}
For the mean curvature flow of mean-convex hypersurfaces in $\mathbb{R}^4$ (or more generally in 4-manifolds) every blowup limit is given by one of the solutions from the above list. In particular, for every $\eps>0$ there is an $H_\eps=H_\eps(M_0)<\infty$, such that around any space-time point $(p,t)$ with $H(p,t)\geq H_\eps$ the flow is $\eps$-close to one of the solutions from the above list.
\end{corollary}

Corollary \ref{cor_main} answers several old questions of White \cite{White_nature} and Wang \cite{Wang_convex} regarding the nature of singularities in mean-convex flows. More generally, in light of \cite{CHH,CCMS_rev,BK_mult_one}, the conclusion of the corollary is also expected to hold for blowup limits near any generic singularity.\\

To conclude this overview, let us mention that Theorem \ref{thm_main}, in a spirit similar to Theorem \ref{thm_ADS_BC}, can be viewed both as the end and as the beginning of a story. On the one hand, it gives a definite and final answer to the classification problem for noncollapsed singularities in $\mathbb{R}^4$. On the other hand, it opens the door to several other major open problems that seemed previously out of reach, in particular the construction of mean-convex flow with surgery in $\mathbb{R}^4$, and the mean-convex neighborhood conjecture in $\mathbb{R}^4$. Moreover,  Theorem \ref{thm_main} also suggests a corresponding picture for $\kappa$-solutions in 4d Ricci flow, see \cite{Haslhofer_4dRicci}.
\bigskip

\subsection{The classification program in $\mathbb{R}^4$}

In this subsection, we first review what has been accomplished previously on ancient noncollapsed flows in $\mathbb{R}^4$ in the prior papers by Du and the second author \cite{DH_shape,DH_no_rotation}, our joint work with Hershkovits \cite{CHH_wing,CHH_translator}, and the joint work of the second author with Choi-Daskalopoulos-Du-Sesum \cite{CDDHS}, and then discuss what the remaining major challenges are.\\

Let $\mathcal{M}=\{M_t\}$ be an ancient noncollapsed mean curvature flow in $\mathbb{R}^4$ that is neither a static plane nor a round shrinking sphere. By general theory \cite{CM_uniqueness,HaslhoferKleiner_meanconvex} the tangent flow at $-\infty$ is either a neck or a bubble-sheet. In the neck case, thanks to the classification from \cite{ADS2,BC2} the flow $\mathcal{M}$ is either a round shrinking $\mathbb{R}\times S^2$, or the rotationally symmetric 3d-bowl or 3d-oval. We can thus assume that
\begin{equation}\label{bubble-sheet_tangent_intro1.1}
\lim_{\lambda \rightarrow 0} \lambda M_{\lambda^{-2}t}=\mathbb{R}^{2}\times S^{1}(\sqrt{-2t}).
\end{equation}
The analysis of such ancient solutions starts by consider the bubble-sheet function $u=u(y,\vartheta,\tau)$, which captures the deviation of the renormalized flow $\bar{M}_\tau=e^{\tau/2} M_{-e^{-\tau}}$ from the round bubble-sheet $\mathbb{R}^2\times S^1(\sqrt{2})$. The evolution of $u$ is governed by the Ornstein-Uhlenbeck operator
\begin{equation}\label{prel_OU}
\mathcal L=\partial_{y_1}^2+\partial_{y_2}^2-\tfrac{y_1}{2} \partial_{y_1}-\tfrac{y_2}{2} \partial_{y_2}+\tfrac{1}{2} \partial_{\vartheta}^2+1,
\end{equation}
which is self-adjoint on the Gaussian $L^2$-space, and has the unstable eigenfunctions $1,y_1,y_2,\cos \vartheta, \sin \vartheta$, 
and the neutral eigenfunctions $y_1^2-2,y_2^2-2,y_1y_2,y_1\cos\vartheta,y_1\sin\vartheta,y_2\cos\vartheta,y_2\sin\vartheta$. 
Based on these spectral properties, and taking also into account that the $\vartheta$-dependence is tiny thanks to Zhu's symmetry improvement result for bubble-sheets \cite{Zhu}, which generalizes the neck-improvement result by Brendle and the first author \cite{BC}, Du and the second author \cite{DH_shape,DH_no_rotation} proved:

\begin{theorem}[normal form]\label{thm_norm_form}
For $\tau\to -\infty$, in suitable coordinates, in Gaussian $L^2$-norm we have
\begin{equation}\label{bubble_sheet_quant1.1}
u = O(e^{\tau/2})\quad\mathrm{or}\quad
u= \frac{4-y_1^2-y_2^2}{\sqrt{8}|\tau|}+O(|\tau|^{-1-\gamma})\quad\mathrm{or}\quad 
u= \frac{2-y_2^2}{\sqrt{8}|\tau|}+O(|\tau|^{-1-\gamma}).
\end{equation}
\end{theorem}

According to Theorem \ref{thm_norm_form} (normal form), the classification problem can be split up into 3 cases, which we call the case of fast convergence, slow convergence, and mixed convergence, respectively.\\

The case of fast convergence, which is easiest case, has been settled in a joint paper with Hershkovits \cite{CHH_wing}:
\begin{theorem}[no wings]\label{thm_no_wings}
There are no wing-like ancient noncollapsed flows in $\mathbb{R}^4$. In particular, if the convergence is fast, then $\mathcal{M}$ is either a round shrinking $\mathbb{R}^2\times S^1$ or a translating $\mathbb{R}\times$2d-bowl.
\end{theorem}

Since this will play an important role later on, let us briefly recall the idea how this was proved. Assuming that the convergence is fast and that $\mathcal{M}$ is not a round shrinking $\mathbb{R}^2\times S^1$ it has been shown that there exists some nonvanishing constant vector $a=a(\mathcal{M})\in \mathbb{R}^2$ such that the bubble-sheet function $u^X$ of the renormalized flow suitably (re)centered at any space-time point $X\in \mathcal{M}$ satisfies
\begin{equation}\label{fine_expansion_fst}
u^X=(a_1 y_1 + a_2 y_2)e^{\tau/2}+ o(e^{\tau/2})
\end{equation}
for all $\tau$ sufficiently negative depending only on the bubble-sheet scale $Z(X)$. Geometrically speaking, the vector $a$ captures how the size of the $S^1$-fibres changes as one moves along the $\mathbb{R}^2$-directions. Using this fine bubble-sheet expansion, it was easy to see that the flow must be noncompact (since otherwise near opposing caps the vector $a$ would point in opposite directions), and with a bit more work it has been concluded that $\mathcal{M}$ in fact splits off a line (hence is not wing-like) and is selfsimilarly translating.\\

The case of slow convergence has been settled in joint work of the second author with Choi, Daskalopoulos, Du and Sesum \cite{CDDHS}:

\begin{theorem}[bubble-sheet ovals]\label{thm_bubble_sheet_ovals}
If the convergence is slow, then $\mathcal{M}$ is either the $\mathrm{O}_2\times \mathrm{O}_2$-symmetric 3d-oval, or belongs to the one-parameter family of $\mathbb{Z}_2^2\times \mathrm{O}_2$-symmetric 3d-ovals from \cite{DH_ovals}.
\end{theorem}

Regarding the proof, let us just mention that \eqref{bubble_sheet_quant1.1} in the case of slow convergence means inwards quadratic bending in all directions, which yields that $M_t$ is compact with axes of length approximately $\sqrt{2 |t|\log |t|}$, and hence the problem turned out to be amenable to the techniques from \cite{ADS1,ADS2}.\footnote{There have been numerous technical difficulties, most notably a delicate tensor maximum principle estimate to get quadratic concavity, but since the techniques are rather orthogonal to the techniques of the present paper let us not dwell on these points.}\\

Finally, let us recall that the special case of selfsimilarly translating solutions (with any rate of convergence) has been dealt with in an earlier joint paper with Hershkovits \cite{CHH_translator}:
\begin{theorem}[translators]\label{thm_translators_intro}
If $\mathcal{M}$ is selfsimilarly translating, then it is either $\mathbb{R}\times$2d-bowl, or belongs to the 1-parameter family of $\mathbb{Z}_2\times \mathrm{O}_2$-symmetric entire translators from \cite{HIMW}.
\end{theorem}

We also recall that for these $\mathbb{Z}_2\times \mathrm{O}_2$-symmetric translators moving with unit speed in $x_1$-direction, the level sets $M_0\cap \{x_1 = -t\}$, have the same sharp asymptotics for $t\to -\infty$ as the 2d-ovals in \cite{ADS1}.

\medskip

In light of the results reviewed above, the remaining case to address is thus the most subtle case of mixed convergence (without selfsimilarity assumption), namely
\begin{equation}\label{bubble_sheet_mixed1.1}
u= \frac{2-y_2^2}{\sqrt{8}|\tau|}+O(|\tau|^{-1-\gamma}).
\end{equation}
Here, a major challenge is that the expansion gives inwards quadratic bending in $y_2$-direction, but does not tell anything about the behaviour in $y_1$-direction. Geometrically speaking, the first challenge is thus:

\begin{challenge}[compact solutions]\label{challange1} Rule out exotic ovals, i.e. show that there is no compact ancient noncollapsed solution that satisfies \eqref{bubble_sheet_mixed1.1}.
\end{challenge}

For the potential scenario of exotic ovals, the axis in $x_2$-direction would have length approximately $\sqrt{2 |t|\log |t|}$, but the axis in $x_1$-direction would be much longer. E.g. one could image that if one could compute many further terms in the expansion \eqref{bubble_sheet_mixed1.1}, then potentially the quadratic eigenfunction $2-y_1^2$ would eventually show up with some much smaller coefficient, say with coefficient $|\tau|^{-42}=(\log |t|)^{-42}$.\\

Second, assuming one somehow solved the first challenge, one wants to show that every noncompact solution with mixed convergence is either $\mathbb{R}\times$ 2d-oval or belongs to the 1-parameter family of $\mathbb{Z}_2\times \mathrm{O}_2$-symmetric entire translators from \cite{HIMW}. In light of the above, the second major challenge is thus:

\begin{challenge}[noncompact solutions]\label{challange2} Show that every noncompact ancient noncollapsed solution with mixed convergence is a translator unless it splits off a line.
\end{challenge}

Specifically, one has to rule out the potential scenario that the solution slows down going backwards in time and/or speeds up going forwards in time. Moreover, in light of the 1-parameter family of examples from  \cite{HIMW}, the problem comes with several a priori unrelated scales, namely the curvature scale, the bubble-sheet scale and the quadratic scale. In particular, one has to rule out the potential scenario that for very negative or very positive times the solution looks almost like a 3d-bowl close to the cap and then has an uncontrollably long neck-region before it eventually transitions to a bubble-sheet region.\\

In general, the question whether an ancient (or eternal) noncompact solution is selfsimilar is very subtle, as pointed out first by White \cite[Conjecture 1]{White_nature} and Perelman \cite[Remark 11.9]{Perelman1}. There are some prior instances in the literature where selfsimilarity has been established successfully, namely \cite{BC,BC2,CHH,CHHW,Brendle_3dRicci,BN}. However, the method of proof in all these prior papers was to establish rotational symmetry, and then selfsimilarity only followed a posteriori by uniqueness of rotationally symmetric solutions. Also, if one drops the noncollapsing condition then there are (many) examples of eternal strictly convex solutions that are not selfsimilar, see \cite{BLT_GeomTop}.

\bigskip

\subsection{The differential neck theorem and its consequences} In this subsection, we state our differential neck theorem, and explain how it can be used to overcome the major challenges described above.\\

Let $\mathcal{M}=\{M_t\}$ be a strictly convex ancient noncollapsed mean curvature flow in $\mathbb{R}^4$, whose tangent flow at $-\infty$ is given by \eqref{bubble-sheet_tangent_intro1.1}, and whose bubble-sheet function has the mixed convergence behaviour \eqref{bubble_sheet_mixed1.1}.\\

Motivated by the analysis of the fast convergence, one might hope in the case of mixed convergence that $u$ as a function of $y_1$ still behaves like $ay_1 e^{\tau/2}$. The problem with this approach is that the leading term $(2-y_2^2)/(\sqrt{8}|\tau|)$ is way larger, and thus trying to detect an exponentially small correction term sounds like a rather hopeless endeavour. However, we can turn this into a more meaningful approach, by instead considering the derivative of the profile function with respect to $y_1$, which kills the leading term and thus gives some hope to be able to detect the (still dauntingly small) correction term.\\

Loosely speaking, we will show that the slope $u_1^X=\partial u^X/\partial y_1$ behaves like $ae^{\tau/2}$, where $a=a(\mathcal{M})$ is a nonvanishing constant. To describe this in more detail,
fixing a small constant $\varepsilon_0>0$, similarly as in \cite[Section 2.2]{CHH_wing}, we denote by $Z(X)$ the bubble-sheet scale of $X$, i.e. the smallest radius $r$ such that $\mathcal{D}_{1/r}(\mathcal M-X)$ is $\varepsilon_0$-close in $C^{\lfloor 1/\eps_0\rfloor}$ in $B(0,1/\eps_0)\times [-2,-1]$ to $\mathbb{R}^2\times S^1(\sqrt{-2t})$. Fixing a suitable graphical radius function $\rho$ and a suitable cutoff function $\eta$, as described later, we then prove:

\begin{theorem}[differential neck theorem]\label{thm:strong_differential_neck_intro}
There exists a constant $a=a(\mathcal{M})\neq 0$, such that for all $X\in\mathcal{M}$ the function $w^X(y,\vartheta,\tau)=\eta(|y|/\rho(\tau)) u_1^X(y,\vartheta,\tau)$ for all $\tau\leq \tau_*-\log \max\{1,Z(X)^2\}$ satisfies
\begin{equation}\label{diff_neck_exp_intr}
w^X=ae^{\frac{1}{2}\tau}+E^X,
\end{equation}
where $\tau_\ast>-\infty$ is a universal constant, and where the Gaussian $L^2$-norm of the error term can be estimated by
\begin{equation}\label{error_est}
\|E^X\|_{\mathcal{H}}\leq\frac{Ce^{\frac{1}{2}\tau}}{|\tau+\log \max\{1,Z(X)^2\}|^{\frac{1}{2}}} \left(|a|+\max\big\{1,Z(X)^{\frac{7}{3}}\big\}e^{\frac{2}{3}\tau}\right).
\end{equation}
\end{theorem}

Geometrically, the slope $ae^{\tau/2}$ of course captures how the size of the $S^1$-fibres changes as the observer moves along the $y_1$-direction. For example, for the selfsimilarly translating solutions from \cite{HIMW}, normalized such that they translate in positive $x_1$-direction with speed $1$, one has $a=1/\sqrt{2}$.  

We emphasize that Theorem \ref{thm:strong_differential_neck_intro} (differential neck theorem) is a theorem in hard analysis, where the details matter. In particular, thanks to the numerical exponents in \eqref{error_est} the error term becomes small for $\tau\leq \tau_*-\tfrac{7}{4}\log \max\{1,Z(X)^2\}$, and in fact we will use that this time shift appears with  prefactor $\tfrac{7}{4}<2$. Moreover, in the proof it will be important that we cut off close to the optimal graphical radius $\sqrt{2|\tau|}$.\\

Applying Theorem \ref{thm:strong_differential_neck_intro} (differential neck theorem) we overcome Challenge \ref{challange1} (compact solutions) and Challenge \ref{challange2} (noncompact solutions), namely we prove the following two theorems.

\begin{theorem}[no exotic ovals]\label{thm:compact_intro}
There exist no exotic ovals in $\mathbb{R}^4$.
\end{theorem}

In case $Z(X)$ is bounded along some sequences of cap points this follows directly. Indeed, near opposing caps the slope $u_1$ would have opposite signs, contradicting \eqref{diff_neck_exp_intr}. Ruling out the potential scenario of exotic ovals with unbounded $Z$ is more involved, and requires further ideas, which will be described later.

 \begin{theorem}[selfsimilarity]\label{thm:noncompact_intro}
Every noncompact strictly convex ancient noncollapsed flow in $\mathbb{R}^4$ is selfsimilarly translating.
\end{theorem}

In case $Z(X)$ is bounded along some sequences of cap points with $x_1\to \pm \infty$, this follows rather directly from our differential neck theorem via standard arguments. Ruling out the potential scenario of unbounded bubble-sheet scale is again more involved, and will be discussed in the next subsection.

\bigskip

\subsection{Outline of the proofs}
In this subsection, we give an overview of our approach. Since the actual proof is rather involved, we will sweep many fine points under the rug, and instead focus on the main ideas and morals. Also, to simplify the notation we will assume $\mathrm{O}_2$-symmetry and pretend that $Z(X)\leq 1$.\\

Let us start by discussing the Merle-Zaag analysis for the slope function $u_1=\partial u^X/\partial y_1$. In light of the commutator identity $[\mathcal{L},\partial_{y_1}]=\tfrac12 \partial_{y_1}$, the evolution of $u_1$ is governed by the operator
\begin{equation}
\mathcal{L}'=\mathcal{L}-\tfrac12.
\end{equation}
The only unstable eigenfunction for $\mathcal{L}'$ is the constant function $1$, and the only $\mathrm{O}_2$-invariant neutral eigenfunctions of $\mathcal{L}'$ are the linear functions $y_1$ and $y_2$. We denote the corresponding projections by $p'_\pm$ and $p_0'$. Fixing a suitable graphical radius function $\rho$ and a suitable cutoff function $\eta$, we work with the localized slope function
\begin{equation}
w(y,\vartheta,\tau)=\eta(|y|/|\rho(\tau)|)u_1(y,\vartheta,\tau).
\end{equation}
We then consider the evolution of
\begin{align}
W_\pm(\tau)=\| p_\pm' w(\cdot,\tau)\|_{\mathcal{H}}^2\qquad\mathrm{and}\qquad W_0(\tau)= \| p_0' w(\cdot,\tau)\|_{\mathcal{H}}^2.
\end{align}
The classical Merle-Zaag lemma \cite{MZ} suggests that for $\tau\to -\infty$ either $W_+$ or $W_0$ dominates. However, there are additional terms from the cutoff, since the function $u_1$, in contrast to $u$, does not satisfy an inverse Poincare inequality. If we cut off slightly before the optimal radius $|y|=\sqrt{2|\tau|}$, then in light of the Gaussian weight $e^{-\frac{1}{4}|y|^2}$ these error terms decay somewhat slower than $e^{\frac{1}{2}\tau}$. More precisely, we will establish the Merle-Zaag alternative that either
\begin{equation}\label{MZ_alt_intro1}
W_-+W_0\leq \frac{C}{|\tau|}\left(W_{+}+ e^{\frac{12}{25}\tau}\omega^2\right),
\end{equation}
or
\begin{equation}\label{MZ_alt_intro2}
W_-+W_++ e^{\frac{12}{25}\tau}\omega^2= o(W_0),
\end{equation}
where the error terms coming from the cutoff are captured by the quantity
 \begin{equation}\label{eq_omega_intro}
 \omega(\tau)=\sup_{\tau'\leq \tau}\sup_{\{\eta\neq 0\}}|u_1(\cdot,\tau')|.
 \end{equation}
We can rule out the scenario \eqref{MZ_alt_intro2} using the asymptotics \eqref{bubble_sheet_mixed1.1}. Hence, \eqref{MZ_alt_intro1} must hold, and taking also into account the ODE for $W_+\sim \| w\|_{\mathcal{H}}^2$ we obtain some decay, specifically\footnote{To understand the exponents, note that $12/25$ is slightly smaller than $1/2$, and $1/5$ is slightly smaller than $6/25$.}
 \begin{equation}\label{decay_w}
\| w\|_{\mathcal{H}}\leq C e^{\frac{1}{5}\tau}.
\end{equation}
This is still very far from our goal $w\sim a e^{\frac{1}{2}\tau}$, but at least some nice decay to get the analysis started.\\

Next, let us emphasize that due to the mixed type behaviour from  \eqref{bubble_sheet_mixed1.1}, our problem is of anisotropic type. In particular, we need to work in regions that are much longer in $y_1$-direction than in $y_2$-direction. To get such elongated regions we take the 2d-barriers $\Sigma_a\subset \mathbb{R}^3$ from \cite{ADS1} and rotate them around ellipses to construct elongated 3d-barriers in $\mathbb{R}^4$. Specifically, given $L\gg 1$, for $a_1\geq a_2\gg 1$ we set
\begin{equation}  
\Gamma_{a_1,a_2}=\left\{\left(\frac{a_1}{a_2} r\cos\vartheta,r\sin\vartheta,y_3,y_4\right)\in \mathbb{R}^4:\vartheta \in [0,2\pi), r \geq L ,  (r-1,y_3,y_4)\in \Sigma_{a_2-1}\right\}.
\end{equation}
The hypersurfaces $\Gamma_{a_1,a_2}$ act as inner barriers for the renormalized mean curvature flow in $\mathbb{R}^4$. These barriers will be used multiple times throughout the paper, and depending on the information available at each stage we can apply them with different eccentricities $a_1/a_2$. For example, in early sections where we only have Merle-Zaag type information in spirit of \eqref{bubble_sheet_mixed1.1}, we can get cylindrical regions of the form
\begin{equation}
\bar{E}_\tau=\left\{ \beta^2 y_1^2 + (2-\beta)^{-1} y_2^2 \leq |\tau| \right\}.
\end{equation}
where $\beta>0$ is a small, but fixed, constant. All points in $\bar{E}_\tau$ are $\eps_0$-cylindrical, and we get the bound
\begin{equation}
\sup_{\bar{E}_\tau} |\bar{A}|^2\leq 5\beta^{-1}.
\end{equation}
On the other hand, after some decay such as \eqref{decay_w} is available, we can iteratively apply barriers with increasing eccentricity alternated with our anisotropic propagation of smallness estimate (to be discussed below), yielding even more elongated cylindrical regions, specifically cylindrical regions of the form\footnote{In particular, this gives a (nonsharp) lower bound for cap distance, namely $d(t)\geq |t|^{1/2}(\log |t|)^{(k+1)/2}$.}
\begin{equation}
\bar{E}_\tau^k=\left\{ |\tau|^{-k} y_1^2 + (2-\beta)^{-1} y_2^2 \leq |\tau| \right\},
\end{equation}
for any finite integer, say for $k=20$, again with the same curvature bound, namely
\begin{equation}\label{curv_est_elong_reg}
\sup_{\bar{E}_\tau^k} |\bar{A}|^2\leq 5\beta^{-1}.
\end{equation}

\bigskip

Another crucial idea is what we call propagation of smallness estimates. To motivate this, note that the Gaussian $L^2$-decay for $u_1$ from \eqref{decay_w} via standard interior estimates can be upgraded to $L^\infty$-decay for $u_1$, however only in the region where the Gaussian weight is comparable to $1$, say in the region $\{|y|\leq L\}$. The purpose of our estimates is to propagate this smallness of $u_1$, or equivalently $\nu_1$, to a much larger region. Specifically, given any $\alpha>0$, e.g. $\alpha=\beta^2$, we consider the anisotropic renormalized distance
\begin{equation}\label{eq_ren_an_dist_intro}
r_\alpha=\left(\alpha y_1^2 + (2-\beta)^{-1} y_2^2\right)^{1/2},
\end{equation}
and prove that for any $p\in [0,q]$, where $q=q(\beta)>1/2$, we have
\begin{equation}\label{eq_aniso_prop_intro}
\sup_{\{ r_\alpha\leq |\tau|^{1/2}\}}e^{(\beta-p)\tau}e^{-\frac{1-\beta}{2}r_\alpha^2} |\nu_1|
 \leq C\sup_{\tau'\leq \tau}\sup_{\{r_\alpha\leq L\}}e^{-p\tau'}\left(|\nu_1|^2+e^{2q\tau'}\right)^{1/2}.
\end{equation}
Namely, the estimate says that if $|\nu_1|$ decays like $e^{p\tau}$ in the core region $\{r_\alpha\leq L\}$, then we can propagate this decay to the larger region $\{r_\alpha\leq |\tau|^{1/2}\}$ up to losing a small exponent. In particular, unwinding the definition of $r_\alpha$, note that in the $y_1$-variable we only have the factor $e^{\frac{1-\beta}{2}\alpha y_1^2}$, which is harmless for $\alpha$ small, so our estimate works particularly well to propagate smallness along the $y_1$-direction. Ping-ponging our anisotropic propagation of smallness estimate (with decreasing $\alpha$) and the anisotropic barrier estimate from above, we can upgrade \eqref{decay_w} to $L^\infty$-decay in a long central strip, specifically we can show that
\begin{equation}\label{nu1_decay}
\sup_{\{ |y_1|\leq |\tau|^{10},\; |y_2|\leq 10 \}} |\nu_1|\leq e^{\frac{1}{10}\tau}.
\end{equation}
We also have another propagation of smallness estimate, which allows us to propagate smallness in $y_2$-direction from the central strip to the entire level sets, specifically we show that
\begin{equation}\label{eq_level_set_prop_intro}
\sup_{\{    |y_1| \leq |\tau|^b\}}|\nu_1| \leq C \sup_{\tau'\leq \tau}\sup_{\{  |y_1| \leq \, 2 |\tau'|^b,\; |y_2|\leq 10\}} |\tau'||\nu_1|+e^{10\tau}.
\end{equation}
To discuss the proofs of both propagation estimates, we recall the basic fact that $\nu_1=\langle \nu , e_1\rangle$ solves
\begin{equation}\label{Jac_eq_intr}
\partial_t \nu_1 = (\Delta +|A|^2)\nu_1,
\end{equation}
where $|A|^2\sim |t|^{-1}$ by \eqref{curv_est_elong_reg}. To prove the anisotropic propagation estimate \eqref{eq_aniso_prop_intro}, we consider the function
\begin{equation}
w=\frac{1}{2}\log \left(\nu_1^2+|t|^{-2q} \right)-\frac{\beta}{4} f_\alpha +\Lambda+p\log|t|,
\end{equation}
where $f_\alpha$ denotes the squared anisotropic distance in the $(x,t)$-variables, and apply the maximum principle to the function $\psi=\eta w/\log|t|$,
where $\eta$ is a carefully chosen weight function. To prove the level set propagation estimate \eqref{eq_level_set_prop_intro}, we first introduce a suitable supersolution of the Jacobi equation \eqref{Jac_eq_intr}, specifically
\begin{align}\label{def_super_hor_intro}
\Psi_b=\exp\left( \frac{x_1^2}{4|t|}-\left(\log|t|\right)^{2b}\right),
\end{align}
and then apply the comparison principle for $\pm \nu_1$ versus $\Psi_b+K\nu_2$ in the region where $y_2\geq 10$.

\bigskip

Yet another issue is that the level set propagation of smallness estimate only works provided that we have control for the diameter of the level sets. Namely, it only works provided that we have the spectral coefficient estimate $\langle  2-y_2^2,u \rangle_{\mathcal{H}}\geq c |\tau|^{-1}$
for some $c>0$. To address this, we reconsider the classical Merle-Zaag analysis, and introduce a new variant that we call Merle-Zaag switch dynamics. It turns out that there is some switch time $\tau_s(X)$, where the spectral coefficients $\vec{\alpha}=(\alpha_1,\alpha_2,\alpha_3)$ with respect to the eigenfunctions $y_1^2-2$, $y_2^2-2$ and $2y_1y_2$ reach some definite size $\kappa_0>0$. Specifically, we have
\begin{equation}
 -\alpha_2 \geq \tfrac12\kappa_0 |\tau|^{-1}\quad \mathrm{for}\;\; \tau \leq \tau_s(X)
 \qquad\mathrm{and} \qquad
 |\vec{\alpha}|\leq \kappa_0 |\tau|^{-1}\quad \mathrm{for}\;\; \tau_s(X) < \tau \leq \tau_\ast
\end{equation}
So for $\tau\leq \tau_s(X)$ we have some definite inwards quadratic bending, and thus can apply the level set propagation of smallness estimate. On the other hand, for $\tau_s(X) < \tau \leq \tau_\ast$ we get a graphical radius much larger than $\sqrt{2|\tau|}$, so instead of \eqref{MZ_alt_intro1} we can get an estimate with much smaller error terms, and we are fine as well. 
\bigskip

Let us now sketch how we can establish our differential neck theorem by combining the above ingredients. We can assume that $\tau\leq \tau_s(X)$, since in the easy case $\tau_s(X) < \tau \leq \tau_\ast$ there is no iteration needed. Using \eqref{nu1_decay} and the level set propagation of smallness, we get decay for the error term from \eqref{eq_omega_intro}, specifically we get
$\omega(\tau)\leq C|\tau|e^{\frac{1}{10}\tau}$.
Feeding this back into \eqref{MZ_alt_intro1}, we can improve \eqref{decay_w} to $\| w\|_{\mathcal{H}}\leq C e^{\frac{3}{10}\tau}$, which in turn yields a better decay for $\omega(\tau)$, etc. Bootstrapping this argument for several rounds we eventually obtain enough decay, so that considering the evolution of $a_+(\tau)=\langle w(\tau),1\rangle_{\mathcal{H}}$ we can obtain a weak version of the differential neck theorem. Specifically, we can show that there is some $a=a(\mathcal{M})\in\mathbb{R}$, such that
\begin{equation}\label{eq_weakDNT_intro}
\big\|w-ae^{\frac{1}{2}\tau}\big\|_{\mathcal{H}}\leq\frac{Ce^{\frac{1}{2}\tau}}{|\tau|^{\frac{1}{2}}} \left(|a|+e^{\frac{1}{10}\tau}\right).
\end{equation}
Next, if we had $a=0$, then $w$ would decay strictly faster than $e^{\frac{1}{2}\tau}$. We then show, that we could propagate this super-smallness, and argue that this would allow us to put super-long barriers, yielding a contradiction with the fact that the cap distance is bounded by $C|t|$. Finally, using the weak differential neck theorem as input, we can further improve our propagation of smallness estimates, and after some more rounds of bootstrap we can eventually improve the $e^{\frac{1}{10}\tau}$ term  from \eqref{eq_weakDNT_intro} to the desired  $e^{\frac{2}{3}\tau}$ term.

\bigskip

Finally, let us outline how we can establish the classification theorem using the differential neck theorem. By scaling we can assume $a(\mathcal{M})=1/\sqrt{2}$. There is a unique point $p_t=p_t^{-}\in M_t$, where $x_1$ is minimal. In case $M_t$ is compact, there is also a unique point $p_t^+$, where $x_1$ is maximal. Denote by $\mathcal{T}(h)$ the time when a cap arrives at level $x_1=h$.
This gives rise to a unique space-time point $P_h=(p^\ast_{\mathcal{T}(h)},\mathcal{T}(h))$, which we call the cap at level $h$. If the bubble-sheet scale $Z(P_{h_i})$ is bounded along some sequences $h_i\to \pm \infty$, then it is quite easy to conclude. Roughly speaking, in the compact case around $P_{h_i}$ with $h_i\to +\infty$ the slope would have the wrong sign, and in the noncompact case via Hamilton's Harnack inequality \cite{Hamilton_Harnack} we can show that $H(P_h)\equiv 1$, since for $h_{i}\to \pm \infty$ we have convergence to translators with differential neck constant $a=1/\sqrt{2}$, hence with speed $1$. Ruling out the potential scenario that $Z(P_{h_i})\to \infty$ is more subtle. In that scenario, the ratio between the bubble-sheet scale and the curvature scale would go to infinity, namely $(ZH)(P_{h_i})\to \infty$. Fortunately, there are also some helpful features, specifically rescaling by $H(P_{h_i})$ around $P_{h_i}$ we would have convergence to the 3d-bowl, and rescaling by $Z(P_{h_i})^{-1}$ we would have convergence to $\mathbb{R}\times$2d-oval. The latter in turn yields $Q(P_h)\leq C Z(P_h)$, where $Q$ denotes the quadratic scale, i.e. the scale from which the inwards quadratic bending in $y_2$-direction starts to holds. To get a grasp on the potential scenario of unbounded bubble-sheet scale we consider the eccentricity at level $h$,
\begin{equation}
\varsigma(h,t)=\frac{\pi}{4}\frac{\mathcal{W}^2}{\mathcal{A}}(h,t),
\end{equation}
defined in terms of the width and the maximal section area at level $h$, namely
\begin{equation}
\mathcal{W}(h,t)= \sup_{M_t\cap \{x_1=h\}}x_2-\inf_{M_t\cap \{x_1=h\}}x_2 \quad \mathrm{and}\quad
\mathcal{A}(h,t)=\sup_{x}\mathcal{H}^2\!\left(K_t\cap \{x_1=h,x_2=x\}\right).
\end{equation}
Using the differential neck theorem, and some further propagation of smallness arguments, we show the translator ratio $\Theta=-\nu_1/H$ is close to $1$ on the whole level set whenever the eccentricity is large. Specifically, we prove that there are $Z_\star<\infty$ and $N=N(\eps)<\infty$, such that if $Z(P_h)\geq Z_\star$, then for every height $h$ and time $t\leq \mathcal{T}(h)$ satisfying $\varsigma(h,t) \geq N+\frac{7}{4}\log Z(P_h)$ we have
\begin{equation}
\sup_{\{x_1=h\}} \left|\Theta(x,t)-1\right|\leq \varepsilon .
\end{equation}
Taking also into account the $h$-evolutions of $\mathcal{A}$ and $\log \mathcal{W}$, this allows us to rule out exotic ovals, and also allows us to show that in the noncompact case we have $\lim_{h\to -\infty} H(P_h)\geq 1$, since otherwise there would not be enough room for the eccentricity to decrease to its value comparable to $1$ near the cap. Finally, arguing similarly as in \cite[Section 6]{BC}, we show that $H(P_h)\leq 1$ for all $h$ sufficiently negative, which in light of the rigidity case of Hamilton's Harnack inequality \cite{Hamilton_Harnack}, allows us to conclude.

\bigskip

\noindent\textbf{Structure of the paper.} In Section \ref{sec_not_and_prel}, we collect some notation and preliminaries. In Section \ref{sec_MZ_switch}, we establish our Merle-Zaag type switch dynamics. In Section \ref{sec_cyl}, we introduce elongated barriers and apply them to find long cylindrical regions. In Section \ref{sec_aniso_propag}, we propagate smallness of $|\nu_1|$ over elongated cylindrical regions. In Section \ref{sec_diff_MZ}, we carry out a Merle-Zaag type analysis for the slope function $u_1$. In Section \ref{sec_cap_and_co}, we combine the decay from the differential Merle-Zaag analysis and the elongated barriers to derive a lower bound for the cap distance and some further consequences. In Section \ref{sec_level_set_propagation}, we show that smallness of $|\nu_1|$ propagates over the entire level sets including the tip regions. In Section \ref{sec_DNT}, we combine our results from the previous sections to establish our differential neck theorem. In Section \ref{sec_bdd}, we complete the classification assuming that the bubble-sheet scale is bounded. Finally, in Section \ref{sec_unbdd}, we rule out the potential scenario of unbounded bubble-sheet scale.

\bigskip

\noindent\textbf{Acknowledgments.}
KC has been supported by the KIAS Individual Grant MG078902, an Asian Young Scientist Fellowship, and the National Research Foundation (NRF) grants RS-2023-00219980 and RS-2024-00345403 funded by the Korea government (MSIT). RH has been supported by the NSERC Discovery grants RGPIN-2016-04331 and RGPIN-2023-04419, and a Sloan Research Fellowship. We are very grateful to Wenkui Du and Or Hershkovits for countless helpful discussions on ancient noncollapsed flows in $\mathbb{R}^4$, and for their fundamental contributions in our collaborations on many closely related papers.

\bigskip

\section{Notation and preliminaries}\label{sec_not_and_prel}

Throughout this paper, $M_t$ denotes an ancient noncollapsed mean curvature flow in $\mathbb{R}^4$. By \cite{Brendle_inscribed,HK_inscribed}, the noncollapsing holds with $\alpha=1$. If $M_t$ splits off a line, then by \cite{BC,ADS2} it must be $\mathbb{R}^3$, $\mathbb{R}\times S^2$, $\mathbb{R}^2\times S^1$, $\mathbb{R}\times$2d-bowl or $\mathbb{R}\times$2d-oval. We can thus assume from now on that $M_t$ does not split off a line. Equivalently, by \cite{HaslhoferKleiner_meanconvex} this means that $M_t$ is strictly convex. If the tangent flow at $-\infty$, c.f. \cite{CM_uniqueness,HaslhoferKleiner_meanconvex}, is a neck, i.e. if
$\lim_{\lambda \rightarrow 0} \lambda M_{\lambda^{-2}t}=\mathbb{R}\times S^{2}(\sqrt{-4t})$,
then by \cite{BC2,ADS2} the flow is a 3d rotationally symmetric bowl or a 3d rotationally symmetric oval. We can thus assume from now on that the tangent flow at $-\infty$ is a bubble-sheet, i.e.
\begin{equation}\label{bubble-sheet_tangent_intro}
\lim_{\lambda \rightarrow 0} \lambda M_{\lambda^{-2}t}=\mathbb{R}^{2}\times S^{1}(\sqrt{-2t}).
\end{equation}
Recall that, in the setting of Theorem \ref{thm_norm_form} (normal form), in case the convergence of the bubble-sheet function $u=u(y,\vartheta,\tau)$ is fast or slow, the classification has already been obtained in \cite{CHH_wing,DH_shape} and \cite{CDDHS}, respectively. We can thus assume that we are in the case of mixed convergence, namely
\begin{equation}\label{DH_asymp}
u= \frac{2-y_2^2}{\sqrt{8}|\tau|}+O(|\tau|^{-1-\gamma}).
\end{equation}
Our goal is to show that $M_t$ belongs to the one-parameter family of 3d oval-bowls constructed in \cite{HIMW}. To show this, by \cite{CHH_translator} it suffices to prove that $M_t$ is noncompact and selfsimilarly translating.\\

Recall from \cite[Section 2.1]{DH_shape} that the evolution over $\Gamma=\mathbb{R}^2\times S^1(\sqrt{2})$ is governed by the Ornstein-Uhlenbeck operator
\begin{equation}\label{prel_OU_res}
\mathcal L=\partial_{y_1}^2+\partial_{y_2}^2-\tfrac{y_1}{2} \partial_{y_1}-\tfrac{y_2}{2} \partial_{y_2}+\tfrac{1}{2} \partial_{\vartheta}^2+1,
\end{equation}
and also recall that we have the associated Gaussian $L^2$-space
\begin{equation}
\mathcal{H}=L^2\left(\Gamma,(4\pi)^{-3/2}e^{-|q|^2/4}dq\right)=\mathcal{H}_+\oplus \mathcal{H}_0\oplus \mathcal{H}_{-},
\end{equation}
where 
\begin{equation}
\mathcal{H}_+=\textrm{span}\{ 1,y_1,y_2,\cos \vartheta, \sin \vartheta\},
\end{equation}
and 
\begin{equation}
\mathcal{H}_0=\textrm{span}\{ y_1^2-2,y_2^2-2,y_1y_2,y_1\cos\vartheta,y_1\sin\vartheta,y_2\cos\vartheta,y_2\sin\vartheta\}.
\end{equation}
As usual, we denote by $p_0$ and $p_\pm$ the projection to the neutral and unstable/stable eigenspace.\\

 Given any space-time point $X=(x,t)\in\mathcal{M}$, we consider the (centered) renormalized flow 
\begin{equation}
\bar{M}_\tau^X=e^{\frac{\tau}{2}}(M_{t-e^{-\tau}}-x),
\end{equation}
To capture how our estimates degenerate as the bubble-sheet scale becomes large, it is useful to define
\begin{equation}
\hat Z(X)=\max\{Z(X),1\}.
\end{equation}
This quantity $\hat{Z}(X)$ will frequently appear in the statements of our results, but on the other hand in the proofs by scaling we can usually reduce to the case $Z(X)\leq 1$, namely we can assume $\hat{Z}(X)=1$.
For $\tau\leq -2\log \hat{Z}(X)$ we can locally express $\bar{M}_\tau^X$ as graph of a bubble-sheet function $u=u^X(y,\vartheta,\tau)$, namely
\begin{equation}
\left\{ q+u(q,\tau)\nu_{\Gamma}(q) \, : \, q\in \bar{\Omega}_\tau \right\} \subset \bar{M}_\tau^X\, ,
\end{equation}
where $\nu_{\Gamma}$ is the outwards unit normal, and where by \cite[Proposition 2.5]{DH_shape} the domains $\bar{\Omega}_\tau\subset \Gamma$ satisfy
\begin{equation}
\big\{  |y|/\hat{Z}(X) \leq |\tau+2\log\hat{Z}(X)|^{4\gamma}  \big\}\subset \bar{\Omega}_\tau\qquad \mathrm{for }\,\, \tau\leq \tau_\ast-2\log \hat{Z}(X).
\end{equation}

\begin{convention}[universal constants]
Throughout this paper, we use the convention that $\tau_\ast >-\infty$ and $C<\infty$ denote universal constants that may be adjusted finitely many times as needed or convenient.
\end{convention}

\noindent To conclude this preliminaries section, let us recall two basic facts that will be used frequently:

First, if $w$ is a bounded solution of a uniformly parabolic equation, then by standard parabolic estimates its $C^k$-norm on any compact set is controlled by its $\mathcal{H}$-norm on a somewhat larger set, for example
\begin{equation}
\| w(\cdot,\tau) \|_{C^{k}(\{ |y| \leq 100 \} )}\leq C \sup_{\tau' \in [\tau-1,\tau]}\| w(\cdot,\tau')1_{ \{ |y| \leq 110 \}} \|_{\mathcal{H}},
\end{equation}
see e.g. \cite[Appendix A]{CHH_wing}. In particular, such estimates hold for $w=u$ and $w=u_{y_1}$.

Second, by Hamilton's Harnack inequality \cite{Hamilton_Harnack}, which applies in our setting thanks to \cite{DS,BLL}, we have $\partial_t H\geq 0$, and hence can control the speed at earlier times. In particular, if  $\partial K_t$ is an ancient noncollapsed flow normalized such that $0\in K_0$, then at any $t<0$ at any $p\in \partial K_t$ we have
\begin{equation}\label{Harnack_curv}
H(p,t)\leq \frac{\langle p,\nu(p,t) \rangle }{|t|}.
\end{equation}

\bigskip

 \section{Merle-Zaag type switch dynamics}\label{sec_MZ_switch}
 
In this section, we carry out a quantitative Merle-Zaag type analysis. We fix a cutoff function $\chi:\mathbb{R}_+\to [0,1]$, such that  $\chi(r)=1$ for $r\leq 1$ and $\chi(r)=0$ for $r\geq 2$, and work with the truncated graph function
\begin{equation}\label{eq_trunc_u}
\hat{u}(y,\vartheta,\tau)=u^X(y,\vartheta,\tau)\chi(|y|/|\tau+2\log \hat{Z}(X)|^{\gamma}).
\end{equation}
Consider
\begin{equation}
U_\pm(\tau)=\left\|p_\pm\hat{u}(\cdot,\tau)\right\|_{\mathcal{H}}^2
\quad\mathrm{and}\quad
U_0(\tau)=\left\|p_0\hat{u}(\cdot,\tau)\right\|_{\mathcal{H}}^2.
\end{equation}
Then, similarly as in \cite[Equation (2.24)]{DH_shape}, for $\tau\leq\tau_\ast-2\log \hat{Z}(X)$ we have the Merle-Zaag type ODEs
\begin{align}
&\tfrac{d}{d\tau} U_+ \geq U_+ - C_0|\tau+2\log \hat{Z}(X)|^{-\gamma} (U_+ + U_0 + U_-), \nonumber\\ 
&\left | \tfrac{d}{d\tau} U_0 \right | \leq C_0|\tau+2\log \hat{Z}(X)|^{-\gamma}  (U_+ + U_0 + U_-), \label{eq:U_system}\\ 
&\tfrac{d}{d\tau} U_- \leq -U_- + C_0|\tau+2\log \hat{Z}(X)|^{-\gamma} (U_+ + U_0 + U_-). \nonumber
\end{align}

\begin{lemma}[{quantitative Merle-Zaag lemma, c.f. \cite{DH_no_rotation}}]\label{prop:domanance_criteria}
There exists  $\tau_+(X)\leq \tau_\ast-2\log \hat{Z}(X)$, such that
\begin{equation} \label{eq:sharp_unstable_mode}
U_+ \geq |\tau+2\log \hat{Z}(X)|^{-\gamma/2} U_0 \qquad \textrm{ if } \,\,  \tau_+(X)<\tau\leq \tau_\ast-2\log \hat{Z}(X),
\end{equation}
and
\begin{equation}
U_+ \leq |\tau+2\log \hat{Z}(X)|^{-\gamma/2} U_0 \qquad \textrm{ if } \,\,  \tau\leq \tau_+(X).
\end{equation}
Moreover, for all $\tau\leq \tau_\ast-2\log \hat{Z}(X)$ we also have the estimate $U_{-} \leq  2C_0|\tau+2\log \hat{Z}(X)|^{-\gamma}(U_0+U_{+})$.
\end{lemma}

\begin{proof} By scaling we can assume $\hat{Z}(X)=1$.
Set $\eps(\tau)=C_0|\tau|^{-\gamma}$. Similarly as in the proof of the Merle-Zaag ODE lemma \cite{MZ}, we consider the quantity $f=U_{-} -2\eps(U_{+}+U_{0})$. By the ODE system \eqref{eq:U_system} at any time where $f\geq 0$ we get
\begin{align}\label{neg_too_large_der}
\tfrac{d}{d\tau}f \leq -U_- +\eps(1+4\eps)(U_{+}+U_0+U_-)-2\eps U_{+}
\leq  -U_-+\eps(1+4\eps)\left(1+\tfrac{1}{2\eps}\right)U_- \leq 0.
\end{align}
Hence, if it was the case that $f(\tau_0)>0$ for some $\tau_0\leq \tau_\ast$, then it would follow that $f\geq f(\tau_0)$ on $(-\infty, \tau_0]$, contradicting the fact that $\lim_{\tau\to -\infty}U_-(\tau)=0$. This shows that for all $\tau\leq \tau_\ast$ we have  
\begin{equation}\label{zero_dom_mz2}
U_{-} \leq  2\eps(U_0+U_{+}).
\end{equation}  
Thus, we can reduce the ODE system \eqref{eq:U_system} to
\begin{align} \label{red_ode_sys}
\tfrac{d}{d\tau} U_+ \geq U_+ - 2\eps (U_+ + U_0 ), \qquad
\left | \tfrac{d}{d\tau} U_0 \right | \leq 2\eps \, (U_+ + U_0).
\end{align}
Now, consider the quantity $g=U_+ -\delta U_0$, where $\delta(\tau)=|\tau|^{-\gamma/2}$, and define
\begin{equation}
\tau_+(X)=\sup \{\tau\leq \tau_\ast : g(\tau)\leq 0\}.
\end{equation}
Here, the supremum is indeed taken over a nonempty set thanks to \eqref{DH_asymp}. Finally,  by the reduced ODE system \eqref{red_ode_sys} at any time when $g=0$ we have
\begin{equation}
\tfrac{d}{d\tau}g \geq U_+-2\eps (1+\delta) (U_++U_0)-\tfrac{\gamma}{2 |\tau|}\delta U_0 \geq \tfrac{1}{2}\delta U_0- 4\eps U_0\geq 0,
\end{equation}
where last two inequalities are justified by adjusting $\tau_\ast$ as needed. This implies the assertion.
\end{proof}

In the following, we consider  the spectral coefficients $\vec{\alpha}=(\alpha_1,\alpha_2,\alpha_3)$ defined by
\begin{equation}\label{def_exp_coeffs}
\alpha_1 = \frac{\langle  y_1^2-2,\hat{u} \rangle_{\mathcal{H}}}{ \| y_1^2-2 \|_{\mathcal{H}}^{2}},\qquad \alpha_2 = \frac{\langle  y_2^2-2,\hat{u} \rangle_{\mathcal{H}}}{ \| y_2^2-2 \|_{\mathcal{H}}^{2}},\qquad
\alpha_3 = \frac{\langle  2y_1y_2,\hat{u} \rangle_{\mathcal{H}}}{ \| 2y_1y_2 \|_{\mathcal{H}}^{2}}.
\end{equation}
Similarly, via the eigenfunctions $y_1\cos\vartheta, y_1\sin\vartheta,y_2\cos\vartheta,y_2\sin\vartheta$ one can define further coefficients $\alpha_4,\ldots,\alpha_7$, but thanks to the almost symmetry from \cite[Proposition 2.7]{DH_shape} they are tiny, specifically
\begin{equation}\label{small_rot_spec_coeff}
\sum_{i=4}^7|\alpha_i|\leq \frac{C}{|\tau+2\log \hat{Z}(X)|^{100}}.
\end{equation}

 \begin{theorem}[spectral coefficients]\label{thm_spec_coeff}
 For all $\kappa>0$ there exists $\tau_\kappa>-\infty$, such that for all $\tau\leq\tau_\kappa -2\log \hat{Z}(X)$ we have
 \begin{equation}\label{spec_coeffs_thm_eq}
 -\frac{1+\kappa}{\sqrt{8}|\tau+2\log \hat{Z}(X)|} \leq \alpha_2\leq \frac{\kappa}{|\tau+2\log \hat{Z}(X)|}\quad\mathrm{and}\quad \max\{ |\alpha_1|,|\alpha_3|\}\leq \frac{\kappa}{|\tau+2\log \hat{Z}(X)|}.
\end{equation}
 \end{theorem}

\begin{proof}By scaling we can assume $\hat{Z}(X)=1$.
Consider the time $\tau_+(X)$ from Lemma \ref{prop:domanance_criteria} (quantitative Merle-Zaag lemma). Then, taking also into account the evolution of $U_+$ from the ODE system \eqref{eq:U_system}, we have
\begin{equation}\label{first_observation}
U_0\leq  |\tau|^{\gamma/2}U_+\leq |\tau|^{\gamma/2}e^{\frac{1}{2}(\tau-\tau_\ast)}\qquad\textrm{for}\;\; \tau_+(X)<\tau\leq\tau_\ast.
\end{equation}
Assume from now on that $\tau\leq\tau_+(X)$. Then, the stable and unstable mode are controlled by
\begin{equation}
U_+^{1/2}+U_{-}^{1/2}\leq C |\vec{\alpha}|/|\tau|^{\gamma/4}.
\end{equation}
Hence, arguing similarly as in \cite[Proposition 3.1]{DH_shape}, we obtain the spectral ODEs
\begin{align}\label{odes0}
   |\tfrac{d}{d\tau}{\alpha}_{1}&+\sqrt{8}(\alpha^2_{1}+\alpha_{3}^2)|\leq C(|\vec{\alpha}|^2/|\tau|^{\gamma/4}+|\tau|^{-10}),\nonumber\\
   |\tfrac{d}{d\tau}{\alpha}_{2}&+\sqrt{8}(\alpha^2_{2}+\alpha_{3}^2)|\leq C(|\vec{\alpha}|^2/|\tau|^{\gamma/4}+|\tau|^{-10}),\\
   |\tfrac{d}{d\tau}{\alpha}_{3}&+\sqrt{8}(\alpha_{1}+\alpha_2)\alpha_{3}|\leq C(|\vec{\alpha}|^2/|\tau|^{\gamma/4}+|\tau|^{-10}).\nonumber
\end{align}
Moreover, similarly as in \cite[Proposition 3.3]{DH_shape}, thanks to convexity we get the estimates
\begin{equation}\label{a_priori_convexity}
\max\{ \alpha_1,\alpha_2\}\leq C(|\vec{\alpha}|/|\tau|^{\gamma/4}+|\tau|^{-10}),\qquad \alpha_3^2-\alpha_1\alpha_2 \leq C(|\vec{\alpha}|^2/|\tau|^{\gamma/4}+|\tau|^{-10}).
\end{equation}
Now, given any small $\kappa>0$, we consider a switch time, which we define by
\begin{equation}\label{def__tau_kappa_s}
\tau_{\kappa,s}(X)=\inf \left\{ \tau\leq \tau_+(X) \; : \; |\vec{\alpha}(\tau')|\leq \frac{\kappa}{|\tau'|}\;\; \mathrm{whenever}\; \tau\leq\tau'<\tau_+(X) \right\}.
\end{equation}
Note that $\tau_{\kappa,s}(X)$ is a finite number in light of \eqref{DH_asymp}. We can assume from now on that $\tau\leq\tau_{\kappa,s}(X)$, since otherwise there is nothing to prove. 
 To proceed, we need to control the trace $S=\alpha_1+\alpha_2$.
 
\begin{claim}[trace estimate]\label{claim_trace} There exists some $\tau_\kappa>-\infty$, such that for all $\tau\leq  \min\{\tau_{\kappa,s}(X),\tau_{\kappa}\}$ we have
\begin{equation}
-\frac{2}{|\tau|}\leq S(\tau)\leq -\frac{\kappa/3}{|\tau|}.
\end{equation}
\end{claim}

\begin{proof}
First, by definition of $\tau_{\kappa,s}(X)$ we have $|\vec{\alpha}(\tau_{\kappa,s}(X))|\geq \kappa/|\tau_{\kappa,s}(X)|$, and hence \eqref{a_priori_convexity} gives $S(\tau_{\kappa,s}(X))\leq -(\kappa/3)/|\tau_{\kappa,s}(X)|$. Now, thanks to \eqref{odes0} and \eqref{a_priori_convexity}
as long as $S(\tau)\leq -(\kappa/4)/|\tau|$ we have $-\tfrac{d}{d\tau}S \leq 5S^2$. Integrating this yields
\begin{equation}
-\frac{1}{S(\tau)}\leq\frac{|\tau_{\kappa,s}(X)|}{\kappa/3}+5|\tau_{\kappa,s}(X)-\tau| \leq\frac{|\tau|}{\kappa/3}
\end{equation}
for all $\tau\leq\tau_{\kappa,s}(X)$, as desired. Similarly, observing that $-\tfrac{d}{d\tau}S \geq S^2$, for all $\tau\leq\tau_{\kappa,s}(X)$ we get
\begin{equation}
-\frac{1}{S(\tau)}\geq -\frac{1}{S(\tau_{\kappa,s}(X))}+|\tau_{\kappa,s}(X)-\tau|.
\end{equation}
If $\tau_{\kappa,s}(X)<\tau_+(X)$, then we have $|\vec{\alpha}(\tau_{\kappa,s}(X))|= \kappa/|\tau_{\kappa,s}(X)|$, hence $-S(\tau_{\kappa,s}(X))\leq 2\kappa/|\tau_{\kappa,s}(X)|$, which yields $-1/S(\tau)\geq |\tau|/2$ for all $\tau\leq\tau_{\kappa,s}(X)$. On the other hand, if $\tau_{\kappa,s}(X)=\tau_+(X)$, then by \eqref{first_observation} it must be the case that $|\tau_{\kappa,s}(X)|\leq C_\kappa |\tau_\ast|$ for some $C_\kappa<\infty$, and thus we obtain $-1/S(\tau)\geq|\tau|/2$ for $\tau\leq 2C_{\kappa}\tau_\ast$. Choosing $\tau_\kappa=2C_\kappa \tau_\ast$, this concludes the proof of the claim.
\end{proof}

Continuing the proof of the theorem, using \eqref{odes0} and \eqref{a_priori_convexity} and Claim \ref{claim_trace} (trace estimate) we infer that for $\tau\leq  \min\{\tau_{\kappa,s}(X),\tau_{\kappa}\}$ the trace $S=\alpha_1+\alpha_2$ and the determinant $D=\alpha_1\alpha_2-\alpha_3^2$ evolve by
\begin{align}
   |\tfrac{d}{d\tau}S+\sqrt{8}(S^2-2D)|\leq C|\tau|^{-2-\gamma/4},\qquad
   |\tfrac{d}{d\tau}D+\sqrt{8}SD|\leq C|\tau|^{-3-\gamma/4},
\end{align}
and satisfy the a priori estimates
\begin{equation}
\kappa/3\leq -|\tau|S \leq 2,\qquad -C|\tau|^{-2-\gamma/4}\leq D \leq \tfrac{1}{4}S^2.
\end{equation}
In fact, similarly as in \cite{DH_shape,DH_no_rotation}, it is useful to consider the normalized dimensionless variables $x=-\sqrt{2}|\tau|S$ and $y=8|\tau|^2D$.
Viewing them as a function of $\sigma=\log(-\tau)$, they evolve by
\begin{align}\label{odes2}
   \tfrac{d}{d\sigma}x=x+y-2x^2+O(e^{-\frac{\gamma}{4}\sigma}),\qquad
   \tfrac{d}{d\sigma}y=2y-2xy+O(e^{-\frac{\gamma}{4}\sigma}),
\end{align}
and stay confined in the region
\begin{equation}\label{eq_XY_confined}
\sqrt{2}\kappa/3 \leq x \leq 2\sqrt{2},\qquad -Ce^{-\frac{\gamma}{4}\sigma} \leq y \leq x^2.
\end{equation}

The vector field $V(x,y)=(x+y-2x^2,2y-2xy)$ has zeros at $(0,0)$, $(1/2,0)$ and $(1,1)$, but the one at $(0,0)$ does not concern us thanks to \eqref{eq_XY_confined}.
Note that equation \eqref{DH_asymp} tells us that
\begin{equation}\label{eq_limit_XY}
\lim_{\sigma\to \infty}(x(\sigma),y(\sigma))=(1/2,0).
\end{equation}
Hence, observing also that $(1,1)$ is a stable fixed point, it follows that the trajectory can never come too close to $(1,1)$. Namely, possibly after decreasing $\tau_\kappa$, there is an $\eps>0$, such that
\begin{equation}\label{eq_not_too_close}
|x(\sigma)-1|+|y(\sigma)-1|\geq 5\eps \qquad \forall \sigma\geq \log(-\min\{ \tau_{\kappa,s}(X),\tau_{\kappa}\}).
\end{equation}

\bigskip

Let us consider the case $\tau_{\kappa,s}(X) <\tau_\kappa$. Then, by \eqref{first_observation} and the definition of $\tau_{\kappa,s}(X)$ we have $|\vec{\alpha}(\tau_{\kappa,s}(X))|=\kappa/ |\tau_{\kappa,s}(X)|$. In particular, setting $\sigma_{\kappa,s}(X)=\log(-\tau_{\kappa,s}(X))$, we have $x(\sigma_{\kappa,s}(X))\leq 4\kappa$.
We now claim that
\begin{equation}\label{eq_improved_trace}
x(\sigma)\leq 1-\eps\qquad \forall\sigma\geq \sigma_{\kappa,s}(X).
\end{equation}
Indeed, if this failed then there would be $\sigma_0 > \sigma_{\kappa,s}(X)$, such that $x(\sigma_0)=1-\eps$ and $\tfrac{d}{d\sigma}|_{\sigma=\sigma_0}x\geq 0$. Together with \eqref{odes2}, this would tell us that $y(\sigma_0)\geq 1-3\eps$, which, remembering also $y\leq x^2$ from  \eqref{eq_XY_confined}, would contradict \eqref{eq_not_too_close}.
Having established \eqref{eq_improved_trace}, and choosing $\delta=\min\{\kappa/4,\gamma/16\}$ for later use, we claim that
\begin{equation}\label{eq_improved_det}
y(\sigma)\leq \delta \qquad \forall\sigma\geq \sigma_{\kappa,s}(X).
\end{equation}
Indeed, if this failed then $y$ would increase at rate $\tfrac{d}{d\sigma} y\geq \eps\delta$, contradicting \eqref{eq_limit_XY}. 
Having established \eqref{eq_improved_det}, we can now feed it back into the ODE for $x$. This yields the further improved trace estimate
\begin{equation}\label{eq_further_improved_trace}
x(\sigma)\leq \frac{1}{2}+\delta\qquad \forall\sigma\geq \sigma_{\kappa,s}(X),
\end{equation}
which, remembering also \eqref{a_priori_convexity}, establishes the desired bound
\begin{equation}\label{eq_alpha2_des_bd}
-\kappa\leq -|\tau|\alpha_2\leq (1+\kappa)/\sqrt{8}\qquad \forall\tau\leq\tau_{\kappa,s}(X).
\end{equation}
Next, to rule out rotations, similarly as in \cite{DH_no_rotation}, we consider $z=|\tau|\alpha_3$. It satisfies
\begin{equation}\label{ODE_for_Z}
\tfrac{d}{d\sigma}z=z-2xz+O(e^{-\frac{\gamma}{4}\sigma}).
\end{equation}
In particular, $|z|$ is almost monotone as long as $x\leq  1/2$. Considering the first time $\sigma_0$ when $x$ reaches $1/2-\delta/2$, we thus infer that
\begin{equation}
|z|(\sigma)\leq |z|(\sigma_0)+Ce^{-\frac{\gamma}{4}\sigma} \qquad \forall \sigma \in [\sigma_{\kappa,s}(X),\sigma_0].
\end{equation}
Moreover, observing that $|x-1/2|\leq \delta$ for all $\sigma\geq \sigma_0$, we can write \eqref{ODE_for_Z} in the form $\tfrac{d}{d\sigma} z=pz+q$, where $|p|\leq 2\delta$ and $|q|\leq Ce^{-\frac{\gamma}{4}\sigma}$ for all $\sigma\geq \sigma_0$. Taking also into account that $\lim_{\sigma\to \infty}e^{\frac{\gamma}{2}\sigma}z(\sigma)=0$ thanks to  \eqref{DH_asymp}, and remembering that $\delta\leq \gamma/16$, the variation of constants formula yields
\begin{equation}
|z|(\sigma)\leq \limsup_{\sigma'\to \infty}\left| e^{-\int_\sigma^{\sigma'} p} \int_{\sigma}^{\sigma'} q(\hat{\sigma}) e^{\int_{\hat{\sigma}}^{\sigma'} p} \, d\hat{\sigma} \right|
\leq C\limsup_{\sigma'\to \infty}\left| \int_{\sigma}^{\sigma'} e^{-\frac{\gamma}{4}\hat{\sigma}}e^{2\delta(\hat{\sigma}-\sigma)} \, d\hat{\sigma} \right|\leq Ce^{-\frac{\gamma}{8}\sigma}
\end{equation}
for all $\sigma\geq\sigma_0$. We have thus shown that
\begin{equation}\label{alpha3_decay}
|\alpha_3|\leq C|\tau|^{-1-\gamma/8} \qquad \forall\tau\leq\tau_{\kappa,s}(X).
\end{equation}
Finally, feeding this back into the ODE for $\alpha_1$, we infer that
\begin{equation}\label{alpha1_decay}
|\alpha_1|\leq C|\tau|^{-1-\gamma/16} \qquad \forall\tau\leq\tau_{\kappa,s}(X).
\end{equation}
To conclude, equations \eqref{eq_alpha2_des_bd}, \eqref{alpha3_decay} and \eqref{alpha1_decay} show that the estimate \eqref{spec_coeffs_thm_eq} holds for $\tau\leq\tau_{\kappa,s}(X)$. Moreover, we recall that thanks to \eqref{first_observation} and \eqref{def__tau_kappa_s}, the estimate \eqref{spec_coeffs_thm_eq} holds for $\tau_{\kappa,s}(X)\leq \tau \leq \tau_\kappa$ as well. Summarizing, we have thus shown, if $\tau_{\kappa,s}(X)< \tau_\kappa$, then the desired estimate \eqref{spec_coeffs_thm_eq} holds for all $\tau \leq \tau_\kappa$.

Lastly, if $\tau_{\kappa,s}(X)\geq \tau_\kappa$, then after waiting a controlled amount of time the trajectory will come close to $(1/2,0)$, and hence the above argument applies. Adjusting $\tau_\kappa$ once again, this finishes the proof of the theorem.
\end{proof}

\begin{convention}[switch parameter] We now fix a small constant $\kappa_0\in (0,10^{-5})$, and adjust $\tau_\ast$ to arrange that $\tau_\ast<\tau_{\kappa_0}$. From now on we work with the switch time defined using the parameter $\kappa_0$, namely
 \begin{equation}\label{def_switch_time}
\tau_s(X)=\inf \left\{ \tau\leq \tau_\ast-2\log \hat{Z}(X) \; : \; |\vec{\alpha}(\tau')|\leq \frac{\kappa_0}{|\tau'+2\log \hat{Z}(X)|}\;\; \mathrm{if}\; \tau'\in[\tau,\tau_\ast-2\log \hat{Z}(X)) \right\}.
\end{equation}
\end{convention}
Inspecting the above proof we obtain the following corollary:

\begin{corollary}[switch time]\label{cor_switch_time} With the switch time $\tau_s(X)$ defined in \eqref{def_switch_time} we have
\begin{equation}\label{sw_eq_1}
|\vec{\alpha}(\tau)|\leq \frac{\kappa_0}{|\tau+2\log \hat{Z}(X)|}\qquad \mathrm{ for }\;\; \tau_s(X) < \tau\leq \tau_\ast-2\log \hat{Z}(X),
\end{equation}
and
\begin{equation}\label{sw_eq_2}
  -\alpha_2\geq \frac{\kappa_0/2}{|\tau+2\log \hat{Z}(X)|}
  \quad{and}\quad
   |\alpha_1|+|\alpha_3|\leq \frac{C}{|\tau+2\log \hat{Z}(X)|^{1+\gamma/16}}
  \qquad \mathrm{ for }\;\; \tau \leq \tau_s(X).
 \end{equation}
Furthermore, if $\tau_s(X)<\tau_\ast-2\log \hat{Z}(X)$ then we also have the estimate
\begin{equation}
|\vec{\alpha}(\tau)|\leq \frac{100\kappa_0}{|\tau+2\log \hat{Z}(X)|} \qquad \mathrm{ for }\;\; \tau\in \left[\tau_s(X)+98(\tau_s(X)+2\log\hat{Z}(X)),\tau_s(X)\right].
\end{equation}
\end{corollary}
 
 \begin{proof}By scaling we can assume $\hat{Z}(X)=1$. The estimate \eqref{sw_eq_1} follows directly from the definitions. Next, the estimate \eqref{sw_eq_2} follows from \eqref{alpha3_decay} and \eqref{alpha1_decay} and the observation that $\tfrac{d}{d\sigma} x\geq 0$ whenever $x\leq 1/4$. Finally, assuming $\tau_s(X)<\tau_\ast$ we want to show that
$|\tau||\vec{\alpha}(\tau)|$ increases by a factor at most $100$, when $\tau$ decreases from $\tau_s(X)$ to $99\tau_s(X)$. To this end, again thanks to \eqref{alpha3_decay} and \eqref{alpha1_decay}, it suffices to show that $x$ increases at most by a factor $99$, when $\sigma$ increases from $\log(-\tau_s(X))$ to $\log(-\tau_s(X))+\log(99)$, and in fact this follows immediately from $\tfrac{d}{d\sigma} x\leq x$. This concludes the proof of the corollary.
  \end{proof}
 
\bigskip
 
 \section{Elongated cylindrical regions}\label{sec_cyl}

In this section, we introduce elongated barriers and apply them to find long cylindrical regions.\\

Recall from \cite[Lemma 4.2]{ADS1} that there is a 1-parameter family of convex shrinkers in $\mathbb{R}^3$,
\begin{equation}
\Sigma_{a}=\left\{ (x,y,z)\in\mathbb{R}^3 \, : \, v_a(x)=(y^2+z^2)^{1/2},\; 0\leq x\leq a\right\},
\end{equation}
where the concave function $v_a:[0,a]\to \mathbb{R}_+$ is the unique solution of 
\begin{equation}
\frac{v_{yy}}{1+v_y^2}-\frac{y}{2}v_y+\frac{v}{2}-\frac{1}{v}=0 \quad \mathrm{with}\;\; v(a)=0\;\; \mathrm{and}\;\; \lim_{x\nearrow a}v'(x)=-\infty.
\end{equation}
Now, to construct barriers in $\mathbb{R}^4$, we rotate these shrinkers along ellipses. Specifically, given $L\gg 1$, for $a_1\geq a_2\gg 1$ we define\footnote{In the special case $a_1=a_2$ this gives the barriers from our joint work with Hershkovits \cite{CHH_wing}, specifically $\Gamma_{a,a}=\Gamma^{\mathrm{CHH}}_{a-1}$.}
\begin{equation}  
\Gamma_{a_1,a_2}=\left\{\left(\frac{a_1}{a_2} r\cos\vartheta,r\sin\vartheta,y_3,y_4\right)\in \mathbb{R}^4:\vartheta \in [0,2\pi), r \geq L ,  (r-1,y_3,y_4)\in \Sigma_{a_2-1}\right\}.
\end{equation}
To capture the meaning of the parameters $a_1$ and $a_2$ observe that $\max_{\Gamma_{a_1,a_2}}y_1=a_1$ and $\max_{\Gamma_{a_1,a_2}}y_2=a_2$. 

\begin{proposition}[elongated barriers]\label{prop_elong_barr}
The hypersurfaces $\Gamma_{a_1,a_2}$ are inner barriers. Namely, if $\{\partial K_\tau\}_{\tau\in[\tau_1,\tau_2]}$ evolves by renormalized mean curvature flow, and $K_\tau$ contains the region bounded by $\Gamma_{a_1,a_2}\cup \{(a_2/a_1)^2y_1^2+y_2^2=L^2\}$ for all $\tau\in [\tau_1,\tau_2]$, and $\partial K_\tau\cap \Gamma_{a_1,a_2}  =\emptyset$ for all $\tau < \tau_2$, then $\partial K_{\tau_2}\cap \Gamma_{a_1,a_2}   \subset \partial \Gamma_{a_1,a_2}$.
\end{proposition}

\begin{proof}
By the maximum principle we have to show that $\vec{H}+\frac{y^\perp}{2}$ points outwards everywhere on $\Gamma_{a_1,a_2}$. In the special case $a_1=a_2$ this has already been shown in \cite[Corollary 3.4]{CHH_wing}, and we will now reduce to this special case by scaling by $a_1/a_2$ in the $y_1$-direction.  To this end, note that if we write $\Gamma_{a_2,a_2}\cap \{ y_1 >0\}$ as graph of a positive concave function $f(y_2,y_3,y_4)$, then the outwards unit normal is $(1, -Df)/(1+|Df|^2)^{1/2}$, so the fact that $\vec{H}+\frac{y^\perp}{2}$ points outwards is equivalent to the inequality
\begin{equation}\label{eq_barr_gra}
\Delta f -\frac{f_{ij}f_if_j}{1+|Df|^2}+\frac{f-y_i f_i}{2}\geq 0.
\end{equation}
Now, multiplying this by $\lambda=a_1/a_2\geq 1$ and using $f_{ij}\leq 0$, we infer that $f^\lambda=\lambda f$ also satisfies \eqref{eq_barr_gra}.
This shows that $\vec{H}+\frac{y^\perp}{2}$ points outwards everywhere on $\Gamma_{a_1,a_2}$, and thus proves the proposition.
\end{proof}

\begin{lemma}[ellipsoidal domains]\label{prop_ell_dom}
For all $\alpha>0$ and $\delta>0$ there exists $\tau_{\alpha,\delta}>-\infty$, such that for $\tau\leq \tau_{\alpha,\delta}-2\log \hat{Z}(X)$ the profile function $u(\cdot,\tau)=u^X(\cdot,\tau)$ is well-defined in the ellipsoidal domain
\begin{equation}
\left\{ \alpha^2y_1^2+y_2^2 \leq \frac{2}{1+\delta}|\tau+2\log \hat{Z}(X)| \right\},
\end{equation}
with the estimate
\begin{equation}
u(y,\vartheta,\tau)\geq \sqrt{2-\frac{(1+\delta)\left(\alpha^2 y_1^2+y_2^2 \right)}{|\tau+2\log \hat{Z}(X)|}}-\sqrt{2}.
\end{equation}
\end{lemma}

\begin{proof}By scaling we can assume $\hat{Z}(X)=1$.
Set $L=10\delta^{-1}$. Thanks to Theorem \ref{thm_spec_coeff} (spectral coefficients), taking also into account \eqref{small_rot_spec_coeff}, Lemma \ref{prop:domanance_criteria} (quantitative Merle-Zaag lemma) and standard parabolic estimates, for
 $\tau$ sufficiently negative (depending on $\alpha$ and $\delta$) we have
\begin{equation}\label{lower_cpt_reg}
u(y,\vartheta,\tau)\geq - \frac{y_2^2-2}{\sqrt{8}|\tau|}-\frac{\delta}{|\tau|} \qquad \mathrm{in}\; \{\alpha^2 y_1^2+y_2^2 \leq L^2 \}.
\end{equation}
On the other hand, by \cite[Lemma 4.4]{ADS1} for $a_2$ sufficiently large the shrinker function $v_{a_2-1}$ satisfies
\begin{equation}
v_{a_2-1}(L-1)-\sqrt{2}\leq - \frac{(L-1)^2-5}{\sqrt{2}(a_2-1)^2}.
\end{equation}
Hence, given $\tau_0$ sufficiently negative, considering the barrier $\Gamma_{a_1,a_2}$ with the parameters
\begin{equation}
a_2=1+ \sqrt{\frac{2|\tau_0|}{1+\delta}} \quad\mathrm{and}\quad a_1=\frac{a_2}{\alpha},
\end{equation}
the boundary condition $u(\cdot,\tau)> v_{a_2-1}(L-1)-\sqrt{2}$ on $\{(a_2/a_1)^2 y_1^2+y_2^2=L^2\}$ holds for all $\tau\leq \tau_0$. Moreover, since $\bar{M}_\tau$ converges to $\mathbb{R}^2\times S^1(\sqrt{2})$ for $\tau\to-\infty$, it lies outside of $\Gamma_{a_1,a_2}$ for very negative $\tau$. Thus, Proposition \ref{prop_elong_barr} (elongated barriers) yields that $\bar{M}_\tau$ lies outside of $\Gamma_{a_1,a_2}$ for all $\tau\leq \tau_0$, in particular
\begin{equation}
u(y,\vartheta,\tau_0)\geq v_{a_2-1}(r-1)-\sqrt{2},\qquad \mathrm{when}\;\; r=\sqrt{\alpha^2y_1^2+y_2^2}\geq L.
\end{equation}
Together with the lower bound for $v_{a_2-1}$ from \cite[Theorem 8.2]{ADS1}, this implies the assertion.
\end{proof}

\begin{convention}[$\beta$ and $L$]
We now fix parameters $\beta\in(0,10^{-3})$ and $L\geq 1/\beta^{2}$ for the rest of the paper.
\end{convention}

\begin{theorem}[elongated cylindrical regions]\label{thm_elong_cyl}
For all $\tau\leq \tau_\ast-2\log\hat{Z}(X)$ every point in the ellipsoidal region $\bar{E}_\tau=\{ \beta^2 y_1^2+(2-\beta)^{-1}y_2^2\leq |\tau+2\log \hat{Z}(X)|\}$ is $\eps_0$-cylindrical with the curvature and symmetry estimates
\begin{equation}\label{curv_bd_57}
\sup_{\bar{E}_\tau}|\bar{A}|^2\leq \frac{5}{\beta} \qquad\mathrm{and} \qquad
\sup_{\bar{E}_\tau} |u_\vartheta| \leq \frac{1}{|\tau+2\log\hat{Z}(X)|^{10}}.
\end{equation}
Furthermore, if $\tau_s(X)<\tau_\ast-2\log\hat{Z}(X)$, then for all $\tau\in [2\tau_s(X)+2\log \hat{Z}(X),\tau_\ast-2\log \hat{Z}(X)]$ every point in the region $\{ y_1^2+y_2^2\leq 10|\tau+2\log \hat{Z}(X)|\}$ is $\eps_0$-cylindrical with the estimates $|\bar{A}|\leq 1$ and $|u_\vartheta|\leq |\tau+2\log \hat{Z}(X)|^{-10}$.
\end{theorem}

\begin{proof}By scaling we can assume $\hat{Z}(X)=1$.
Thanks to Lemma \ref{prop_ell_dom} (ellipsoidal domains) for $\tau$ sufficiently negative we have
\begin{equation}\label{rad_low_bound}
\left(\sqrt{2}+u(y,\vartheta,\tau)\right)^2 \geq 2 -\frac{(1+\tfrac{\beta}{8})(\beta^2y_1^2+y_2^2)}{|\tau|}\geq \frac{\beta}{4} \quad\mathrm{in}\;\; \left\{\beta^2y_1^2+y_2^2\leq (2-\tfrac{\beta}{2})|\tau|\right\}.
\end{equation}
Moreover, in the somewhat smaller ellipsoidal region $\bar{E}_\tau=\left\{ \beta^2 y_1^2+(2-\beta)^{-1}y_2^2\leq |\tau| \right\}$ we have
\begin{equation}\label{eq_ellipse_decay}
\lim_{\tau\to -\infty}\sup_{\bar{E}_\tau}\left(|\partial_{y_1} u| + |\partial_{y_2} u|\right)=0,
\end{equation}
since otherwise by convexity we would obtain a contradiction with \eqref{rad_low_bound}. Now, if some $(p_i,t_i)$ in the corresponding unrescaled ellipsoidal region were not $\eps_0$-cylindrical for $t_i\to -\infty$, then we could consider the sequence of flows $M^i_t$ obtained from $M_t$ by shifting $(p_i,t_i)$ to the origin and parabolically rescaling by $H(p_i,t_i)$. However, by the general theory of noncollapsed flows \cite{HaslhoferKleiner_meanconvex}, the flows $M^i_t$ would converge to an ancient noncollapsed flow that thanks to \eqref{eq_ellipse_decay} splits off 2 lines, namely to a round shrinking bubble-sheet, giving the desired contradiction. Having established $\eps_0$-cylindricality, the curvature bound follows using again \eqref{rad_low_bound}, and after replacing $\beta$ by $\beta/2$, the symmetry estimate follows from Zhu's bubble-sheet improvement theorem \cite{Zhu} (see also \cite[Proof of Proposition 2.7]{DH_shape} for a similar argument with more details).

Suppose now $\tau_s(X)<\tau_\ast$. Then, thanks to Corollary \ref{cor_switch_time} (switch time), taking also into account Lemma \ref{prop:domanance_criteria} (quantitative Merle-Zaag lemma), Theorem \ref{thm_spec_coeff} (spectral coefficients) and standard parabolic estimates, we have
\begin{equation}\label{lower_cpt_reg}
\sup_{\{y_1^2+y_2^2 \leq L^2\}}|u(y,\vartheta,\tau)|\leq \begin{cases}
\frac{1000 \kappa_0 L^2}{|\tau|}& \mathrm{if}\; 99\tau_s(X)<\tau\leq\tau_\ast\\
\frac{(1+2\kappa_0)L^2}{\sqrt{8}|\tau|}& \mathrm{if}\; \tau\leq 99\tau_s(X).
\end{cases}
\end{equation}
Hence, given any $\tau_0\in [2\tau_s(X),\tau_\ast]$, remembering $\kappa_0< 10^{-5}$ we see that the barrier $\Gamma_{a,a}$ with
\begin{equation}
a=1+ \sqrt{40|\tau_0|}
\end{equation}
satisfies the boundary condition $u(\cdot,\tau)> v_{a-1}(L-1)-\sqrt{2}$ on $\{y_1^2+y_2^2=L^2\}$ for all $\tau\leq\tau_0$. Thus, applying Proposition \ref{prop_elong_barr} (elongated barriers) in the special case $a_1=a_2=a$, and remembering \cite[Theorem 8.2]{ADS1}, we infer that $u(\cdot,\tau_0)$ is well-defined in $\{y_1^2+y_2^2\leq 40|\tau_0|\}$ with the estimate
\begin{equation}
u(y,\vartheta,\tau_0)\geq \sqrt{2-\frac{y_1^2+y_2^2}{20|\tau_0|}}-\sqrt{2}.
\end{equation}
Arguing as above, we conclude that for all $\tau\in [2\tau_s(X),\tau_\ast]$ every point in the region $\{ y_1^2+y_2^2\leq 10|\tau|\}$ is $\eps_0$-cylindrical with the estimates $|\bar{A}|\leq 1$ and $|u_\vartheta|\leq |\tau|^{-10}$. This finishes the proof of the theorem.
\end{proof}

Finally, for later use let us record that as a corollary of the proof we obtain:

\begin{corollary}[very elongated cylindrical regions]\label{cor_elongated_barr}
Suppose that $|\nu_1| \leq \alpha^{1/2}  \beta/  |\tau+2\log \hat{Z}(X)|$ holds in the region $\{\alpha (2-\beta)  y_1^2+y_2^2\leq L^2\}$ for all $\tau\leq\bar{\tau}\leq \tau_\ast-2\log \hat{Z}(X)$. Then, for all $\tau\leq\bar{\tau}$ every point in the region $\{\alpha y_1^2+(2-\beta)^{-1}y_2^2\leq |\tau+2\log \hat{Z}(X)|\}$ is $\eps_0$-cylindrical with the estimates $|\bar{A}|^2 \leq 5 \beta^{-1}$ and $|u_\vartheta|\leq|\tau+2\log \hat{Z}(X)|^{-10}$.
\end{corollary}

\begin{proof}By scaling we can assume $\hat{Z}(X)=1$.
Given $\tau_0\leq\tau_\ast$, considering the barrier $\Gamma_{a_1,a_2}$ with the parameters
\begin{equation}
a_2=1+ \sqrt{(2-\tfrac{\beta}{4})|\tau_0|} \quad\mathrm{and}\quad a_1=\frac{a_2}{\alpha^{1/2}(2-\beta)^{1/2}},
\end{equation}
for all $\tau\leq \tau_0$ we have
\begin{equation}
u(y,\vartheta,\tau)\geq - \frac{y_2^2-2}{\sqrt{8}|\tau|}-\frac{\beta L}{|\tau|} \qquad \mathrm{in}\; \{(a_2/a_1)^2 y_1^2+y_2^2 \leq L^2 \}.
\end{equation}
Arguing as above, this implies the assertion.
\end{proof}

\bigskip

\section{Anisotropic propagation of smallness}\label{sec_aniso_propag}

In this section, we show that smallness of $|\nu_1|$ propagates quite well along the $x_1$-axis. For ease of notation, after recentering in space-time we can assume that $X=0$. To begin with, given any $\alpha\in (0,\beta^2]$, we consider the anisotropic (squared renormalized) distance
\begin{align}\label{def_aniso_dist}
 f_\alpha=\alpha f_1+(2-\beta)^{-1} f_2, \qquad\mathrm{where}\; f_i=x_i^2/|t|.
\end{align}
Throughout this section, we will often impose the following a priori assumption\footnote{In particular, note that $\mathrm{AP}_{\beta^2}$ always holds thanks to Theorem \ref{thm_elong_cyl} (elongated cylindrical regions).}
\begin{equation}\label{ap_ass}
\mathrm{AP}_{\alpha}\! : \textrm{For all } t\leq -\hat{Z}(X)^2e^{-\tau_\ast} \textrm{ we have } |A|^2\leq 5\beta^{-1} |t|^{-1} \textrm{ in  } \{ f_\alpha(\cdot,t)\leq \log|t| -2\log \hat{Z}(X)\}.
\end{equation}

\begin{lemma}[anisotropic distance]\label{prop:para.evol}
For all $t<0$ we have
\begin{align}
 (\partial_t- \Delta) f_\alpha\geq \frac{f_\alpha-2}{|t|}\qquad\mathrm{and}\qquad  |\nabla f_\alpha|^2\leq \frac{(16-4\beta)f_\alpha}{ (8-7\beta)|t| }.
\end{align}
\end{lemma}
 
\begin{proof}
Using the mean curvature flow equation $x_t=\Delta x$, we see that
\begin{equation}
(\partial_t-\Delta)f_i= \frac{f_i-2|\nabla x_i|^2 }{|t|}\geq \frac{f_i-2}{|t|}.
\end{equation}
This implies the first assertion. Next, by the AM-GM inequality we can estimate
\begin{align}
|\nabla f_\alpha|^2 &\leq \left(1+\frac{8-7\beta}{(1+\beta)\beta}\right)\alpha^2|\nabla f_1|^2+\left(1+\frac{(1+\beta)\beta}{8-7\beta}\right)\frac{|\nabla f_2|^2}{(2-\beta)^2}.
\end{align}
Combining this with $|\nabla f_i|^2\leq 4|t|^{-1}f_i$ yields
\begin{equation}
|t||\nabla f_\alpha|^2\leq \frac{16-4\beta}{8-7\beta}\left(\frac{(8-7\beta)(2-\beta)\alpha^2}{(1+\beta)\beta}f_1+f_\alpha-\alpha f_1\right).
\end{equation} 
Since $\alpha/\beta\leq\beta\leq 1/16$, this implies the second assertion, and thus completes the proof.
\end{proof} 

Recall that under the mean curvature flow the components of the unit normal evolve by
\begin{equation}\label{eq_unit_normals_evol}
(\partial_t-\Delta) \nu_i =|A|^2 \nu_i.
\end{equation} 
Instead of working with $|\nu_1|$, we work with its regularization
\begin{equation}
\nu_1^\eps = (|\nu_1|-\eps \zeta)_+,\qquad \mathrm{where}\; \zeta=e^{\frac{1}{4}f_\alpha}.
\end{equation}

\begin{proposition}[regularized slope]\label{prop:reg_normal}
Under the a priori assumption $\mathrm{AP}_{\alpha}$ from \eqref{ap_ass}, for  $t\leq -\hat{Z}(X)^2e^{-\tau_\ast}$ we have
\begin{equation}\label{evol_eq_reg}
(\partial_t-\Delta-|A|^2) \nu_1^\eps \leq 0 \qquad \mathrm{in }\; \{ \nu_1^\eps >0 \}\cap \{ 60\beta^{-1}\leq f_\alpha\leq \log |t|-2\log \hat{Z}(X)\}.
\end{equation}
\end{proposition}

\begin{proof}
Using Lemma \ref{prop:para.evol} (anisotropic distance) and $\beta\leq 1/16$ we see that
\begin{equation}
\frac{(\partial_t-\Delta)\zeta}{\zeta}=\frac{1}{4}\left( (\partial_t-\Delta) f_\alpha-\frac{1}{4}|\nabla f_\alpha|^2 \right)\geq \frac{f_\alpha}{10|t|}-\frac{2}{|t|}.
\end{equation}
Hence, by our a priori assumption $\mathrm{AP}_{\alpha}$, for $t\leq -\hat{Z}(X)^2e^{-\tau_\ast}$ we obtain
\begin{equation}
(\partial_t-\Delta-|A|^2) \zeta \geq 0 \qquad \mathrm{in }\; \{ 60\beta^{-1}\leq f_\alpha\leq \log |t|-2\log \hat{Z}(X)\}.
\end{equation}
Together with \eqref{eq_unit_normals_evol} this implies the assertion.
\end{proof}

Now, with the goal of proving an anisotropic exponential growth estimate for $\nu_1^\eps$, we consider the function
\begin{equation}
w^\eps=\frac{1}{2}\log \left((\nu_1^\eps)^2+\hat{Z}(X)^{4q}|t|^{-2q} \right)-\frac{\beta}{4} f_\alpha +\Lambda+p\left(\log|t|-2\log\hat{Z}(X)\right),
\end{equation}
where $q=(2+\beta)/4$, and where the parameters $\Lambda\in \mathbb{R}$ and $p\in [0,q]$ will be chosen below. We will apply the maximum principle to the function
\begin{equation}
\psi^\eps=\frac{\eta w^\eps}{\log|t|-2\log \hat{Z}(X)},
\end{equation}
where $\eta$ denotes a suitable weight function, specifically
\begin{equation}
\eta(r)=\left(1+\left(\frac{q}{\beta}-1\right)r\right)^{-1}, \qquad  \text{with} \quad r=\frac{f_\alpha}{\log|t|-2\log \hat{Z}(X)}.
\end{equation}
  
\begin{lemma}[maximum principle for anisotropic growth]\label{lemma_max_aniso}
Under the a priori assumption $\mathrm{AP}_{\alpha}$ from \eqref{ap_ass}, if the function $\psi^\eps$ attains a positive local maximum over a backwards parabolic ball at some $(x_0,t_0)\in \{\nu_1^\eps>0 \}$ with $100\beta^{-3}<f_\alpha(x_0,t_0)<\log|t_0|-2\log \hat{Z}(X)$ and $t_0\leq - \hat{Z}(X)^2 e^{-\tau_\ast}$, then  $\psi^\eps(x_0,t_0) > \beta$.
\end{lemma}

\begin{proof}By scaling we can assume $\hat{Z}(X)=1$. Moreover, throughout this proof we drop the superscript $\eps$, and abbreviate
\begin{equation}
\varphi= \left( (|\nu_1|-\eps \zeta)_+^2+|t|^{-2q} \right)^{1/2},\qquad f=f_\alpha, \qquad h=\log|t|.
\end{equation}
Suppose towards a contradiction there is some $(x,t)\in \{100\beta^{-3}< f< h\}\cap \{\nu_1^\eps>0 \}$ with $t\leq - e^{-\tau_\ast}$, at which
 \begin{equation}\label{max_princ_ass}
 \nabla\psi=0\;\;  \mathrm{and}\;\;  (\partial_t-\Delta )\psi\geq 0,\;\;  \mathrm{but}\;\; 0< \psi\leq \beta.
 \end{equation}

We begin by deriving
\begin{equation}
(\partial_t-\Delta) \psi=\frac{\eta w}{|t|h^2}+\frac{\eta (\partial_t-\Delta) w-2\nabla \eta \nabla w+w(\partial_t-\Delta)\eta}{h}.
\end{equation}
Using Lemma \ref{prop:para.evol} (anisotropic distance), $ f \geq 100\beta^{-3}$, and $\eta''\geq 0$, we compute
\begin{equation}
(\partial_t-\Delta) \eta=\eta' \left(\frac{f}{|t|h^2}+\frac{f_t-\Delta f}{h}\right)-\frac{\eta''|\nabla f|^2}{h^2}\leq -\frac{(q/\beta-1)\eta^2(f-2)}{|t|h}.
\end{equation}
Hence, using \eqref{max_princ_ass} and $\eta^{-1} \leq q/\beta$, at the space-time point under consideration we obtain
\begin{equation}\label{anis_max_1st_ineq}
0\leq  q+|t|(\partial_t-\Delta)w +2(q/\beta-1)\eta |t|h^{-1}  \nabla f\nabla w  -(q/\beta-1)\psi(f-2).
\end{equation}

\bigskip

Next, using Proposition \ref{prop:reg_normal} (regularized slope) we compute
\begin{align}
(\partial_t-\Delta) w &=\frac{(\partial_t-\Delta)\varphi^2}{2\varphi^2}+2|\nabla \log \varphi|^2-\frac{\beta}{4} (\partial_t-\Delta) f-\frac{p}{|t|}\nonumber \\
& \leq  \frac{q|t|^{-2q-1}+|A|^2(\nu_1^\eps)^2-|\nabla \nu_1^\eps|^2}{\varphi^{2}}+2|\nabla \log \varphi|^2-\frac{\beta}{4} \frac{f-2}{|t|}.
\end{align}
Thanks to our a priori assumption $\mathrm{AP}_{\alpha}$ from \eqref{ap_ass} we can estimate
\begin{equation}
q|t|^{-2q-1}+|A|^2(\nu_1^\eps)^2\leq 5\beta^{-1} |t|^{-1}\varphi^2.
\end{equation}
Together with $|\nabla \nu_1^\eps |\geq |\nabla \varphi|$ and $\nabla w=\nabla \log \varphi-(\beta/4)\nabla f$ this yields
\begin{equation}
(\partial_t-\Delta) w \leq |\nabla w+(\beta/4) \nabla f|^2-(\beta/4) |t|^{-1}f+(5\beta^{-1}+1)|t|^{-1}.
\end{equation}
Setting $\Psi=(\frac{q}{\beta}-1)\psi+\frac{1}{4}\beta$, and remembering also that $q=(2+\beta)/4$, we thus infer that
\begin{equation}
0\leq    |t||\nabla w+(\beta/4)\nabla   f|^2 +2(q/\beta-1)\eta|t|h^{-1}\nabla f\nabla  w  -\Psi f +6\beta^{-1}.
\end{equation}

\bigskip

Finally, we employ $\nabla \psi=0$, namely $\nabla w=(q/\beta-1)\psi\nabla f$, which yields
\begin{align}
|\nabla w+(\beta/4)\nabla f|^2= \Psi^2|\nabla f|^2\;\;  \mathrm{and}\;\;  |\nabla f\nabla w|\leq \Psi|\nabla f|^2.
\end{align}
We may assume $\tau_\ast\leq -16\beta^{-2}$, hence $h^{-1}\leq \beta^2/16$. So, using $\beta/4\leq \Psi \leq (1-\beta)/2$ we conclude that
\begin{equation}
0\leq  (\Psi+\tfrac{1}{16} \beta) |t||\nabla f|^2-f+6\beta^{-1}\Psi^{-1}\leq \tfrac{1}{2}(1-\tfrac{7}{8} \beta) |t||\nabla f|^2-f+24\beta^{-2}.
\end{equation}
However, by  the gradient bound from Lemma \ref{prop:para.evol} (anisotropic distance) we have
\begin{equation}
\tfrac{1}{2}(1-\tfrac{7}{8} \beta) |t||\nabla f|^2\leq (1-\tfrac{1}{4}\beta )f.
\end{equation}
Since $f\geq 100\beta^{-3}$, this gives the desired contradiction, and thus finishes the proof of the lemma.
\end{proof} 

\begin{remark}[maximum principle for modified quantity]\label{rem_mod_quant} For later use let us observe that the above proof goes through when we drop the $|t|^{-2q}$-term, namely when we replace $w^\eps$ by
\begin{equation}
\tilde{w}^\eps=\frac{1}{2}\log \left((\nu_1^\eps)^2 \right)-\frac{\beta}{4} f_\alpha +\Lambda+p\left(\log|t|-2\log\hat{Z}(X)\right).
\end{equation}
\end{remark}

\medskip

Now, to state the main theorem of this section, let us abbreviate
\begin{equation}\label{eq_ren_an_dist}
r_\alpha=\left(\alpha y_1^2 + (2-\beta)^{-1} y_2^2\right)^{1/2}.
\end{equation}

\begin{theorem}[anisotropic propagation of smallness]\label{thm:propagation}
Under the a priori assumption $\mathrm{AP}_{\alpha}$ from \eqref{ap_ass}, for all $p\in [0,q]$ and all $\tau \leq \tau_\ast-2\log \hat{Z}(X)$ we have
\begin{equation}
\sup_{\{ r_\alpha\leq |\tau+2\log\hat{Z}(X)|^{1/2}\}}\hat{Z}(X)^{2\beta}e^{(\beta-p)\tau}e^{-\frac{1-\beta}{2}r_\alpha^2} |\nu_1|
 \leq \mu(\tau)\sup_{\tau'\leq \tau}\sup_{\{r_\alpha\leq L\}}e^{-p\tau'}\left(|\nu_1|^2+\hat{Z}(X)^{4q}e^{2q\tau'}\right)^{1/2} ,
\end{equation}
where
\begin{align}
\mu(\tau)=\max\Big\{\hat{Z}(X)^{2\beta}e^{\beta \tau}, \sup_{\tau'\leq \tau}\sup_{\{r_\alpha \leq |\tau'+2\log\hat{Z}(X)|^{1/2}\}} \big(|\nu_1|^2+\hat{Z}(X)^{4q}e^{2q\tau'} \big)^{1/2}\Big\}.
\end{align}
\end{theorem}

\begin{proof}We will work in the unrenormalized variables. By scaling we can assume $\hat{Z}(X)=1$. Moreover, we can assume that
\begin{equation}
\limsup_{t\to -\infty} \sup_{\{f_\alpha(\cdot,t)\leq L^2\}} |t|^p |\nu_1| < \infty,
\end{equation}
since otherwise there is nothing to prove. As before, given any $\eps>0$, we consider the functions
\begin{equation}
w^\eps=\frac{1}{2}\log \left((|\nu_1|-\eps\zeta)_+^2 + |t|^{-2q} \right)-\frac{\beta}{4} f_\alpha +\Lambda+p\log|t|,\qquad \psi^\eps=\frac{\eta w^\eps}{\log|t|}.
\end{equation}
Given $\bar{t}\leq - e^{-\tau_\ast}$, we now choose $\Lambda=-\log(\bar{\mu}K)$, where
\begin{align}
K=\sup_{t\leq \bar t}\sup_{\{f_\alpha(\cdot,t)\leq L^2\}}|t|^{p}\left(|\nu_1|^2+|t|^{-2q}\right)^{1/2}.
\end{align}
and
\begin{align}
\bar \mu =\max\left\{|\bar t\,|^{-\beta},\sup_{t\leq \bar t}\sup_{\{f_\alpha(\cdot,t) \leq \log |t|\}} (|\nu_1|^2+|t|^{-2q})^{1/2}\right\}.\end{align}
Note that as a consequence of the definitions we have
\begin{equation}\label{ineq_in_part}
-\log K\leq  (q-p)\log|\bar{t}|,
\end{equation}
and
\begin{equation}\label{ineq_cons_def}
\sup_{t\leq \bar{t}}\sup_{\{f_\alpha(\cdot,t)\leq L^2\}} w^\eps \leq - \log \bar \mu \leq \beta \log |\bar t \,|.
\end{equation}
This in turn implies the global bound $w^\eps\leq  q \log|\bar{t}|$, and consequently, remembering that $\eta=\beta/q$ on $\{f_\alpha=\log|t|\}$, we infer that $\psi^\eps \leq \beta$ on $\{f_\alpha=\log|t|\}$ for all $t\leq \bar{t}$. Moreover, using Lemma \ref{prop_ell_dom} (ellipsoidal domains) we see that $|\nu_1|\leq\eps\leq \eps\zeta$ for $t$ very negative (depending on $\eps$), so using again \eqref{ineq_in_part} we get that  $\psi^\eps <\beta$ in the whole region $\{L^2\leq f_\alpha \leq \log|t|\}$ for $t$ very negative. Furthermore, thanks to our choice of $\Lambda$, specifically by \eqref{ineq_in_part} and the latter inequality in \eqref{ineq_cons_def}, we have the inclusion $\{\psi^\eps\geq \beta\}\subset\{\nu_1^\eps>0\}$. We can thus apply Lemma \ref{lemma_max_aniso} (maximum principle for anisotropic growth) to conclude that $\psi^\eps \leq \beta$, namely
\begin{equation}
\frac{1}{2}\log \left((|\nu_1|-\eps\zeta)_+^2+|t|^{-2q} \right) \leq \frac{1-\beta}{2}f_\alpha+(\beta-p)\log|t|+\log( \bar \mu K).
\end{equation}
Since $\eps>0$ was arbitrary, this implies the assertion.
\end{proof}

\bigskip

\section{Differential Merle-Zaag analysis}\label{sec_diff_MZ}

In this section, we carry out a Merle-Zaag type analysis for the slope $u_1=\partial_{y_1}u^X$.
To begin with, recall from \cite[Appendix A]{DH_shape} that the bubble-sheet function $u=u^X$ evolves by
\begin{equation}\label{evol_u}
\partial_\tau u= \mathcal{L} u -\tfrac{1}{2}\varrho u^2 -\left(\varrho u  -\tfrac{1}{ 2}\varrho^2u^2\right)u_{\vartheta\vartheta}-\frac{P }{ Q}, 
\end{equation}
where $\mathcal{L}$ is the Ornstein-Uhlenbeck operator from \eqref{prel_OU}, and where
\begin{align}
P=u_iu_ju_{ij}+\varrho^4 u_\vartheta^2u_{\vartheta\vartheta} +2\varrho^2   u_iu_\vartheta u_{i\vartheta}  +\varrho^3 u_\vartheta^2\qquad\mathrm{and}\qquad Q=1+|\nabla u|^2 -u_\vartheta^2(\varrho u -\tfrac{1}{2}\varrho^2u^2),
\end{align}
with
\begin{align}
 \varrho=(\sqrt{2}+u)^{-1}\qquad\mathrm{and}\qquad |\nabla u|^2=u_1^2+u_2^2+\tfrac{1}{2}u_\vartheta^2.
\end{align}
Differentiating \eqref{evol_u} with respect to $y_1$, and taking also into account the identity $2\varrho^2=1-2\varrho u +\varrho^2u^2$, we get
\begin{align}
\partial_\tau u_1=\mathcal{L}'u_1+\mathcal{E}u_1+\mathcal{F}u_1,
\end{align} 
where
\begin{align}
\mathcal{L}'=\mathcal{L}-\tfrac12, \quad\qquad \mathcal{E}=-\mathcal{E}^{ij}\partial_{i}\partial_j-\mathcal{E}^{i\vartheta}\partial_i\partial_{\vartheta}-\mathcal{E}^{\vartheta\vartheta}\partial_{\vartheta}\partial_{\vartheta} -\mathcal{E}^{i}\partial_i-\mathcal{E}^{\vartheta}\partial_\vartheta,
\end{align}
and
\begin{equation}\label{eq:error.F}
\mathcal{F}=-\varrho u+\frac{1}{2}\varrho^2u^2-2\varrho^3 u_{\vartheta\vartheta}+\frac{4\varrho^5u_\vartheta^2u_{\vartheta\vartheta}+4\varrho^3u_iu_\vartheta u_{i\vartheta}+3\varrho^4u_\vartheta^2}{Q}-\frac{2\varrho^3 P u_\vartheta^2}{Q^2}.
\end{equation}
Here, $\mathcal{F}$ is just a multiplication operator, while $\mathcal{E}$ is a differential operator with coefficients given by
\begin{align}
&\mathcal{E}^{ij}=\frac{u_iu_j}{Q}, && \mathcal{E}^{i\vartheta}=\frac{2\varrho^2u_iu_\vartheta}{Q}, && \mathcal{E}^i=\frac{2u_ju_{ij}+2\varrho^2 u_\vartheta u_{i\vartheta}}{Q}-\frac{2Pu_i}{Q^2},
\end{align}
and
\begin{align}
& \mathcal{E}^{\vartheta\vartheta}=\varrho u-\frac{1}{2}\varrho^2u^2+\frac{\varrho^4u_\vartheta^2}{Q}, &&\mathcal{E}^\vartheta=\frac{2\varrho^4u_\vartheta u_{\vartheta\vartheta}+2\varrho^2u_iu_{i\vartheta}+2\varrho^3u_\vartheta}{Q}-\frac{\varrho^2  Pu_\vartheta }{Q^2}.
\end{align}
Since $\mathcal{L}'=\mathcal{L}-\tfrac12$, the only unstable eigenfunction is now the constant function $1$, and the neutral eigenfunctions are now $y_1$, $y_2$, $\cos\vartheta$, and $\sin\vartheta$. We denote  the corresponding projections by $p_\pm'$ and $p_0'$.\\

We fix a monotone function $\lambda$, such that $\lambda(\zeta)=\sqrt{8}$ for $\zeta\leq 3/2$ and $\lambda(\zeta)=1$ for $\zeta\geq 2$, and set
\begin{equation}\label{eq_rho_switch}
\rho(\tau)=
\left\{\begin{array}{ll}
|\tau+2\log\hat{Z}(X)|^{1/2} & \mathrm{if}\; \tau_s(X)=\tau_\ast-2\log\hat{Z}(X)\\
\lambda\left(\frac{\tau+2\log \hat Z(X)}{\tau_s(X)+2\log \hat Z(X)}\right)|\tau+2\log\hat{Z}(X)|^{1/2} & \mathrm{if}\; \tau_s(X)<\tau_\ast-2\log\hat{Z}(X).
\end{array}\right.
\end{equation}
Fixing another monotone function $\eta$ that cuts off slightly before $z=\sqrt{2}$, specifically such that $\eta(z)=1$ for $ z\leq \frac{7}{5}$ and $\eta(z)=0$ for $z\geq \frac{141}{100}$, we then work with the function
\begin{equation}
w(y,\vartheta,\tau)=\eta\left(\frac{|y|}{\rho(\tau)}\right)u_1(y,\vartheta,\tau).
\end{equation}
This localized slope function evolves by
\begin{equation}
\partial_\tau w=\mathcal{L}'w+ \mathcal{E}w+ \mathcal{F}w+\mathcal{G} ,
\end{equation}
where
\begin{align}
\mathcal{G} = \eta\mathcal{L}'u_1-\mathcal{L}'w+\eta \mathcal{E}u_1-\mathcal{E}w+u_1\eta_\tau.
\end{align}

\bigskip

Throughout this section we will frequently use that by standard cylindrical estimates (c.f. \cite[Lemma 4.16]{CDDHS}) we have
\begin{align}\label{std_cyl_est}
\sup_{\{\eta>0\}}\left(|\nabla u|+ |\nabla^2u|+|\nabla^3u|\right) \leq C|\tau+2\log\hat{Z}(X)|^{-\frac{1}{2}},
\end{align}
and that by Theorem \ref{thm_elong_cyl} (elongated cylindrical regions) and Corollary \ref{cor_elongated_barr} (very elongated cylindrical regions) and standard interior estimates we have
\begin{align}\label{rot_est}
\sup_{\{\eta>0\}}\left(|u_\vartheta|+ |\nabla u_\vartheta|+|\nabla^2u_\vartheta|\right) \leq C|\tau+2\log\hat{Z}(X)|^{-10}.
\end{align}
 
\bigskip 

\subsection{Differential Merle-Zaag alternative}
In this subsection, we consider the evolution of the functions 
\begin{align}
W_\pm(\tau)=\| p_\pm' w(\cdot,\tau)\|_{\mathcal{H}}^2,\qquad W_0(\tau)= \| p_0' w(\cdot,\tau)\|_{\mathcal{H}}^2.
\end{align}
For convenience, let us also abbreviate
\begin{equation}
W=W_++W_0+W_-.
\end{equation}

\begin{proposition}[unstable slope mode]\label{lem:evol.W_+} For $\tau \leq \tau_\ast-2\log\hat{Z}(X)$ we have
\begin{equation}
|\tfrac{d}{d\tau} W_+-W_+|\leq C|\tau+2\log\hat{Z}(X)|^{-1} W+2|\langle \mathcal{G},p_+'w\rangle_{\mathcal{H}}|.
\end{equation}
\end{proposition}

\begin{proof}
By scaling we can assume $\hat{Z}(X)=1$.
To begin with, observe that
\begin{equation}
\tfrac{d}{d\tau} W_+=2\langle w_\tau,p_+'w\rangle_{\mathcal{H}}=W_++2\langle \mathcal{E}w+ \mathcal{F}w,p_+'w\rangle_{\mathcal{H}}+2\langle \mathcal{G},p_+'w\rangle_{\mathcal{H}}.
\end{equation}
Since $p_+'$ projects to a finite-dimensional space, it is thus enough to show that
\begin{equation}
\left| \int (\mathcal{E}w +  \mathcal{F} w)  e^{-\frac{|y|^2}{4}} \right| \leq C|\tau|^{-1}W^{\frac{1}{2}}.
\end{equation}
To this end, note that for any smooth function $h$ we have the integration by parts formula
\begin{equation}\label{eq:cal_F.int.part}
\int \mathcal{E}(w) h e^{-\frac{|y|^2}{4}}=\int w\, \mathcal{E}^\ast \big(he^{-\frac{|y|^2}{4}}\big),
\end{equation}
where
\begin{align}
\mathcal{E}^\ast (v)=& -\partial_i\partial_j(\mathcal{E}^{ij}v)-\partial_i\partial_\vartheta(\mathcal{E}^{i\vartheta}v)-\partial_\vartheta\partial_\vartheta(\mathcal{E}^{\vartheta\vartheta}v)
+\partial_i(\mathcal{E}^{i}v)+\partial_\vartheta(\mathcal{E}^{\vartheta}v).
\end{align}
Remembering the derivative estimates \eqref{std_cyl_est} and \eqref{rot_est}, observe that
\begin{equation}\label{eq_F1_estimate_pointwise}
\Big| \, {\mathcal{E}}^{\ast}\!\big(e^{-\frac{|y|^2}{4}}\big)\Big| \leq C|\tau|^{-1}(1+|y|^2)e^{-\frac{|y|^2}{4}}.
\end{equation}
Together with the Cauchy-Schwarz inequality this yields
\begin{equation}\label{eq_F1_estimate_L2} 
\left|\int w\, \mathcal{E}^\ast \! \big(e^{-\frac{|y|^2}{4}}\big)\right| \leq C|\tau|^{-1}W^{\frac{1}{2}}.
\end{equation}
Regarding the remaining term, recall first that thanks to Lemma \ref{prop_ell_dom} (ellipsoidal domains) on the support of $\eta$ we have $\sqrt{2}+u\geq \beta/10$.
Hence, remembering \eqref{std_cyl_est} and \eqref{rot_est}, we see that
\begin{equation}
|\langle w,\mathcal{F}\rangle_{\mathcal{H}}| \leq C W^{\frac{1}{2}}\big(\|1_{\{\eta> 0\}} |u|\|_{\mathcal{H}}{+|\tau|^{-10}}\big).
\end{equation}
Also recall that by the inverse Poincare inequality from \cite[Proposition 4.4]{CHH_wing} we have
\begin{equation}\label{inv_poinc}
\|1_{\{\eta> 0\}} (1+|y|) |u|\|_{\mathcal{H}}+\|1_{\{\eta> 0\}} |\nabla u|\|_{\mathcal{H}}\leq C\|\hat u\|_\mathcal{H},
\end{equation}
where $\hat{u}$ is the truncated graph function from \eqref{eq_trunc_u}. Since $\|\hat u\|_\mathcal{H}\leq C|\tau|^{-1}$ this shows that
\begin{equation}\label{eq_F2_estimate_L2} 
|\langle w,\mathcal{F}\rangle_{\mathcal{H}}| \leq C|\tau|^{-1}W^{\frac{1}{2}},
\end{equation}
and thus completes the proof.
\end{proof}

\begin{proposition}[neutral slope mode]\label{lem:evol.W_0}
For $\tau \leq \tau_\ast-2\log\hat{Z}(X) $ we have
\begin{align}
| \tfrac{d}{d\tau} W_0|\leq  C|\tau+2\log\hat{Z}(X)|^{-1}W+2|\langle \mathcal{G},p_0'w\rangle_{\mathcal{H}}|.
\end{align}
\end{proposition}

\begin{proof}
By scaling we can assume $\hat{Z}(X)=1$.
To begin with, observe that
\begin{equation}
\tfrac{d}{d\tau} W_0=2\langle  \mathcal{E}w+\mathcal{F}w,p_0'w\rangle_{\mathcal{H}}+2\langle \mathcal{G},p_0'w\rangle_{\mathcal{H}}.
\end{equation}
Hence, it is enough to show that for $\varphi_0=y_1,y_2,\cos\vartheta$ or $\sin\vartheta$ we have
\begin{equation}
\left| \int (\mathcal{E}w+\mathcal{F}w) \varphi_0  e^{-\frac{|y|^2}{4}} \right| \leq C|\tau|^{-1}W^{\frac{1}{2}}.
\end{equation}
Arguing as in the proof of Lemma \ref{lem:evol.W_+} (unstable slope mode), we can estimate
\begin{equation}
\left|\int w\,  \mathcal{E}^\ast\! \big( \varphi_0 e^{-\frac{|y|^2}{4}}\big)\right| \leq C|\tau|^{-1} \int |w|(1+|y|^3)e^{-\frac{|y|^2}{4}} \leq C|\tau|^{-1}W^{\frac{1}{2}},
\end{equation}
and
 \begin{equation}
|\langle w,\mathcal{F} {\varphi_0}\rangle_{\mathcal{H}}| \leq CW^{\frac{1}{2}}\big(\|1_{\{\eta> 0\}}{(1+|y|)} |u|\|_{\mathcal{H}}+|\tau|^{-10}\big)\leq C|\tau|^{-1}W^{\frac{1}{2}}.
\end{equation}
This implies the assertion.
\end{proof}

\begin{proposition}[stable slope mode]\label{lem:evol.W_-}
For $\tau \leq \tau_\ast -2\log\hat{Z}(X)$ we have
\begin{align}
\tfrac{d}{d\tau} W_-+\tfrac{3}{4} W_-\leq  C|\tau+2\log\hat{Z}(X)|^{-1}W+2|\langle \mathcal{G},p_-'w\rangle_{\mathcal{H}}|.
\end{align}
\end{proposition}

\begin{proof}
By scaling we can assume $\hat{Z}(X)=1$.
To begin with, observe that
\begin{align}
\tfrac{d}{d\tau} W_-=W_--2\|\nabla p_-'w\|_{\mathcal{H}}^2+2\langle \mathcal{E}w+ \mathcal{F}w,p_{-}'w\rangle_{\mathcal{H}} +2\langle \mathcal{G},p_-'w\rangle_{\mathcal{H}}\quad\mathrm{and}\quad \| \nabla p_{-}' w\|_{\mathcal{H}}^2\geq W_{-}.
\end{align}
It thus suffices to show that
\begin{equation}\label{stab_grad_to_show}
\langle  \mathcal{E}w+\mathcal{F}w,p_{-}'w\rangle_{\mathcal{H}}\leq {\tfrac{1}{8}\|\nabla p_-'w\|_{\mathcal{H}}^2}+ C|\tau|^{-1}W. 
\end{equation}
In fact, since the unstable and neutral mode have already been controlled in the above proofs, it actually suffices to show that
\begin{equation}\label{stab_grad_to_show_in_fact}
\langle  \mathcal{E}w+ \mathcal{F}w,w\rangle_{\mathcal{H}}\leq {\tfrac{1}{10} \|\nabla p_-'w\|_{\mathcal{H}}^2}+C|\tau|^{-1}(\|\nabla w\|^2_{\mathcal{H}}+ W). 
\end{equation}
To this ends, observe that using $|\mathcal{E}^i|+|\mathcal{E}^\vartheta|\leq C|\tau|^{-1}$ we can estimate
\begin{align}
 \left|\int  \mathcal{E}^iw_i w e^{-\frac{|y|^2}{4}}\right|+ \left|\int  \mathcal{E}^\vartheta w_\vartheta w e^{-\frac{|y|^2}{4}}\right|\leq C|\tau|^{-1}(\|\nabla w\|_{\mathcal{H}}^2+W). 
\end{align}
Moreover, by using $\mathcal{E}^{ij}y_iy_j\geq 0$ and integration by parts, we can observe that
\begin{equation}
\int w^2\left[\mathcal{E}^{ij}e^{-\frac{|y|^2}{4}}\right]_{ij}\geq \int \left[-\tfrac{1}{2} w^2\mathcal{E}^{ij}\delta_{ij}-(2w^2(\mathcal{E}^{ij})_i)_j  +w^2(\mathcal{E}^{ij})_{ij}\right]e^{-\frac{|y|^2}{4}}.
\end{equation}
Thus, we can estimate
\begin{align}
-\int \mathcal{E}^{ij}w_{ij}we^{-\frac{|y|^2}{4}}=\int \mathcal{E}^{ij}w_iw_je^{-\frac{|y|^2}{4}}-\frac{1}{2}\int w^2\left[\mathcal{E}^{ij}e^{-\frac{|y|^2}{4}}\right]_{ij}\leq C|\tau|^{-1}(\|\nabla w\|_{\mathcal{H}}^2+W).
\end{align}
Furthermore, using integration by parts again we have
\begin{equation}
-\int (\mathcal{E}^{i\vartheta}w_{i\vartheta}+\mathcal{E}^{\vartheta\vartheta}w_{\vartheta\vartheta}) we^{-\frac{|y|^2}{4}} \leq C|\tau|^{-1}(\|\nabla w\|_{\mathcal{H}}^2+W)+ \int (\varrho u -\tfrac{1}{2}\varrho^2u^2)w_\vartheta^2e^{-\frac{|y|^2}{4}}.
\end{equation}
To analyze the latter summand let us abbreviate
\begin{equation}
\upsilon=1_{\{\eta>0\}}(\varrho u-\tfrac{1}{2}\varrho^2u^2).
\end{equation}
Recalling that the kernel of  $\mathcal{L}'$ is spanned by $y_1$, $y_2$, $\cos\vartheta$ and $\sin\vartheta$, and observing that $\partial_\vartheta$ commutes with $p_\ast'$, we see that
\begin{align}
&w_\vartheta-p'_-w_\vartheta=p'_0w_\vartheta=b_1(\tau)\cos\vartheta+b_2(\tau)\sin\vartheta,  &&|b_1|^2+|b_2|^2\leq W_0.
\end{align}
Thus, using the inverse Poincare inequality from \eqref{inv_poinc} and remembering that  $\|\hat u\|_{\mathcal{H}}\leq C|\tau|^{-1}$ we get
\begin{equation}
\int \upsilon(p'_0w_\vartheta)^2 e^{-\frac{|y|^2}{4}}\leq CW_0\int 1_{\{\eta>0\}}|u| e^{-\frac{|y|^2}{4}}\leq C|\tau|^{-1}W.
\end{equation}
Similarly, using also integration by part and the almost symmetry estimate from \eqref{rot_est} we infer that
\begin{equation}
\int  \upsilon(p'_-w_\vartheta)(p'_0w_\vartheta)  e^{-\frac{|y|^2}{4}}\leq CW_0^{\frac{1}{2}} \int 1_{\{\eta>0\}}  (|u | +|u_\vartheta|) |p'_-w|     e^{-\frac{|y|^2}{4}}\leq C|\tau|^{-1}W.
\end{equation}
Moreover, thanks to Lemma \ref{claim_upper_p} (upper bound for profile function) below, we can estimate
\begin{equation}
\int  \upsilon(p'_-w_\vartheta)^2  e^{-\frac{|y|^2}{4}}\leq C|\tau|^{-1/2}\|\nabla p_{-}' w \|_{\mathcal{H}}^2.
\end{equation}
Finally, to control $\langle \mathcal{F}w,w\rangle_{\mathcal{H}}$, note that as a consequence of Lemma \ref{prop_ell_dom} (ellipsoidal domains) and convexity we have the estimate
\begin{equation}
|u|\leq \frac{C(1+|y|)^2}{|\tau|},
\end{equation}
and also recall that for the Gaussian norm we have the weighted Poincare inequality
\begin{equation}
\| (1+|y|)f \|_{\mathcal{H}}\leq C(\| f \|_{\mathcal{H}}+\| \nabla f \|_{\mathcal{H}}).
\end{equation}
This yields
\begin{equation}
|\langle  \mathcal{F} w,w\rangle_{\mathcal{H}}| \leq C |\tau|^{-1} (W+ \| \nabla  w\|_{\mathcal{H}}^2),
\end{equation}
and thus finishes the proof of the proposition.
\end{proof}
 
In the above proof we used the following lemma: 
 
\begin{lemma}[upper bound for profile function]\label{claim_upper_p}
We have $\sup_{\{\eta >0 \}} u \leq C|\tau+2\log\hat{Z}(X)|^{-1/2}$.
\end{lemma}

\begin{proof}
Denoting by $\bar{K}_\tau$ the convex domain enclosed by $\bar{M}_\tau=\bar{M}_\tau^X$, consider the radius function
\begin{equation}
\mathcal{R}(y',\tau)=\left(\mathcal{H}^2(\bar K_\tau\cap \{y=y'\})\right)^{1/2}.
\end{equation}
By the Brunn-Minkowski inequality the function $y\mapsto \mathcal{R}(y,\tau)$ is concave, hence
\begin{equation}
\mathcal{R}(y,\tau)\leq \mathcal{R}(0)+ |y|\left(\mathcal{R}(y/|y|,\tau)-\mathcal{R}(0,\tau)\right).
\end{equation}
Recalling that $\|u\|_{L^{\infty}(B_{10}(0))}\leq C|\tau+2\log\hat{Z}(X)|^{-1}$, we thus obtain
\begin{equation}
\sup_{\{ |y|\leq 10|\tau+2\log\hat{Z}(X)|^{1/2} \}}\mathcal{R}(y,\tau)\leq (2\pi)^{1/2}+C|\tau+2\log\hat{Z}(X)|^{-1/2}.
\end{equation}
Together with the almost symmetry estimate from \eqref{rot_est} this establishes the lemma.
\end{proof}

Combining the above propositions we can now establish the main result of this subsection. To keep track of the error terms coming from the cutoff we consider the monotone quantity
 \begin{equation}\label{eq_omega}
 \omega(\tau)=\sup_{\tau'\leq \min\{\tau,\tau_s(X)\}}\sup_{\{|y|\leq {\frac{141}{100}} |\tau'+2\log\hat{Z}(X)|\}}|u_1(y,\tau')|.
 \end{equation}

 \begin{theorem}[differential Merle-Zaag alternative]\label{thm_diff_MZ}
 Either for all $\tau \leq \tau_\ast -2\log\hat{Z}(X)$ we have
\begin{equation}
W_-+W_0\leq \frac{C}{|\tau+2\log\hat{Z}(X)|}\left(W_{+}+ \hat{Z}(X)^{\frac{24}{25}}{e^{\frac{12}{25}\tau}\omega^2+\hat{Z}(X)^4e^{2\tau}}\right),
\end{equation}
or for $\tau\to -\infty$ we have 
\begin{equation}\label{MZ_grad_neutral}
  W_-+W_++ {e^{\frac{12}{25}\tau}\omega^2+ e^{2\tau}}= o(W_0).
\end{equation}
 \end{theorem}
 
\begin{proof}
By scaling we can assume $\hat{Z}(X)=1$.
Note that on the support of $\eta'$ the Gaussian weight satisfies
\begin{equation}
e^{-\frac{1}{4}|y|^2}\leq e^{-\frac{49}{100}\rho(\tau)^2}.
\end{equation}
Remembering \eqref{eq_rho_switch}, and taking also into account \eqref{std_cyl_est}, this yields
\begin{equation}
|\langle \mathcal{G},p_\ast' w\rangle_{\mathcal{H}}|\leq \begin{cases}
e^{\frac{485}{1000}\tau}\omega^2& \mathrm{if}\; \tau\leq\tau_s(X)\\
e^{3\tau}& \mathrm{if}\; \tau_s(X)< \tau\leq \tau_\ast.
\end{cases}
\end{equation}
Here, to illustrate how we estimated a typical term of $\langle \mathcal{G},w\rangle_{\mathcal{H}}$ for $\tau\leq \tau_s(X)$, note that abbreviating $a_{ij}=\delta_{ij}-\mathcal{E}^{ij}$ and integrating by parts we have
\begin{equation}
\int a_{ij} \eta_i(u_1)_jw e^{-\frac{|y|^2}{4}} = -\frac{1}{2}\int (u_1)^2 \left(a_{ij} \eta\eta_ie^{-\frac{|y|^2}{4}}\right)_j\leq C|\tau|^{1/2}e^{\frac{49}{100}\tau}\omega^2 \sup_{\{\eta'\neq 0\}}(|a_{ij}|+|\nabla a_{ij}|).
\end{equation}
In the same manner, we estimated the terms involving $\eta_i(u_1)_\vartheta$, and the remaining terms can be bounded directly without integration by parts.

Now, we set
\begin{equation}
\tilde{W}_{+}=W_{+}+ e^{\frac{12}{25}\tau}\omega^2+e^{2\tau}.
\end{equation}
Then, thanks to Proposition \ref{lem:evol.W_+} (unstable gradient mode), Proposition \ref{lem:evol.W_0} (neutral gradient mode) and Proposition \ref{lem:evol.W_-} (stable gradient mode) we have
\begin{align}
&\tfrac{d}{d\tau} \tilde{W}_+ \geq \tfrac{12}{25} \tilde{W}_+ - C|\tau|^{-1} \, (\tilde{W}_+ + W_0 + W_-), \nonumber\\ 
&\left | \tfrac{d}{d\tau} W_0 \right | \leq C|\tau|^{-1} \, (\tilde{W}_+ + W_0 + W_-), \label{eq:W_system}\\ 
&\tfrac{d}{d\tau} W_- \leq - {\tfrac{3}{4}}W_- + C|\tau|^{-1} \, (\tilde{W}_+ + W_0 + W_-). \nonumber
\end{align}
Together with the ODE lemma from Merle-Zaag \cite{MZ}, this implies the assertion.
\end{proof}

 \bigskip 
 
 \subsection{Ruling out the neutral slope mode}
In this subsection, we show that $W_0$ cannot dominate. To this end, given any center point $X\in\mathcal{M}$, we consider the expansion coefficients
\begin{equation}
a_k(\tau)=\langle  w(\cdot,\tau),y_k\rangle_{\mathcal{H}},\qquad (k=1,2).
\end{equation}
Note that 
\begin{equation}\label{eq:coeff.a.evol}
\tfrac{d}{d\tau} a_k=\langle \mathcal{E}w+ \mathcal{F}w , y_k\rangle_{\mathcal{H}}+O(e^{\frac{1}{4}\tau}).
\end{equation}

\begin{proposition}[leading term]\label{prop_leading}
If \eqref{MZ_grad_neutral} holds, then for $\tau\to -\infty$ we have
\begin{equation}
\langle \mathcal{F}w,y_k\rangle_{\mathcal{H}}= -(16\pi e)^{-\frac{1}{2}}a_i \langle y_i y_k, \hat{u}\rangle_{\mathcal{H}}+\frac{o(|a|+{|\tau|^{-10}})}{|\tau|}.
\end{equation}
\end{proposition}

\begin{proof}We will repeatedly use the almost symmetry estimate from \eqref{rot_est}, which in particular implies
\begin{equation}
W_0^{1/2}\leq C(|a|+|\tau|^{-10}).
\end{equation}
Observing that $|\mathcal{F}|\leq C(|u|+|u_\vartheta|+|u_{\vartheta\vartheta}|)$ on the support of $\eta$, and using \eqref{inv_poinc} and \eqref{MZ_grad_neutral}, we can estimate
\begin{equation}
\left| \langle 1_{\{\eta>0\}}\mathcal{F} (w-p_0'w) , y_k \rangle_{\mathcal{H}}\right|  \leq C {\|w-p_0'w\|_{\mathcal{H}}}(\|\hat{u}\|_{\mathcal{H}}+|\tau|^{-9}) =\frac{o(|a|+|\tau|^{-10})}{|\tau|}.
\end{equation}
Since $\| y_k \|_{\mathcal{H}}=(8\pi e)^{1/4}$ in our normalization, this shows that
\begin{equation}
\langle \mathcal{F}w,y_k\rangle_{\mathcal{H}}= (8\pi e)^{-1/2} a_i \langle y_iy_k, 1_{\{\eta>0\}}\mathcal{F}\rangle_{\mathcal{H}}+\frac{o(|a|+|\tau|^{-10})}{|\tau|}.
\end{equation}
Moreover, observing that $|\mathcal{F}+2^{-1/2}u|\leq C(u^2+|u_\vartheta|+|u_{\vartheta\vartheta}|)$ on the support of $\eta$, using again \eqref{inv_poinc} we can estimate
\begin{equation}
\left| \langle y_iy_k, 1_{\{\eta>0\}}(\mathcal{F}+2^{-1/2}u)\rangle_{\mathcal{H}}\right|\leq C|\tau|^{-2}.
\end{equation}
Combining the above estimates yields the assertion.
\end{proof}

\begin{proposition}[error term]\label{prop_error}
If \eqref{MZ_grad_neutral} holds, then for $\tau\to -\infty$ we have
\begin{equation}
\langle \mathcal{E}w,y_k\rangle_{\mathcal{H}}= \frac{o(|a|+|\tau|^{-10})}{|\tau|}.
\end{equation}
\end{proposition}

\begin{proof} We will frequently use the derivative estimates \eqref{std_cyl_est} and \eqref{rot_est}.
First, since $w_{ij}= (p_{-}'w)_{ij}$, we can estimate
\begin{equation}
 \left|\int  \mathcal{E}^{ij} w_{ij}y_k   e^{-\frac{|y|^2}{4}}\right| \leq   \int |p_-'w| \left| \left[\eta \mathcal{E}^{ij}y_k e^{-\frac{|y|^2}{4}}\right]_{ij} \right| +Ce^{\frac14 \tau}= \frac{o(|a|+|\tau|^{-10})}{|\tau|}.
\end{equation}
Next, using $w_j=(8\pi e)^{-1/4}a_j+(p_-'w)_j$, we see that
\begin{equation}
  \int  \mathcal{E}^jw_jy_k e^{-\frac{|y|^2}{4}}= (8\pi e)^{-1/4}a_j\int  \mathcal{E}^jy_ke^{-\frac{|y|^2}{4}}dy+\frac{o(|a|+|\tau|^{-10})}{|\tau|}.
\end{equation}
Regarding the first summand, note that 
\begin{equation}
\int  \mathcal{E}^jy_ke^{-\frac{|y|^2}{4}}dy=\int  2\hat{u}_i\hat{u}_{ij} y_k e^{-\frac{|y|^2}{4}}dy+\frac{o(1)}{|\tau|}\leq \frac{C\|\nabla \hat u\|_{\mathcal{H}}}{|\tau|^{1/2}}+\frac{o(1)}{|\tau|}=\frac{o(1)}{|\tau|}.
\end{equation}
Finally, observe that
\begin{equation}
 \int  \mathcal{E}^{j\vartheta} w_{j\vartheta} y_k  e^{-\frac{|y|^2}{4}}=\int  w \left[ (\mathcal{E}^{i\vartheta})_{\vartheta} y_k   e^{-\frac{|y|^2}{4}} \right]_j=\frac{O(|a|+|\tau|^{-10})}{|\tau|^{9}},
\end{equation}
and by using integration by parts again that
\begin{equation}
 \int  (\mathcal{E}^{\vartheta\vartheta} w_{\vartheta\vartheta}+\mathcal{E}^{\vartheta} w_{\vartheta}) y_k  e^{-\frac{|y|^2}{4}}=\frac{O(|a|+|\tau|^{-10})}{|\tau|^{9}}.
\end{equation}
Combining the above estimates yields the assertion.
\end{proof}

\begin{theorem}[neutral slope mode cannot dominate]\label{thm_unst_dom}
For all $\tau \leq \tau_\ast -2\log\hat{Z}(X)$ we have
\begin{equation}
W_-+W_0\leq \frac{C}{|\tau+2\log\hat{Z}(X)|}\left(W_{+}+ \hat{Z}(X)^{\frac{24}{25}}e^{\frac{12}{25}\tau}\omega^2+\hat{Z}(X)^4e^{2\tau}\right),
\end{equation}
\end{theorem}

\begin{proof}
If the assertion fails, then by Theorem \ref{thm_diff_MZ} (differential Merle-Zaag alternative) it must be the case that \eqref{MZ_grad_neutral} holds. Hence, Proposition \ref{prop_leading} (leading term) and Proposition \ref{prop_error} (error term) imply
\begin{equation}
\tfrac{d}{d\tau} a_k= -(16\pi e)^{-\frac{1}{2}}a_i \langle y_i y_k, \hat{u}\rangle_{\mathcal{H}}+\frac{o(|a|+|\tau|^{-10})}{|\tau|}.
\end{equation}
Together with the profile function asymptotics from \eqref{DH_asymp}, it follows that
\begin{align}
\tfrac{d}{d\tau} a_1=\frac{o(|a|)}{|\tau|}+{o(|\tau|^{-10})},\qquad  \tfrac{d}{d\tau}a_2=\frac{a_2}{|\tau|}+\frac{o(|a|)}{|\tau|}+{o(|\tau|^{-10})}.
\end{align}
Hence, considering $\frac{d}{d\tau} \log (|a|^2+|\tau|^{-5})$, for all $\delta>0$ we get 
\begin{equation}
\liminf_{\tau\to -\infty} |\tau|^{1+\delta} |a(\tau)| = \infty.
\end{equation}
This contradicts \eqref{DH_asymp}, and thus proves the theorem.
\end{proof}

 \begin{corollary}[exponential decay]\label{cor:fast.decay}
For all $\tau \leq \tau_\ast -2\log\hat{Z}(X)$ we have
 \begin{equation}
\| w\|_{\mathcal{H}}\leq C \hat{Z}(X)^{\frac{2}{5}}e^{\frac{1}{5}\tau}.
\end{equation}
\end{corollary}

\begin{proof}Assume $\hat{Z}(X)=1$ as usual.
By Theorem \ref{thm_unst_dom} (neutral slope mode cannot dominate) the function $\tilde{W}_+$ dominates, and the first ODE from \eqref{eq:W_system} yields
$\tfrac{d}{d\tau}\log \tilde{W}_+ \geq   \tfrac{2}{5}$.
This implies the assertion.
\end{proof}

\bigskip 

\section{Cap distance and consequences}\label{sec_cap_and_co}

In this section, we establish a lower bound for the renormalized cap distance
\begin{equation}
\bar{d}^X(\tau)=\min \left\{ \sup_{\bar{M}^X_\tau}y_1, -\inf_{\bar{M}^X_\tau}y_1\right\},
\end{equation}
and also derive some estimates for related geometric quantities.

\begin{theorem}[cap distance]\label{thm:cap_distance}
For $\tau\leq\tau_\ast-2\log\hat{Z}(X)$ we have
\begin{equation}
\bar{d}^X(\tau)\geq |\tau+2\log\hat{Z}(X)|^{10}.
\end{equation}
Moreover, for $\tau\leq\tau_\ast-2\log\hat{Z}(X)$ we also have
\begin{equation}
\sup_{\{ |\tau+2\log\hat{Z}(X)|^{-19}y_1^2+(2-\beta)^{-1}y_2^2 \leq |\tau+2\log\hat{Z}(X)|\}} |\bar{A}|^2 \leq 5\beta^{-1},
\end{equation}
and
\begin{equation}
\sup_{\{  |y_1|\leq |\tau+2\log\hat{Z}(X)|^{10},\; |y_2|\leq 10 \} }   |\nu_1|  \leq \hat{Z}(X)^{\frac{1}{5}}e^{\frac{1}{10}\tau}.
\end{equation}
\end{theorem}

\begin{proof}
By scaling we can assume $\hat{Z}(X)=1$.
Moreover, for ease of notation we abbreviate
\begin{equation}
\bar{f}_k=\beta^{2-2k}L^{2k}|\tau|^{-k} y_1^2 + (2-\beta)^{-1} y_2^2.
\end{equation}
To begin with, recall that by Theorem \ref{thm_elong_cyl} (elongated cylindrical regions) for $\tau\leq\tau_\ast$ the renormalized flow is defined in the ellipsoidal region $\{\bar{f}_0 \leq |\tau| \}$ with the estimate
\begin{equation}\label{curv_bound0}
\sup_{\{\bar{f}_0 \leq |\tau| \} }|\bar{A}|^2\leq 5\beta^{-1}.
\end{equation}
Moreover, by Corollary \ref{cor:fast.decay} (exponential decay) and standard interior estimates we have
\begin{equation}
\sup_{\{ \bar{f}_0 \leq L^2\} }|\nu_1|\leq C e^{\frac{1}{5}\tau}.
\end{equation}
Thus, applying Theorem \ref{thm:propagation} (anisotropic propagation of smallness) with $p=\frac{1}{5}$ we get
\begin{equation}
\sup_{\{\bar{f}_0\leq  \beta^2|\tau|\} }|\nu_1| \leq Ce^{(\frac{1}{5}-2\beta)\tau}.
\end{equation}
Now, given any integer $k\in \{0,\ldots,19\}$, assume recursively that for $\tau\leq\tau_\ast$ the renormalized flow is defined in $\{\bar{f}_k \leq |\tau| \}$ with the estimates
\begin{equation}
\sup_{\{\bar{f}_k \leq |\tau| \} }|\bar{A}|^2\leq 5\beta^{-1} \qquad\mathrm{and}\qquad \sup_{\{\bar{f}_k \leq \beta^2 |\tau| \} }|\nu_1|\leq Ce^{(\frac{1}{5}-2(k+1)\beta)\tau}.
\end{equation}
Also, note that $\{\bar{f}_{k+1}\leq L^2\}\subset \{\bar{f}_{k}\leq \beta^2|\tau|\}$. Hence, applying Corollary \ref{cor_elongated_barr} (very elongated cylindrical regions) we infer that for $\tau\leq\tau_\ast$ the renormalized flow is defined in $\{\bar{f}_{k+1} \leq |\tau| \}$ with the estimate
\begin{equation}
\sup_{\{\bar{f}_{k+1} \leq |\tau| \} }|\bar{A}|^2\leq 5\beta^{-1},
\end{equation}
and applying Theorem \ref{thm:propagation} (anisotropic propagation of smallness) with $p=\frac{1}{5}-2(k+1)\beta$ we get
\begin{equation}
\sup_{\{\bar{f}_{k+1} \leq \beta^2 |\tau| \} }|\nu_1|\leq Ce^{(\frac{1}{5}-2(k+2)\beta)\tau}.
\end{equation}
Therefore, the estimates hold for $k=20$, and this implies the assertion.
\end{proof}

\begin{corollary}[geometric bounds]\label{cor_geom_cons}
For $\tau\leq \tau_s(X)$ in the region $\{|y_1|\leq 2|\tau+2\log\hat{Z}(X)|^7\}\cap \{y_2 \geq 10\}$ we have
\begin{align}
&|\nu_1|\leq \frac{1}{|\tau+2\log\hat{Z}(X)|^{3/2}}, && \nu_2\geq \frac{\kappa_0}{|\tau+2\log\hat{Z}(X)|}, && \bar{H}\leq C|\tau+2\log\hat{Z}(X)|.
\end{align}
\end{corollary}

\begin{proof}
By scaling we can assume $\hat{Z}(X)=1$.
By Corollary \ref{cor_switch_time} (switch time), taking also into account Lemma \ref{prop:domanance_criteria} (quantitative Merle-Zaag lemma) and standard parabolic estimates, for $\tau\leq\tau_s(X)$ we have
\begin{equation}
u(0,0,\vartheta,\tau)-u(0,10,\vartheta,\tau) \geq 15\kappa_0|\tau|^{-1}.
\end{equation}
Together with Theorem \ref{thm:cap_distance} (cap distance) this yields
\begin{equation}
\sup_{\{|y_1|\leq|\tau|^{10}\}}u_2(y_1,10,\vartheta,\tau) \leq - \tfrac32\kappa_0|\tau|^{-1}.
\end{equation}
We thus get the lower bound $\nu_2\geq \kappa_0|\tau|^{-1}$ in the claimed region, and also get the width bound
\begin{align}\label{eq_level_diam}
\sup_{\{|y_1|\leq|\tau|^{10}\}} |y_2|\leq \kappa_0^{-1}|\tau|.
\end{align}
In particular, by the Harnack estimate \eqref{Harnack_curv} this implies $\bar{H}\leq 2\kappa_0^{-1}|\tau|$ in $\{|y_1|\leq 2|\tau|^7\}$. Finally, if at some point in $\{|y_1|\leq 2|\tau|^7\}$ we had $|\nu_1|>|\tau|^{-3/2}$, then considering the tangent plane at that point we would obtain a contradiction with the bound $\bar{d}^X(\tau)\geq |\tau|^{10}$ from Theorem \ref{thm:cap_distance} (cap distance).
\end{proof}

As a corollary of the proof we also obtain the following result, where the error term has the exponent
\begin{equation}
\hat{q}=q-7\beta^2 > 1/2.
\end{equation}

\begin{corollary}[improved anisotropic propagation of smallness]\label{thm:smallness_center}
Suppose that 
\begin{equation}\label{eq_coarse_fast_decay}
\sup_{\{|y_1|\leq |\tau+2\log\hat{Z}(X)|^7\}}|\nu_1|\leq \tfrac{1}{2}\hat{Z}(X)^{2\beta}e^{\beta\tau}\qquad \textrm{for all}\;\; \tau\leq\tau_*-2\log\hat{Z}(X).
\end{equation}
Then, for every $\tau \leq \tau_\ast-2\log\hat{Z}(X)$ and $p\in [0,\hat{q}]$ we get
\begin{equation}
\sup_{\{|y_1|\leq 2|\tau+2\log\hat{Z}(X)|^3,\; |y_2|\leq 10\}}  |\nu_1|  \leq  e^{p\tau}\sup_{\tau'\leq \tau}\sup_{\{ |y| \leq L/\beta\}} \frac{e^{-(p+7\beta^2)\tau'}}{\hat{Z}(X)^{14\beta^2}}|\nu_1|+7\hat{Z}(X)^{2\hat{q}}e^{\hat{q}\tau}.
\end{equation}
\end{corollary}
 
\begin{proof}
By scaling we can assume $\hat{Z}(X)=1$.
Thanks to \eqref{curv_bound0} and \eqref{eq_coarse_fast_decay}, applying Theorem \ref{thm:propagation} (anisotropic propagation of smallness) we get
\begin{align}
\sup_{\{\bar{f}_0\leq \beta^2 |\tau|\}}|\nu_1| \leq e^{ p \tau}\sup_{\tau'\leq \tau}\sup_{\{\bar{f}_0\leq L^2\}} e^{-(p+\beta^2)\tau'}|\nu_1|+e^{(q-\beta^2)\tau}.
\end{align}
Iterating this argument six more times we obtain
\begin{align}
\sup_{\{\bar{f}_6\leq \beta^2 |\tau|\}}|\nu_1| \leq e^{ p \tau}\sup_{\tau'\leq \tau}\sup_{\{\bar{f}_0\leq L^2\}} e^{-(p+7\beta^2)\tau'}|\nu_1|+7e^{(q-7\beta^2)\tau}.
\end{align}
This implies the assertion.
\end{proof} 
 
\bigskip 

\section{Level set propagation of smallness}\label{sec_level_set_propagation}

In this section, we show that smallness of $|\nu_1|$ spreads out well in $x_2$-direction over the entire level sets. To begin with, centering at $X=0$ for ease of notation, given any $b\in [2,7]$, we consider the function
\begin{align}\label{def_super_hor}
\Psi_b=\exp\left( \frac{x_1^2}{4|t|}-\left(\log|t|-2\log\hat{Z}(X)\right)^{2b}\right).
\end{align}

\begin{proposition}[Jacobi supersolution]\label{prop_Jac_super_long}
For $t\leq -e^{-\tau_s(X)}$ we have
\begin{equation}
(\partial_t - \Delta -|A|^2)\Psi_b>0 \qquad\textrm{in}\;\; \big\{ |x_1|/|t|^{1/2}\leq 2(\log|t|-2\log\hat{Z}(X))^b\big\}.
\end{equation}
\end{proposition}

\begin{proof}
First, note that $f_1=x_1^2/|t|$ satisfies
\begin{equation}\label{partial_f1}
(\partial_t-\Delta)f_1 \geq \frac{f_1-2}{|t|}\geq \frac{1}{4}|\nabla f_1|^2-\frac{2}{|t|}.
\end{equation}
This yields
\begin{equation}
\frac{(\partial_t-\Delta)\Psi_b}{\Psi_b} \geq  \frac{2b(\log|t|-2\log\hat{Z}(X))^{2b-1}}{|t|}-\frac{1}{2|t|}.
\end{equation}
Hence, Corollary \ref{cor_geom_cons} (geometric bounds) implies the assertion.
\end{proof} 
 
\begin{theorem}[level set propagation of smallness]\label{thm:horizontal_propagation}
For $\tau\leq \tau_s(X)$ we have
\begin{equation}
\sup_{\{    |y_1| \leq |\tau+2\log\hat{Z}(X)|^b\}}|\nu_1| \leq C \sup_{\tau'\leq \tau}\sup_{\{  |y_1| \leq \, 2 |\tau'+2\log\hat{Z}(X)|^b,\; |y_2|\leq 10\}} |\tau'+2\log\hat{Z}(X)||\nu_1|+\hat{Z}(X)^{20}e^{10\tau}.
\end{equation}
\end{theorem}

\begin{proof}
Suppose $\hat{Z}(X)=1$ and $X=0$ for ease of notation. Working in the unrescaled variables, for $t\leq -e^{-\tau_s(X)}$ we consider the domains
\begin{equation}
\Omega_t =\left\{    |x_1| \leq 2 |t|^{1/2} (\log|t|)^b,\;\; x_2\geq 10|t|^{1/2} \right\}.
\end{equation}
By Corollary \ref{cor_geom_cons} (geometric bounds) the function $\nu_2$ is positive in $\Omega_t$ with the lower bound
\begin{equation}\label{nu_2_lower}
\nu_2\geq \frac{\kappa_0}{\log|t|},
\end{equation}
Now, given any $\bar{t}\leq -e^{-\tau_s(X)}$, set
\begin{equation}
K=\tfrac{1}{\kappa_0}\sup_{t\leq\bar{t}}\sup_{\{  |x_1| \leq 2 |t|^{1/2} (\log|t|)^b,\; x_2=10|t|^{1/2} \}} |\nu_1|\log|t|.
\end{equation}
We would like to compare $\pm \nu_1$, which solves the Jacobi equation, with $K \nu_2+\Psi_b$, which is a supersolution thanks to Proposition \ref{prop_Jac_super_long} (Jacobi supersolution).
To this end, note first that on the horizontal boundary $\Omega_t\cap \{ x_2=10|t|^{1/2}\}$ by the definition of $K$ and  \eqref{nu_2_lower} we have $|\nu_1|\leq K\nu_2$. Next, on the vertical boundary $\Omega_t\cap \{  |x_1| = 2 |t|^{1/2}(\log|t|)^b\}$ we have $\Psi_b=1$, and thus in particular $|\nu_1|\leq \Psi_b$. Furthermore, thanks to Corollary \ref{cor_geom_cons} (geometric bounds) the ratio $\nu_1/\nu_2$ satisfies the Cauchy condition
\begin{equation}
\lim_{t\to -\infty}\sup_{\Omega_t}\frac{|\nu_1|}{\nu_2}=0.
\end{equation}
Therefore, by the comparison principle we infer that $|\nu_1|\leq K\nu_2+\Psi_b$ in $\Omega_t$ for $t\leq\bar t$. Finally, restricting to the smaller region where $|x_1| \leq  | t|^{1/2}(\log| t|)^b$ we have $\Psi_b\leq |t|^{-100}$. This implies the assertion.
\end{proof}
  
Combining the established propagation of smallness results we obtain:

\begin{corollary}[combined propagation of smallness]\label{cor:gradient_estimate_bootstrap}
For $\tau\leq \tau_s(X)$ and $p\in [0,\hat{q}]$ we have
\begin{equation}
\sup_{\{    |y_1| \leq |\tau+2\log\hat{Z}(X)|^3\}}|\nu_1| \leq  C|\tau+2\log\hat{Z}(X)|e^{p\tau} \sup_{\tau'\leq \tau}\sup_{\{ |y| \leq L/\beta\}} \frac{e^{-(p+7\beta^2)\tau'}}{\hat{Z}(X)^{14\beta^2}} |\nu_1|+C\hat{Z}(X)^{2\hat{q}}e^{\hat{q}\tau}.
\end{equation}
\end{corollary} 
 
\begin{proof}
By scaling we can assume $\hat{Z}(X)=1$. Applying Theorem \ref{thm:cap_distance} (cap distance) and Theorem \ref{thm:horizontal_propagation} (level set propagation of smallness) with $b=7$ we see that the inequality \eqref{eq_coarse_fast_decay} is satisfied. Hence, Corollary \ref{thm:smallness_center} (improved anisotropic propagation of smallness) yields
\begin{equation}
\sup_{\{|y_1|\leq 2|\tau|^3,\; |y_2|\leq 10\}}  |\nu_1| \leq  e^{ p\tau} \sup_{\tau'\leq \tau}\sup_{\{ |y| \leq L/\beta\}} e^{-(p+7\beta^2)\tau'}|\nu_1|+7e^{\hat{q}\tau}.
\end{equation}
Applying Theorem \ref{thm:horizontal_propagation} (level set propagation of smallness) again, now with $b=3$, finishes the proof.
\end{proof} 

\bigskip
 
\section{Differential neck theorem}\label{sec_DNT}

In this section, we combine the results from the previous section to establish our differential neck theorem.
 
\subsection{Weak differential neck theorem}

In this subsection, we consider the expansion coefficient
\begin{equation}
a_+(\tau)=\langle w(\cdot,\tau),1\rangle_{\mathcal{H}},\qquad \mathrm{where }\;\; w(y,\tau)=\eta(|y|/\rho(\tau)) u_1^X(y,\tau).
\end{equation}
To see that the error terms in the evolution of $a_+(\tau)$ are integrable, we start with the following lemma.

\begin{lemma}[quadratic decay]\label{lem:profile_integration}
For $\tau\leq \tau_*-2\log \hat{Z}(X)$ we have
\begin{equation}
\left|\int_{\{\eta>0\}} u e^{-\frac{|y|^2}{4}} \right| \leq \frac{C}{|\tau+2\log \hat{Z}(X)|^{2}}.
\end{equation}
\end{lemma}

\begin{proof}By scaling we can assume $\hat{Z}(X)=1$.
To begin with, note that by the evolution equation \eqref{evol_u}, taking also into account the almost symmetry estimate \eqref{rot_est}, the function $\tilde{u}(y,\tau)=\eta(|y|/\rho(\tau))u(y,\tau)$ satisfies
\begin{equation}
\left|\int \left(\tilde{u}_\tau-\mathcal{L}\tilde{u} +\frac{u_iu_j\tilde{u}_{ij}}{1+|\nabla u|^2}+\frac{u\tilde{u}}{2(\sqrt{2}+u)}\right)e^{-\frac{|y|^2}{4}}\right|\leq C|\tau|^{-5/2}.
\end{equation}
Using the derivative estimates from \eqref{std_cyl_est} and the inverse Poincare inequality from \eqref{inv_poinc}, and remembering also that $\|\hat u\|_{\mathcal{H}}\leq C|\tau|^{-1}$, we can estimate the remaining nonlinear terms by
\begin{equation}
\left|\int \frac{u_iu_j\tilde{u}_{ij}}{1+|\nabla u|^2}e^{-\frac{|y|^2}{4}}\right|+\left|\int \frac{u\tilde{u}}{2(\sqrt{2}+u)}e^{-\frac{|y|^2}{4}}\right| \leq C|\tau|^{-2}.
\end{equation}
It follows that the function $U=\int \tilde{u}e^{-\frac{|y|^2}{4}}$ satisfies the differential inequality
\begin{equation}
|\tfrac{d}{d\tau}U-U|\leq C|\tau|^{-2},
\end{equation}
and so the variations of constants formula yields $|U|\leq C|\tau|^{-2}$. This implies the assertion.
\end{proof}

\begin{proposition}[evolution of expansion coefficient]\label{prop:W_+_evol_improved}
 For $\tau\leq \tau_*-2\log \hat{Z}(X)$ we have
\begin{equation}
\left|\tfrac{d}{d\tau} a_+-\tfrac{1}{2} a_+ \right| \leq \frac{C|a_+|}{|\tau+2\log \hat{Z}(X)|^{2}} + \frac{C(W_0+W_{-})^{\frac{1}{2}}}{|\tau+2\log \hat{Z}(X)|}+C\left( \hat{Z}(X)^{\frac{24}{25}}e^{\frac{12}{25}\tau}\omega+\hat{Z}(X)^{4}e^{2\tau}\right).
\end{equation}
\end{proposition}

\begin{proof}By scaling we can assume $\hat{Z}(X)=1$.
We will improve the estimates from Proposition \ref{lem:evol.W_+} (unstable slope mode). We begin by computing
\begin{equation}
\tfrac{d}{d\tau}a_+=\tfrac{1}{2} a_++\langle \mathcal{E}w+ \mathcal{F}w,1\rangle_{\mathcal{H}}+\langle \mathcal{G},1\rangle_{\mathcal{H}}.
\end{equation}
Remembering definitions of $\rho$ and $\omega$ from \eqref{eq_rho_switch} and \eqref{eq_omega}, and using \eqref{std_cyl_est}, we can estimate
\begin{equation}
\left|\langle  \mathcal{G},1\rangle_{\mathcal{H}}\right| \leq C\left(e^{\frac{12}{25}\tau}\omega +e^{2\tau}\right).
\end{equation}
Since $\mathcal{E}$ annihilates constants, by subtracting its average from $w$, we can improve \eqref{eq_F1_estimate_L2} to
\begin{equation}
\left|\int \mathcal{E}(w) e^{-\frac{|y|^2}{4}}\right|\leq C|\tau|^{-1}(W_0+W_-)^{\frac{1}{2}}.
\end{equation}
Likewise, using Lemma \ref{lem:profile_integration} (quadratic decay) we can improve  \eqref{eq_F2_estimate_L2} to
\begin{equation}
\left|\int  \mathcal{F} w e^{-\frac{|y|^2}{4}}\right|\leq C|\tau|^{-2}|a_{+}|+C|\tau|^{-1}(W_0+W_-)^{\frac{1}{2}}.
\end{equation}
Combining the above inequalities implies the assertion.
\end{proof}

\begin{theorem}[weak differential neck theorem]\label{thm:fine_neck_L2}
There exists a constant $a=a(\mathcal{M})\in\mathbb{R}$, such that for all $X\in\mathcal{M}$ the function $w(y,\tau)=\eta(|y|/\rho(\tau)) u_1^X(y,\tau)$ for all $\tau\leq \tau_*-2\log \hat{Z}(X)$ satisfies
\begin{equation}
\big\|w-ae^{\frac{1}{2}\tau}\big\|_{\mathcal{H}}\leq \eps(\tau)e^{\frac{1}{2}\tau},\qquad\mathrm{where}\;\; \eps(\tau)\leq\frac{C|a|}{|\tau+2\log \hat{Z}(X)|^{\frac{1}{2}}}+C\hat{Z}(X)^{\frac{6}{5}}e^{\frac{1}{10}\tau}.
\end{equation}
\end{theorem}

\begin{proof}For ease of notation, let us focus first on points with $Z(X)\leq 1$ (for general points one has to replace $\tau$ by $\tau+2\log \hat{Z}(X)$). To begin with, note that Proposition \ref{prop:W_+_evol_improved} (evolution of expansion coefficient), taking also into account 
Theorem \ref{thm_unst_dom} (neutral slope mode cannot dominate), implies the evolution inequality
\begin{equation}\label{evol_ineq_w}
\left|\tfrac{d}{d\tau} W_+- W_+ \right| \leq C|\tau|^{-3/2}W_+ + C|\tau|^{-3/2}\big( e^{\frac{12}{25}\tau}\omega^2+e^{2\tau}\big).
\end{equation}
Now, by Corollary \ref{cor:fast.decay} (exponential decay) and standard interior estimates we have
\begin{equation}
\sup_{\{|y|\leq L/\beta\}}|u_1|\leq Ce^{\frac{1}{5}\tau}.
\end{equation}
Thus, applying Corollary \ref{cor:gradient_estimate_bootstrap} (combined propagation of smallness) with $p=\frac{1}{5}-7\beta^2$ we get
\begin{equation}\label{eq_omega_b}
\omega\leq Ce^{(\frac{1}{5}-8\beta^2)\tau}.
\end{equation}
Here, while Corollary \ref{cor:gradient_estimate_bootstrap} (combined propagation of smallness) only applies for $\tau\leq \tau_s(X)$, since $\omega$ is constant for $\tau\geq \tau_s(X)$, the estimate \eqref{eq_omega_b} holds  for $\tau_s(X)<\tau\leq\tau_\ast$ as well. 
It follows that
\begin{equation}
\frac{d}{d\tau}\log \left(W_++e^{\frac{21}{25}\tau}\right) \geq \frac{4}{5}.
\end{equation}
This yields 
\begin{equation}
\sup_{\{|y|\leq L/\beta\}}|u_1|\leq Ce^{\frac{2}{5}\tau}.
\end{equation}
Thus, using Corollary \ref{cor:gradient_estimate_bootstrap} (combined propagation of smallness) with $p=\frac{2}{5}-7\beta^2$, Proposition \ref{prop:W_+_evol_improved} (evolution of expansion coefficient) and Theorem \ref{thm_unst_dom} (neutral slope mode cannot dominate) we get
\begin{equation}
\left|\tfrac{d}{d\tau} a_+-\tfrac{1}{2} a_+ \right| \leq C|\tau|^{-\frac{3}{2}}|a_+| +Ce^{\frac{3}{5}\tau}.
\end{equation}
Integrating this differential inequality, and taking also into account again Theorem \ref{thm_unst_dom} (neutral slope mode cannot dominate), we infer that there is some $a_0(X)\in\mathbb{R}$, such that for all $\tau\leq \tau_*$ we have
\begin{equation}
\big\|w-a_0(X)e^{\frac{1}{2}\tau}\big\|_{\mathcal{H}}\leq C|a_0(X)| |\tau|^{-\frac{1}{2}}e^{\frac{1}{2}\tau}+Ce^{\frac{3}{5}\tau}.
\end{equation}
For a general center-point X, i.e. dropping the extra assumption that $\hat{Z}(X)$ equals $1$, we can replace $\tau$ by $\tau+2\log \hat{Z}(X)$ and set $a(X)=a_0(X)\hat{Z}(X)$ to infer that for all $\tau\leq \tau_*-2\log\hat{Z}(X)$ we have
\begin{equation}\label{dn_gen_cent}
\big\|w-a(X)e^{\frac{1}{2}\tau}\big\|_{\mathcal{H}}\leq C|a(X)| |\tau+2\log \hat{Z}(X)|^{-\frac{1}{2}}e^{\frac{1}{2}\tau}+C\hat{Z}(X)^{\frac{6}{5}}e^{\frac{3}{5}\tau}.
\end{equation}
Finally, given any two center-points $X$ and $X'$, by considering the established estimate \eqref{dn_gen_cent} at very negative times we see that in fact $a(X)=a(X')$. This concludes the proof of the theorem.
\end{proof} 
 
\bigskip   

\subsection{Ruling out degenerate differential necks} 
 
In this subsection, we show that $a(\mathcal{M})$ does not vanish. Normalizing such that $X=(0,0)\in\mathcal{M}$, we start with the following basic curvature bound.

\begin{lemma}[curvature bound]\label{prop_global_curvature}
For all $\tau$ sufficiently negative we have
\begin{equation}
\sup_{\{|y_1|\leq \beta \bar{d}(\tau)\}}\bar{H} \leq 10|\tau|^{1/2}\qquad\mathrm{and }\qquad \inf_{\{|y_1|\leq \beta \bar{d}(\tau)\}}\bar{H} \geq \tfrac{1}{10}.
\end{equation}
\end{lemma} 
 
\begin{proof}
Denote by $\bar{K}_\tau$ the region enclosed by $\bar{M}_\tau$. As a consequence of \eqref{DH_asymp} for $\tau$ sufficiently negative the level set $\bar{K}_\tau\cap \{y_1=0\}$ has diameter at most $5|\tau|^{1/2}$. Thus, by convexity the level set $\bar{K}_\tau\cap \{y_1=h\}$ with $|h| \leq \beta \bar d(\tau)$ has diameter at most $6|\tau|^{1/2}$, and hence the Harnack estimate \eqref{Harnack_curv} yields the upper bound $\bar{H}\leq 10|\tau|^{1/2}$. Similarly, again by \eqref{DH_asymp} and convexity, for $\tau\ll 0$ we have $y_3^2+y_4^2 \leq 5 $ in $\{|y_1|\leq \beta \bar d(\tau)\}$. By the noncollapsing property, this implies the lower bound $\bar{H}\geq \tfrac{1}{10}$, and thus concludes the proof.
\end{proof}
We now suppose towards a contradiction that
\begin{equation}\label{eq_assum_supersmall}
\lim_{\tau\to -\infty}  \sup_{\{|y|\leq L/\beta\}} e^{-q\tau} |u_1(y,\vartheta,\tau)| =0.
\end{equation}

\begin{proposition}[propagation of super-smallness]\label{thm:super_small}
Suppose that \eqref{eq_assum_supersmall} holds. If for some $\delta \in (0,1/2)$ and $\bar{\tau}\ll 0$ we have $\bar{d}(\tau)\geq 10\delta^{-1}\beta^{-1}|\tau|$ for all $\tau\leq\bar{\tau}$, then $|\nu_1|\leq e^{100\tau}$ holds in $\{\delta |y_1|\leq 5|\tau|\}$ for $\tau\leq\bar{\tau}$.
\end{proposition} 

\begin{proof}Working in the unrescaled variables, we consider the function
\begin{align}\label{def_super_hor}
\Phi_\delta=\exp\left({ \frac{\delta^2 x_1^2}{|t|}-100(\log|t|)^2}\right).
\end{align}
Remembering \eqref{partial_f1}, we see that
\begin{equation}
\frac{(\partial_t-\Delta)\Phi_\delta}{\Phi_\delta}=\delta^2 (\partial_t-\Delta)f_1-\delta^4 |\nabla f_1|^2+200\frac{\log|t|}{|t|}\geq \frac{100\log|t|}{|t|}.
\end{equation}
Together with Lemma \ref{prop_global_curvature} (curvature bound) this shows that $\Phi_\delta$ is a supersolution to the Jacobi equation in
$\Omega_t=\{\delta|x_1|\leq 10|t|^{1/2}(\log |t|)^{1/2}\}$ for $t$ sufficiently negative. Hence, we can compare $\pm \nu_1$ with $\Phi_\delta +\varepsilon H$, where $\eps>0$. Since $\Phi_\delta=1$ on $\partial \Omega_t$, we have $| \nu_1|\leq \Phi_\delta$ on $\partial \Omega_t$. Moreover, by assumption \eqref{eq_assum_supersmall} and Corollary \ref{cor:gradient_estimate_bootstrap} (combined propagation of smallness), taking also into account Lemma \ref{prop_global_curvature} (curvature bound) again, we have
\begin{equation}
\lim_{t\to -\infty}\sup_{\{\delta|x_1|\leq 10|t|^{1/2}(\log |t|)^{1/2}\}}\frac{|\nu_1|}{H}=0.
\end{equation}
Hence, the maximum principle yields $|\nu_1|\leq \Phi_\delta+\varepsilon H$. Since $\varepsilon$ was arbitrary, this implies the assertion.
\end{proof}

\begin{theorem}[nondegenerate neck theorem]\label{thm:non_degenerate_neck}
We have $a\neq 0$.
\end{theorem} 
  
\begin{proof}
Suppose towards a contradiction that $a=0$. Then, using Theorem \ref{thm:fine_neck_L2} (weak differential neck theorem) and standard interior estimates we see that the condition \eqref{eq_assum_supersmall} holds. Moreover, 
by Theorem \ref{thm:cap_distance} (cap distance) we can choose a sequence $\tau_i\to -\infty$, such that $\bar{d}(\tau)/|\tau|\geq \bar{d}(\tau_i)/|\tau_i| > 20  \beta^{-1}$ for $\tau\leq \tau_i$. Then, setting $\delta_i=\frac{10|\tau_i|}{\beta \bar{d}(\tau_i)}$, by Proposition \ref{thm:super_small} (propagation of super-smallness) for $\tau\leq\tau_i$ we get
\begin{equation}
\sup_{ \{ |y_1|\leq\frac12 \beta \bar{d}(\tau_i) \} } |\nu_1|\leq e^{100\tau}.
\end{equation}
Note also that by Hamilton's Harnack inequality  \cite{Hamilton_Harnack}  we have $\bar{d}(\tau)\leq Ce^{-\frac12 \tau}$.
Hence, applying Corollary \ref{cor_elongated_barr} (very elongated cylindrical regions) with $\alpha_i=\frac{4L^2}{\beta^2 \bar{d}(\tau_i)^2}$ we infer that
\begin{equation}
\bar{d}(\tau_i)\geq \frac{\beta \bar{d}(\tau_i)}{2L}|\tau_i|^{1/2}.
\end{equation}
This contradicts the fact that $\tau_i\to -\infty$, and thus proves the theorem.
\end{proof}

\bigskip

\subsection{Strong differential neck theorem}

Recall that in Theorem \ref{thm:fine_neck_L2} (weak differential neck theorem) we proved the error estimate
\begin{equation}
 \eps(\tau)\leq\frac{C|a|}{|\tau+2\log \hat{Z}(X)|^{\frac{1}{2}}}+C\hat{Z}(X)^{\frac{6}{5}}e^{\frac{1}{10}\tau}.
\end{equation}
While this is perfectly satisfying for solutions with bounded bubble-sheet scale, a weakness in case $\hat{Z}(X)\gg 1$ is that the second summand only becomes small for $\tau\leq \tau_\ast-12\log\hat{Z}(X)$. In this subsection, we derive a technically improved error estimate that already kicks in for $\tau \leq \tau_\ast - \frac{7}{2}\log \hat{Z}(X)$.\footnote{For later applications it will be enough that the quantity $\log\!\left( \hat{Z}(X)^2\right)$ appears with a prefactor strictly less than $2$.}\\

To begin with, using Theorem \ref{thm:fine_neck_L2} (weak differential neck theorem) as an input, we can improve the exponent $q$ in the propagation estimate from Theorem \ref{thm:propagation} (anisotropic propagation of smallness).
   
\begin{proposition}[anisotropic propagation of smallness II]\label{thm:quad.propagation}
For $\tau\leq \tau_s(X)$, $p\in [0,q-8\beta^2]$ and  $\alpha \in [ |\tau+2\log \hat Z(X)|^{-5} , \beta^2]$ we have
\begin{equation}
\sup_{\{ r_\alpha\leq |\tau+2\log \hat Z(X)|^{1/2}\}}e^{-p \tau}e^{-\frac{1-\beta}{2}r_\alpha^2} |\nu_1|
 \leq  \sup_{\tau'\leq \tau}\sup_{\{r_\alpha\leq L\}}  e^{-p \tau'}\left(|\nu_1|^2+\hat{Z}(X)^{4(1-\frac{3}{4}\beta)}e^{2(1-\frac{3}{4}\beta)\tau'}\right)^{1/2}.
\end{equation}
\end{proposition}

\begin{proof} By scaling we can assume $\hat{Z}(X)=1$. Given $\bar{t}\leq e^{-\tau_s(X)}$ and  $\alpha\geq (\log |\bar{t}|)^{-5}$, for any $t\leq \bar{t}$
 the set $\{ f_\alpha(\cdot,t)\leq \log|t|\}$ is contained in $\{ |x_1|\leq |t|^{1/2} (\log|t|)^3\}$.
Hence, applying Theorem \ref{thm:fine_neck_L2} (weak differential neck theorem)  and Corollary \ref{cor:gradient_estimate_bootstrap} (combined propagation of smallness) with $p=\frac{1}{2}-7\beta^2$ we get 
\begin{equation}\label{eq:global.rough.slope.bound}
\sup_{\{f_\alpha(\cdot,t)\leq \,\log |t|\}}|\nu_1|\leq  |t|^{-\frac{1}{2}+8\beta^2}.
\end{equation}
Also observe that the a priori assumption $\mathrm{AP}_{\alpha}$ from \eqref{ap_ass} is satisfied thanks to Theorem \ref{thm:cap_distance} (cap distance).

Now, given any $\eps>0$, we consider the functions from the proof of Theorem \ref{thm:propagation} (anisotropic propagation of smallness) without the $|t|^{-2q}$-term, namely we consider
\begin{equation}
\tilde{w}^\eps= \log  (|\nu_1|-\eps\zeta)_+ -\frac{\beta}{4} f_\alpha +\Lambda+p\log |t|,\qquad \tilde{\psi}^\eps=\frac{\eta \tilde{w}^\eps}{\log|t|}.
\end{equation}
We now choose $\Lambda=- \log K+\beta \log |\bar t|$, where
\begin{align}
K=\sup_{t\leq \bar t}\sup_{\{f_\alpha(\cdot,t)\leq L^2\}}|t|^{p}(|\nu_1|^2+|t|^{-1+2\beta-2q})^{1/2}.
\end{align}
Note that as a consequence of the definitions we have $\Lambda\leq(\tfrac12 +q-p)\log|\bar{t}|$. Hence,
using the estimate \eqref{eq:global.rough.slope.bound} we infer that $\tilde{w}^\eps \leq (q+8\beta^2-\frac{1}{4}\beta)\log |\bar t|$ on $\{f_\alpha=\log|t|\}$ for all $t\leq \bar{t}$. Remembering $\eta=\beta/q$ on $\{f_\alpha=\log|t|\}$, this shows that $\tilde{\psi}^\eps \leq \beta$ on $\{f_\alpha=\log|t|\}$ for all $t\leq \bar{t}$. Therefore, after checking the Cauchy condition as in the proof of Theorem \ref{thm:propagation} (anisotropic propagation of smallness), Remark \ref{rem_mod_quant} (maximum principle for modified quantity) yields $\tilde{\psi}^\eps \leq \beta$. Since $\eps >0$ was arbitrary, this implies the assertion.
\end{proof}

As a consequence, we can improve Corollary \ref{cor:gradient_estimate_bootstrap} (combined propagation of smallness) as well.

\begin{corollary}[combined propagation of smallness II]\label{cor:quad.gradient.bootstrap}
For $\tau\leq \tau_s(X)$ we have
\begin{equation}
\sup_{\{    |y_1| \leq  |\tau+2\log Z(X)|^2\}}|\nu_1| \leq    \hat{Z}(X)^{-12\beta^2} e^{(\frac{1}{2}-6\beta^2)\tau} \sup_{\tau'\leq \tau}\sup_{\{ |y| \leq L/\beta\}} e^{-\frac{1}{2} \tau'} |\nu_1|+ \hat{Z}(X)^{2(1-\beta)}e^{(1-\beta)\tau}.
\end{equation}
\end{corollary} 
 
\begin{proof}
By scaling we can assume $\hat{Z}(X)=1$.  Applying Proposition \ref{thm:quad.propagation} (anisotropic propagation of smallness II), similarly as in the proof of Corollary \ref{thm:smallness_center} (improved anisotropic propagation of smallness),  we get
\begin{equation}
\sup_{\{\bar f_0 \leq \beta^2|\tau|\}}  |\nu_1| \leq  e^{p\tau} \sup_{\tau'\leq \tau}\sup_{\{ |y| \leq L/\beta\}} e^{-(p+\beta^2)\tau'}|\nu_1|+e^{(1-\frac{3}{4}\beta -\beta^2)\tau}.
\end{equation}
Iterating this argument four more times, we obtain
\begin{equation}
\sup_{\{\bar f_4 \leq \beta^2|\tau|\}}  |\nu_1| \leq   e^{p\tau} \sup_{\tau'\leq \tau}\sup_{\{ |y| \leq L/\beta\}} e^{-(p+5\beta^2)\tau'}|\nu_1|+5 e^{(1-\frac{3}{4}\beta -5\beta^2)\tau}.
\end{equation}
Choosing $p=\frac{1}{2}-5\beta^2$ and applying Theorem \ref{thm:horizontal_propagation} (level set propagation of smallness) with $b=2$, this implies the assertion.
\end{proof} 
 
 Combining the above results we can now prove the main result of this section.
 
\begin{theorem}[strong differential neck theorem]\label{thm:strong_differential_neck}
There exists a constant $a=a(\mathcal{M})\neq 0$, such that for all $X\in\mathcal{M}$ the function $w(y,\tau)=\eta(|y|/\rho(\tau)) u_1^X(y,\tau)$ for all $\tau\leq \tau_*-2\log \hat{Z}(X)$ satisfies
\begin{equation}
\big\|w-ae^{\frac{1}{2}\tau}\big\|_{\mathcal{H}}\leq \eps(\tau)e^{\frac{1}{2}\tau},\qquad\mathrm{where}\;\; \eps(\tau)\leq\frac{C}{|\tau+2\log \hat{Z}(X)|^{\frac{1}{2}}} \left(|a|+\hat{Z}(X)^{\frac{7}{3}}e^{\frac{2}{3}\tau}\right).
\end{equation}
In particular, for all $\tau \leq \tau_*-\frac{7}{2}\log \hat Z(X)$ we have
\begin{equation}
\|w-a e^{\frac{1}{2}\tau}\|_{\mathcal{H}} \leq \frac{C(|a|+1)}{|\tau+2\log \hat Z(X)|^{\frac{1}{2}}}e^{\frac{1}{2}\tau}.
\end{equation}
\end{theorem}

\begin{proof}
By Theorem \ref{thm:fine_neck_L2} (weak differential neck theorem) and Theorem \ref{thm:non_degenerate_neck} (nondegenerate neck theorem)  there exists a constant $a=a(\mathcal{M})\neq 0$, such that setting $\hat{\tau}=\tau+2\log \hat{Z}(X)$ and $\hat{a}(X)=a\hat{Z}(X)^{-1}$ we have
\begin{equation}\label{weak_DNT_reformulated}
\|w-\hat{a} e^{\frac{1}{2}\hat{\tau}}\|_{\mathcal{H}} \leq C|\hat{a}||\hat{\tau}|^{-\frac{1}{2}}e^{\frac{1}{2}\hat{\tau}}+Ce^{\frac{3}{5}\hat{\tau}}.
\end{equation}
Also recall that by Proposition \ref{prop:W_+_evol_improved} (evolution of expansion coefficient) and Theorem \ref{thm_unst_dom} (neutral slope mode cannot dominate) the expansion coefficient $a_+=\langle w, 1\rangle_{\mathcal{H}}$ satisfies
 \begin{equation}
\left|\tfrac{d}{d\tau} a_+-\tfrac{1}{2} a_+ \right| \leq C|\hat{\tau}|^{-\frac{3}{2} }|a_+|+Ce^{\frac{6}{25}\hat{\tau}}\omega +Ce^{\hat{\tau}}.
\end{equation}
Now, using \eqref{weak_DNT_reformulated} and Corollary \ref{cor:quad.gradient.bootstrap} (combined propagation of smallness II) we see that
\begin{equation}\label{eq:quad.propagation.error}
\omega \leq  C |\hat{a}|e^{(\frac{1}{2}-6\beta^2)\hat{\tau}}+e^{(\frac{3}{5}-6\beta^2)\hat{\tau}},
\end{equation}
hence
\begin{equation}
\left|\tfrac{d}{d\tau} a_+-\tfrac{1}{2} a_+ \right| \leq C|\hat{\tau}|^{-\frac{3}{2} }|a_+|+C|\hat{a}|e^{\frac{3}{5}\hat{\tau}}+C e^{\frac{4}{5}\hat{\tau}}.
\end{equation}
We integrate this, and then apply Theorem \ref{thm_unst_dom} (neutral slope mode cannot dominate) again to get
\begin{equation}
\|w-\hat{a}e^{\frac{1}{2}\hat{\tau}}\|_{\mathcal{H}}  \leq  C |\hat{a}||\hat{\tau}|^{-\frac{1}{2}}e^{\frac{1}{2}\hat{\tau}}+Ce^{\frac{4}{5}\hat{\tau}}.
\end{equation}
Iterating the above argument, we can improve \eqref{eq:quad.propagation.error} to
\begin{equation}
\omega \leq C  |\hat{a}|e^{(\frac{1}{2}-6\beta^2)\hat{\tau}}+e^{(1-\beta)\hat{\tau}},
\end{equation} 
and thus we finally get
\begin{equation}
\|w-\hat{a}e^{\frac{1}{2}\hat{\tau}}\|_{\mathcal{H}} \leq C |\hat{a}||\hat{\tau}|^{-\frac{1}{2}}e^{\frac{1}{2}\hat{\tau}}+Ce^{\frac{6}{5}\hat{\tau}}.
\end{equation}
Since $\frac{6}{5}>\frac{1}{2}+\frac{2}{3}$, this implies the assertion.
\end{proof}

\bigskip

\section{Classification of solutions with bounded bubble-sheet scale}\label{sec_bdd}

In this section, we first show that if the bubble-sheet scale blows up, then after rescaling we have convergence to $\mathbb{R}\times$2d-oval. Using this, we then classify all solutions with bounded bubble-sheet scale.\\

As usual, we denote by $\mathcal{M}=\{M_t\}$ a strictly convex ancient noncollapsed mean curvature flow in $\mathbb{R}^{4}$, whose tangent flow at $-\infty$ is given by \eqref{bubble-sheet_tangent_intro} and whose bubble-sheet function satisfies \eqref{DH_asymp}. By Theorem \ref{thm:fine_neck_L2} (weak differential neck theorem) and Theorem \ref{thm:non_degenerate_neck} (nondegenerate neck theorem) there is a nonvanishing constant $a=a(\mathcal{M})$ associated to our flow, such that for any space-time point $X$ the profile function $u^X$ of the renormalized flow $\bar{M}_\tau^X$ centered at $X$ for all  $\tau$ sufficiently negative depending on $Z(X)$ satsifies
\begin{equation}\label{eq_diff_bubb_exp}
u_1^X = ae^{\tau/2} + o(e^{\tau/2}).
\end{equation}
Here, by standard interior estimates the expansion holds in $C^k$-norm on compact subsets.  Moreover, possibly after replacing $x_1$ by $-x_1$, we can assume without loss of generality that $a>0$.\\

The expansion \eqref{eq_diff_bubb_exp} and convexity imply that $\inf_{M_t}x_1(p)>-\infty$. By strict convexity this infimum is attained at a unique point $p_t^-$. In case $M_t$ is compact there also is a unique point $p_t^+$ where  $x_1$ is maximal. Now, for $h\in \mathcal{\mathbb{R}}$ we denote by $\mathcal{T}(h)$ the time when a cap arrives at level $x_1=h$. More precisely, in the noncomapct case this is the unique $t$ such that $x_1(p_t^-)=h$, and in the compact case this is the unique $t$ such that either $x_1(p_t^-)=h$ or $x_1(p_t^+)=h$. Hence, for each $h\in \mathbb{R}$ there is a unique space-time point
\begin{equation}
P_h = (p_{\mathcal{T}(h)}^\ast,\mathcal{T}(h)),
\end{equation}
which we call the cap at level $h$. Note that this is well-defined in both the compact and noncompact case.

\begin{proposition}[bubble-sheet scale]\label{prop_dichotomy}
If $Z(P_{h_i}) \to \infty$, then the flows that are obtained from $\mathcal{M}$ by shifting $P_{h_i}$
 to the origin and parabolically rescaling by $Z(P_{h_i})^{-1}$ converge to $\mathbb{R}\times$2d-oval with extinction time $0$. In particular, the ratio between $Z$ and the curvature scale $H^{-1}$ diverges, namely $(ZH)(P_{h_i})\to \infty$.
\end{proposition}

\begin{proof}
Suppose $h_i$ is a sequence converging to infinity or minus infinity, such that $Z(P_{h_i})\to \infty$. Let $\mathcal{M}^i$ be the flows that are obtained from $\mathcal{M}$ by shifting $P_{h_i}$ to the origin, and parabolically rescaling by $Z(P_{h_i})^{-1}$. Then, by \cite[Theorem 1.14]{HaslhoferKleiner_meanconvex} along a subsequence we can pass to an ancient noncollapsed limit flow $\mathcal{M}^\infty$, whose tangent flow at $-\infty$ is a bubble-sheet. 

By \cite{DH_shape,DH_no_rotation} associated to $\mathcal{M}^\infty$ there is a fine-bubble sheet matrix $Q^\infty$.
If $Q^\infty$ had rank 2, then applying \cite[Theorem 1.4]{DH_shape} we would obtain a contradiction with \eqref{eq_diff_bubb_exp} for large $i$. If $\mathcal{M}^\infty$ was a round shrinking $\mathbb{R}^2\times S^1$, then if it became extincting at time $0$ we would obtain a contradiction with the definition of the bubble-sheet scale, and if it became extinct at some later time we would obtain a contradiction with the fact that $M^\infty_t\cap (\mathbb{R}^2\times \{0\})$ is contained in a halfspace for $t>0$. If $\mathcal{M}^\infty$ is $\mathbb{R}\times$2d-oval, then by the same halfspace argument it must become extinct at time $0$, hence $\lim_{i\to \infty} (ZH)(P_{h_i})=\infty$.

To conclude, by \cite[Theorem 1.2 and Theorem 1.3]{DH_shape} the only remaining potential options are that our limit $\mathcal{M}^\infty$ is either $\mathbb{R}\times$2d-bowl or is strictly convex with $\textrm{rk}(Q^\infty)=1$. Therefore, either by the explicit form or by Theorem \ref{thm:fine_neck_L2} (weak differential neck theorem) and Theorem \ref{thm:non_degenerate_neck} (nondegenerate neck theorem) there is a constant $a^\infty\neq 0$ associated to $\mathcal{M}^\infty$, such that \eqref{eq_diff_bubb_exp} holds.
However, this contradicts the fact that $a(\mathcal{M}^i)=Z(P_{h_i})^{-1}a(\mathcal{M})\to 0$, and thus finishes the proof of the proposition.
\end{proof}

Let us first address the potential scenario of exotic ovals in $\mathbb{R}^4$, i.e. compact ancient noncollapsed flows in $\mathbb{R}^{4}$, whose tangent flow at $-\infty$ is a bubble-sheet and whose bubble-sheet matrix has rank 1. 

\begin{theorem}[no bounded exotic ovals]\label{class_comp_bdd}
There are no exotic oval in $\mathbb{R}^4$ with $\liminf_{h \to \infty}Z(P_h) < \infty$.
\end{theorem}

\begin{proof}Recall that we arranged that $a>0$, and that by \eqref{bubble-sheet_tangent_intro} we have
\begin{equation}
\liminf_{t\to -\infty} \inf_{p\in M_t\cap \{x_1=x_2=0\}} (-2t)^{-1/2}|p|\geq 1.
\end{equation}
Hence, thanks to the assumption $\liminf_{t\to -\infty}Z(p^+_{t},t) < \infty$ and convexity, considering the expansion \eqref{eq_diff_bubb_exp} with center $(p^+_t,t)$ for $t$ very negative we obtain a contradiction with the fact that $a>0$.
\end{proof}

Let us now consider the noncompact case.

\begin{theorem}[noncompact solutions with bounded bubble-sheet scale]\label{thm_noncpt_bdd}
If $\liminf_{h\to -\infty}Z(P_h) < \infty$ and $\liminf_{h\to \infty}Z(P_h) < \infty$, then our flow $\mathcal{M}$ is selfsimilarly translating.
\end{theorem}

\begin{proof}
Without loss of generality we normalize such that $a(\mathcal{M})=2^{-1/2}$. In the following argument we will frequently use that, since $\partial_t H\geq 0$, the function $h\mapsto H(P_h)$ is monotone.

First, thanks to $\liminf_{h\to -\infty}Z(P_h) < \infty$  and Hamilton's Harnack inequality \cite{Hamilton_Harnack}, for some $h_i\to -\infty$ the sequence $\mathcal{M}^{i}=\mathcal{M}-P_{h_i}$ converges to a bubble-sheet translator, which satisfies \eqref{eq_diff_bubb_exp} with $a=2^{-1/2}$, and hence translates with speed $1$. This implies $\lim_{h\to -\infty}H(P_h)= 1$.

Next, if we had $Z(P_{h_i})\to 0$ along some sequence $h_i\to \infty$, then $\mathcal{M}^{i}=\mathcal{M}-P_{h_i}$ would converge to a round shrinking $\mathbb{R}^2\times S^1$, contradicting the expansion \eqref{eq_diff_bubb_exp}. This shows $\liminf_{h\to \infty}Z(P_h)>0$.

If we had $H(P_{h_i})\to \infty$ along some sequence with $Z(P_{h_i})$ bounded, then the flows $\mathcal{M}^{i}=\mathcal{M}-P_{h_i}$ would converge to an ancient noncollapsed limit flow $\mathcal{M}^\infty$, that becomes extinct at time $0$, and whose tangent flow at $-\infty$ is a bubble-sheet. Hence, arguing as in the proof of Proposition \ref{prop_dichotomy} (bubble-sheet scale), and observing that a round shrinking $\mathbb{R}^2\times S^1$ is now ruled out in light of the previous paragraph, we would infer that $\mathcal{M}^\infty$ is an $\mathbb{R}\times$ 2d-oval, which would again contradict the expansion \eqref{eq_diff_bubb_exp}. This shows that $\lim_{h\to \infty}H(P_h)<\infty$.

Consequently, thanks to $\liminf_{h\to \infty}Z(P_h) < \infty$  and Hamilton's Harnack inequality \cite{Hamilton_Harnack}, for some $h_i\to \infty$ the sequence $\mathcal{M}^{i}=\mathcal{M}-P_{h_i}$ converges to a bubble-sheet translator, which satisfies \eqref{eq_diff_bubb_exp} with $a=2^{-1/2}$, and hence translates with speed $1$. This implies $\lim_{h\to \infty}H(P_h)= 1$.
  
We have thus shown that $H(P_h)\equiv 1$, and this implies the assertion. Indeed, recall that by Hamilton's Harnack inequality \cite{Hamilton_Harnack} for every vector field $V$ we have
\begin{equation}\label{Ham_Harnack}
\partial_t H + 2\nabla_V H + A(V,V)\geq 0.
\end{equation}
By inserting $V=-A^{-1}\nabla H$, in addition to $H(P_h)\equiv 1$, we get that $\nabla H (P_h)\equiv 0$.
Therefore, by the rigidity case of Hamilton's Harnack inequality  \cite{Hamilton_Harnack} our ancient solution $\mathcal{M}$ is a translator.
\end{proof}

To conclude this section let us record the following corollary.

\begin{corollary}[convergence to 3d-bowl]\label{prop_conv_to_bowl}
If $\lim Z(P_h)=\infty$, then the flows that are obtained from $\mathcal{M}$ by shifting $P_h$
 to the origin and parabolically rescaling by $H(P_h)$ converge to the round 3d-bowl.
\end{corollary}

\begin{proof}
Using Proposition \ref{prop_dichotomy} (bubble-sheet scale) we infer that $\lim (ZH)(P_h)=\infty$, and that any subsequential limit is a noncompact ancient noncollapsed flow $\mathcal{M}^\infty$, whose tangent flow at $-\infty$ is not a bubble-sheet. 
 Observe also that by construction the 0-times slice is contained in a halfspace, and that $H(0,0)=1$. Hence, the uniqueness result from \cite{BC2} implies the assertion.
\end{proof}

\bigskip

\section{Ruling out solutions with unbounded bubble-sheet scale}\label{sec_unbdd}

In this final section, we rule out the potential scenario that $Z(P_h)\to \infty$ for $h\to \infty$ or $h\to -\infty$.\\

Throughout this section, $\mathcal{M}=\{M_t=\partial K_t\}$ denotes a strictly convex ancient noncollapsed mean curvature flow in $\mathbb{R}^{4}$, whose tangent flow at $-\infty$ is given by \eqref{bubble-sheet_tangent_intro}, and whose bubble-sheet function satisfies \eqref{DH_asymp}. Moreover, without loss of generality we can normalize such that $\mathcal{T}(0)=0$ and $a(\mathcal{M})=2^{-1/2}$.

\begin{convention}[constants] In this final section all constants may depend on $\mathcal{M}$.
\end{convention}

We begin by observing that if $Z(P_h)\to \infty$ for $h\to \infty$ or $h\to -\infty$, then the bubble-sheet scale and the quadratic scale are comparable. Here, recalling that $\eps_0>0$ is a fixed small constant, we define the quadratic scale as $Q(X)=e^{-\frac{1}{2}\tau_q(X)}$, where $\tau_q(X)$ denotes the largest $\tau\in (-\infty,\tau_s(X)]$, such that
\begin{equation}\label{kappa_quadratic}
\left\| u^X(y,\vartheta,\tau)-\frac{y_2^2-2}{\sqrt{8}|\tau+2\log \hat{Z}(X)|} \right\|_{\mathcal{H}}\leq \frac{\eps_0}{|\tau+2\log \hat{Z}(X)|}.
\end{equation}

\begin{proposition}[quadratic scale]\label{prop_quad_scale}There exists $Z_\star<\infty$, such that if $Z(P_h)\geq Z_\star$, then we have
\begin{equation}\label{eq_quad_sc}
Q(P_h)\leq C Z(P_h),
\end{equation}
and the flow that is obtained from $\mathcal{M}$ by shifting $P_h$
 to the origin and parabolically rescaling by $H(P_h)$ is $\beta$-close to the round 3d-bowl, 
and furthermore for $\mathcal{T}(h)-t\leq C Z(P_h)^{2}$ we also have the diameter bound
\begin{equation}\label{radius_bound_level}
\sup_{p,p'\in M_t\cap \{x_1=h\}}|p-p'|\leq C (\mathcal{T}(h)-t)^{1/2}.
\end{equation}
\end{proposition}

\begin{proof}
Recall from Proposition \ref{prop_dichotomy} (bubble-sheet scale) that if $Z(P_h)\to \infty$, then we have $(HZ)(P_h)\to \infty$ and the flows $\mathcal{M}^h$ that are obtained from $\mathcal{M}$ by shifting $P_{h}$ to the origin, and parabolically rescaling by $Z(P_{h})^{-1}$ converge to $\mathbb{R}\times$2d-oval, which becomes extinct at time $0$ and satisfies $Z(0,0)=1$. Together with the uniqueness result from \cite{ADS2} this implies \eqref{eq_quad_sc}. Next, the $\beta$-closeness to the round 3d-bowl follows from Corollary \ref{prop_conv_to_bowl} (convergence to 3d-bowl).
Furthermore, by quantitative differentiation \cite{CHN_stratification}, the flow $\mathcal{M}-P_h$ is $\eps_0$-close to a neck away from a controlled number of scales between the curvature scale $H(P_h)^{-1}$ and the bubble-sheet scale $Z(P_h)$. Finally, by the curvature estimate from \cite{HaslhoferKleiner_meanconvex} the mean curvature only changes by a bounded factor over a controlled number of scales. This implies \eqref{radius_bound_level}.
\end{proof}

As a consequence, let us observe that renormalized time relative to the bubble-sheet scale of the caps is captured by the eccentricity of the level sets. Here, we define the eccentricity of the $h$-level,
\begin{equation}
\varsigma(h,t)=\frac{\pi}{4}\frac{\mathcal{W}^2}{\mathcal{A}}(h,t),
\end{equation}
as ratio of the square of the width
\begin{equation}
\mathcal{W}(h,t)= \sup_{K_t\cap \{x_1=h\}}x_2-\inf_{K_t\cap \{x_1=h\}}x_2, 
\end{equation}
and the maximal section area
\begin{equation}
\mathcal{A}(h,t)=\sup_{x}\mathcal{H}^2\!\left(K_t\cap \{x_1=h,x_2=x\}\right).
\end{equation}
Note that on the oval-bowls, normalized as above, by the asymptotics from \cite{CHH_translator} we have $\varsigma(h,t)\sim \log(h-t)$. Motivated by this, we will now control the quantity $\tau_h+2\log Z(P_h)$, where
\begin{equation}
\tau_h = -\log(-t_h),\qquad\mathrm{with}\;\, t_h=t-\mathcal{T}(h).
\end{equation}
 
 \begin{corollary}[eccentricity]\label{prop:eccentricity.scale}
For every $\varepsilon>0$, there exists $N=N(\varepsilon)<\infty$, such that for all $(h,t)$ satisfying $Z(P_h)\geq Z_\star$ and $\varsigma(h,t)\geq N$ we have 
\begin{equation}
(1-\varepsilon)  \varsigma(h,t) \leq  -\tau_h-2\log Z(P_h) \leq (1+\varepsilon) \varsigma(h,t).
\end{equation}
Moreover, the estimate also holds if we assume $\tau_h+2\log Z(P_h)\leq -N$ in lieu of $\varsigma(h,t)\geq N$.
 \end{corollary}
 
 \begin{proof}
By Proposition \ref{prop_quad_scale} (quadratic scale), the condition $\max\{\varsigma(h,t), -\tau_h-2\log Z(P_h)\}\gg 1$ guarantees that $\tau_h+2\log Q(P_h)\ll 0$. Therefore, we can obtain the assertion by applying Lemma \ref{prop_ell_dom} (ellipsoidal domains) and Lemma \ref{thm:width_zero_level} (central width) below.
\end{proof}

In the above proof we used the following standard lemma.

\begin{lemma}[central width]\label{thm:width_zero_level}
For every $\delta>0$, there exists $\tau_\delta>-\infty$, such that
\begin{equation}
\sup_{\bar{M}^X_\tau\cap \{ y_1=0 \}} |y_2|  \leq (1+\delta)\sqrt{2| {\tau+2\log \hat Q(X)}|}
\end{equation}
for $\tau\leq \tau_\delta-2\log \hat Q(X)$, where $\hat Q(X)=\max \{Q(X),1\}$.
\end{lemma}

\begin{proof} 
As in the proof of Corollary \ref{cor_switch_time} (switch time), we see that \eqref{kappa_quadratic} improves going backwards in time.
Hence, thanks to Lemma \ref{prop_ell_dom} (ellipsoidal domains) and the almost symmetry, c.f. \eqref{curv_bd_57}, we can use the argument for establishing the upper bound in \cite[Proposition 3.10]{CHH_wing}, by replacing $v(y,\tau)$ with
\begin{equation}
v(y,\vartheta,\tau)=\big(\sqrt{2}+u(0,y,\vartheta,\tau)\big)^2-2+|\tau+2\log \hat Q(X)|^{-9}.
\end{equation}
This implies the assertion.
\end{proof}

In the following subsections we will write $Z_h=Z(P_h)$ to denote the bubble-sheet scale at level $h$.

\bigskip 

\subsection{Controlling the translator ratio}

In this subsection, we consider the translator ratio
\begin{equation}
\Theta=\frac{-\nu_1}{H}.
\end{equation}

We start with the following technical lemma, which will be used to verify the assumptions of Corollary \ref{cor_geom_cons} (geometric bounds) and Theorem \ref{thm:strong_differential_neck} (strong differential neck theorem) centered at $X=P_h$.

 \begin{lemma}[shifted time]\label{prop:activate_strong_neck}
For every $\eps>0$ and $\tau_0>-\infty$, there exists $\tau_\star=\tau_\star(Z_\star,\eps,\tau_0)>-\infty$, such that if $Z_0 \geq Z_\star$, then for all $\tau \leq \tau_\star-\frac{11}{3}\log Z_0$ and all $|h|\leq 2 e^{\frac{1}{2}|\tau|}|\tau+2\log Z_0|^2$ we have
 \begin{align}
|\tau_h-\tau| \leq \eps,\qquad \tau_h\leq \tau_s(P_h) \qquad\mathrm{and}\qquad \tau_h \leq \tau_0-\tfrac{7}{2}\log \hat{Z}_h.
 \end{align}
\end{lemma}

\begin{proof}
To begin with, note that Theorem \ref{thm:cap_distance} (cap distance) and $\partial_t H\geq 0$ yield $|\mathcal{T}(h)|\leq \frac{1}{2}\eps |t|$ for $|h|\leq e^{\frac{1}{2}|\tau|}|\tau+2\log Z_0|^9$, which in particular implies $|\tau_h-\tau| \leq \eps$. Also note that by Theorem \ref{thm:cap_distance} (cap distance) and convexity, setting $K^h_t=K_t\cap \{x_1=h\}$, for all $h$ in the considered range we have 
\begin{equation}\label{eq:congruent_level}
(1-|\tau+2\log Z_0|^{-7})K_t^0\subset K_t^h \subset (1+|\tau+2\log Z_0|^{-7})K_t^0.
\end{equation}

Also note that by choosing $\tau_\star\leq\tau_0-1-\tfrac72 \log Z_\star$ the estimate $\tau_h \leq \tau_0-\tfrac{7}{2}\log \hat{Z}_h$ holds when $Z_h<Z_\star$. Suppose now $Z_h\geq Z_\star$. By \eqref{eq:congruent_level} the eccentricity barely changes if we change the level, specifically
\begin{equation}
(1-\beta)\varsigma(0,  t) \leq \varsigma (  h,  t) \leq (1+\beta)\varsigma(0,  t).
\end{equation}
Together with Corollary \ref{prop:eccentricity.scale} (eccentricity) and $|\tau_h-\tau|\leq 10^{-3}$, this yields
\begin{equation}\label{eq_change_log_Z}
|\log Z_h-\log Z_0| \leq 1+2\beta\varsigma(0,  t)\leq 1 + 3\beta |\tau +2\log Z_0|.
\end{equation}
Remembering that $\tau \leq \tau_\star-\frac{11}{3}\log Z_0$ by assumption, this implies $\tau_h \leq \tau_0-\tfrac{7}{2}\log \hat{Z}_h$.

Finally, if we had $\tau_{  h} > \tau_s(P_h)$, then Theorem \ref{thm_elong_cyl} (elongated cylindrical region) and $|\tau_h-\tau| \leq 10^{-3}$ would yield
$W(  h,  t)^2  \geq 39 |t|(\log|t|-2 \log \hat Z_h)$, and hence \eqref{eq:congruent_level} and \eqref{eq_change_log_Z} would imply $W( 0,  t)^2  \geq 10 |t|\log|t|$, contradicting the upper bound from Lemma \ref{thm:width_zero_level} (central width), which is applicable thanks to Proposition \ref{prop_quad_scale} (quadratic scale). This concludes the proof of the lemma.
\end{proof}

\begin{proposition}[translator ratio in central region]\label{lem:center_ratio_estimate}
Suppose that $Z_0 \geq Z_\star$. Then, for every $\varepsilon>0$ there exists $\tau^\varepsilon_\star =\tau_\star^\eps(Z_\star) >-\infty$, such that for all $\tau \leq   \tau^\varepsilon_\star-\frac{11}{3}\log  Z_0$ we have
\begin{equation}
\sup_{\{|x_1|\leq 2|\tau+2\log  Z_0|^2e^{|\tau|/2}, |x_2|\leq 10e^{|\tau|/2}\}}\left|\Theta(x,-e^{-\tau})-1\right|\leq \varepsilon.
\end{equation}
\end{proposition}

\begin{proof}
Let $|h| \leq 2e^{|\tau|/2}|\tau+2\log Z_0|^2$. Thanks to Lemma \ref{prop:activate_strong_neck} (shifted time) we can apply Theorem \ref{thm:strong_differential_neck} (strong differential neck theorem) for $\mathcal{M}^{P_h}$, so that for $\tau\leq \tau^\eps_\star(Z_\star)-\frac{11}{3}\log Z_0$ we get
\begin{equation}
\sup_{\{|y| \leq 100\}}\left|u_1^{P_h} -2^{-\frac{1}{2}}e^{\frac{1}{2}\tau_h}\right| \leq \frac{\eps}{10}e^{\frac{1}{2}\tau_h}.
\end{equation}
Moreover, again by Lemma \ref{prop:activate_strong_neck} (shifted time) we can arrange that
\begin{equation}\label{sh_very_cl}
|\tau_h-\tau|\leq\eps/10,\qquad \tau_h\leq \tau_s(P_h),
\end{equation}
and Corollary \ref{cor_geom_cons} (geometric bounds) yields that the maximal section area in the definition of $\mathcal{A}(h,t)$ is attained at $x_2=x_{\max}(h,t)$ with
\begin{equation}\label{x2_max}
\left|x_{\max} (h,t)-x_2\big(p^\ast_{\mathcal{T}(h)}\big)\right| \leq 10|t_h|^{1/2}.
\end{equation}
Finally, using \eqref{sh_very_cl}, we see that $H$ is almost equal to $2^{-1/2}e^{\frac{1}{2}\tau}$ and that $\nu_1$ is almost equal to $u_1^{P_h}$, respectively.
Combining these estimates, the assertion follows.
\end{proof}

\begin{theorem}[translator ratio]\label{thm:speed_level}
Suppose that $Z_h \geq Z_\star$. Then, whenever $\tau_h \leq   \tau^\varepsilon_\star-\frac{11}{3}\log  Z_h$ we have
\begin{equation}
\sup_{\{|x_1-h|\leq |\tau_h +2\log   Z_h|^2e^{|\tau_h|/2}\}}\left|\Theta(x,-e^{-\tau})-1\right|\leq 2\varepsilon,
\end{equation}
\end{theorem}

 \begin{proof}
After replacing $\mathcal{M}$ by $\mathcal{M}-P_h$ we can assume without loss of generality that $h=0$. Similarly as in the proof of Theorem \ref{thm:horizontal_propagation} (level set propagation of smallness) we consider the region
 \begin{equation}
\Omega_t=\{|x_1| \leq 2|t|^{1/2}(\log |t|-2\log Z_0)^2, x_2 \geq 10|t|^{1/2}\}.
\end{equation}
We would like to compare $\varphi_\pm=\pm\nu_1\pm(1\mp\varepsilon)H$, which solves the Jacobi equation, with $\delta \nu_2 + \Psi_2$, which is a supersolution thanks to Proposition \ref{prop_Jac_super_long} (Jacobi supersolution). To this end, note first that by Proposition \ref{lem:center_ratio_estimate} (translator ratio in central region) we have $\varphi_\pm\leq 0$ on the horizontal boundary $\Omega_t\cap \{x_2=10|t|^{1/2}\}$, and also recall that $\Psi_2=1$ on the vertical boundary $\Omega_t\cap \{|x_1|=2|t|^{1/2} (\log |t|-2\log Z_0)^2\}$. Hence, using Corollary \ref{cor_geom_cons} (geometric bounds), which is applicable thanks to Lemma \ref{prop:activate_strong_neck} (shifted time), we see that 
$\varphi_\pm\leq \delta \nu_2 + \Psi_2$ on $\partial\Omega_t$, and see that the Cauchy condition is satisfied as well. Thus, applying the maximum principle we conclude that $\varphi_\pm\leq \delta \nu_2 + \Psi_2$ in $\Omega_t$. Since $\delta>0$ was arbitrary, this implies the assertion.
 \end{proof}

Finally, let us record the following corollary in terms of the eccentricity.

\begin{corollary}[translator ratio]\label{cor:speed_level_ecc}
For every $\varepsilon>0$, there is $N=N(\varepsilon,Z_\star)<\infty$ with the following significance. If $Z_h\geq Z_\star$, then for every $t\leq \mathcal{T}(h)$ satisfying $\varsigma(h,t) \geq N+\frac{7}{4}\log Z_h$ we have
\begin{equation}
\sup_{\{x_1=h\}} \left|\Theta(x,t)-1\right|\leq \varepsilon .
\end{equation}
\end{corollary}

\begin{proof}Thanks to Corollary \ref{prop:eccentricity.scale} (eccentricity) this follows from Theorem \ref{thm:speed_level} (translator ratio).
\end{proof}

\bigskip

\subsection{Conclusion in the compact case}

In this subsection, we rule out exotic ovals in $\mathbb{R}^4$. We start by computing how the maximal section area and the width change, when we change the level.

\begin{proposition}[area and width]\label{thm:width_derivative}
For every $\varepsilon>0$, there exists $N=N(\varepsilon,Z_\star)<\infty$, such that for $(h,t)$ satisfying $Z_h\geq Z_\star$ and $\varsigma(h,t)\geq N+\frac{7}{4}\log Z_h$, we have
\begin{equation}
\left|\partial_h \mathcal{A}(h,t)-2\pi \right|\leq \eps\qquad \mathrm{and}\qquad \left|\partial_h \log \mathcal{W}(h,t)-\frac{1}{2|t_h|}\right|\leq\frac{\eps}{|t_h|}.
\end{equation}
\end{proposition}

\begin{proof}Recalling that the maximal section area $\mathcal{A}(h,t)$ is attained at $x_2=x_{\max}(h,t)$, we compute
\begin{equation}
\partial_h\mathcal{A}(h,t)=\int  |t_h|^{1/2}(\sqrt{2}+u^{P_h} )u^{P_h}_1(0,y_2,\vartheta,\tau_h)\, d\vartheta,
\end{equation}
where $y_2=|t_h|^{-1/2}(x_2-x_2\big(p^\ast_{\mathcal{T}(h)}))$ according to \eqref{x2_max} satisfies $|y_2|\leq 10$. Now, by Corollary \ref{cor:speed_level_ecc} (translator ratio), taking also into account Corollary \ref{prop:eccentricity.scale} (eccentricity), the integrand is very close to $1$, so we obtain
$|\partial_h\mathcal{A}(h,t)-2\pi|\leq \eps$.

Regarding the width, denoting by $q^\pm_{h,t}$ the tips at level $h$, namely the points in $M_t\cap \{x_1=h\}$ for which $x_2$ is maximal or minimal, respectively, we have
\begin{equation}
\partial_h\mathcal{W}(h,t)=\left(|\nu_1(q^+_{h,t},t)|^{-2}-1\right)^{-1/2}+\left(|\nu_1(q^-_{h,t},t)|^{-2}-1\right)^{-1/2}.
\end{equation}
Now, observing that $0<-\nu_1(q_{h,t}^\pm,t) \leq \varepsilon/10$ and applying Corollary \ref{cor:speed_level_ecc} (translator ratio), we see that figuring out the quantity on the right hand side amounts to computing the mean curvature at the tips. To this end, note that by 
Lemma \ref{prop_ell_dom} (ellipsoidal domains) and Lemma \ref{thm:width_zero_level} (central width), which is applicable thanks to Proposition \ref{prop_quad_scale} (quadratic scale), we have
\begin{equation}
\left| \frac{\tfrac12 \mathcal{W}(h,t)}{\sqrt{2 |t_h|(\log|t_h|-2\log Z_h)}}-1\right| \leq\frac{\eps}{100}.
\end{equation}
This expression can be differentiated in $t$ thanks to Hamilton's Harnack inequality \cite{Hamilton_Harnack}, yielding
\begin{equation}
\left| H(q_{h,t}^+,t)+H(q_{h,t}^-,t)-\frac{\mathcal{W}(h,t)}{2|t_h|}\right| \leq\frac{\eps}{10} \frac{\mathcal{W}(h,t)}{|t_h|}.
\end{equation}
Combining these facts,  the assertion follows.
\end{proof}

\begin{theorem}[no exotic ovals]\label{thm:compact_conclusion}
There exist no exotic ovals in $\mathbb{R}^4$.
\end{theorem}

\begin{proof}
Suppose towards a contradiction $\mathcal{M}=\{M_t\}$ is an exotic oval in $\mathbb{R}^4$. We normalize as usual. Moreover, by Theorem \ref{class_comp_bdd} (no bounded exotic ovals) we can assume that $Z_h\geq Z_\star$ for $h\geq h_\star$.

Given $N=N(\eps,Z_\star)<\infty$ sufficiently large, for $t\ll 0$ we consider the transition height
\begin{equation}
h_N(t)=\inf\left\{ h\geq h_\star  \, : \, \varsigma(h,t)\leq \tfrac{7}{4}\log (Z_h)+N\right\}.
\end{equation}
Thanks to the $\beta$-closeness to the round 3d-bowl from Proposition \ref{prop_quad_scale} (quadratic scale) the infimum is attained at some $h<x_1(p_t^+,t)$, hence
\begin{equation}\label{w_per_def}
\mathcal{W}(h_N(t),t)^2= \left(14\log( Z_{h_N(t)})+8N\right)\frac{\mathcal{A}(h_N(t),t)}{2\pi}.
\end{equation}

Integrating the differential equation for $\mathcal{A}$ from Proposition \ref{thm:width_derivative} (area and width), and taking into account that $\lim_{t\to -\infty} |t|^{-1}\mathcal{A}(h_\star,t)= 2\pi$, for all $h\in [h_\star,h_N(t)]$ we get
\begin{equation}\label{area_2p_eps}
\left|\frac{\mathcal{A}(h,t)}{h-t}-2\pi \right|\leq \varepsilon.
\end{equation}
For later use, let us also observe that as a consequence of \eqref{area_2p_eps}, in light of Proposition \ref{prop_quad_scale} (quadratic scale) and Proposition \ref{prop:eccentricity.scale} (eccentricity scale), for all $h\in [h_\star,h_N(t)]$ we obtain
\begin{equation}\label{for_later_use}
\left| \frac{t_h}{t-h}-1\right| \leq 2\eps\qquad \mathrm{and}\qquad \log Z_{h}\leq \tfrac{1}{2}(-t_h)\leq \tfrac{1}{2}\log(h-t)+1.
\end{equation}

Now, from \eqref{w_per_def} and \eqref{area_2p_eps} we get
\begin{equation}\label{eq:threshold_level}
\mathcal{W}(h_N(t),t)^2\leq \big((14+\beta)\log( Z_{h_N(t)})+9N\big)(h_N(t)-t).
\end{equation}

On the other hand, integrating the differential equation for $\mathcal{W}$ from Proposition \ref{thm:width_derivative} (area and width), and taking into account  the first inequality from \eqref{for_later_use} and $\lim_{t\to -\infty}(|t|\log|t|)^{-1}\mathcal{W}(h_\star,t)^2= 8$, we get
\begin{equation}\label{eq:w_otherhand}
\mathcal{W}(h_N(t),t)^2 \geq (8-\beta)|t|\log|t|\left(\frac{h_N(t)-t}{-t}\right)^{1-5\eps}.
\end{equation}

Combining these two inequalities yields
\begin{equation}\label{contr_eq}
\log\left(\frac{(14+\beta)\log(Z_{h_N(t)})+9N}{(8-\beta)\log |t|}\right) \geq -5\varepsilon \log (1+|h_N(t)/t|).
\end{equation}
To conclude, note that by Hamilton's Harnack inequality \cite{Hamilton_Harnack} we have $\limsup_{t\to -\infty}|t|^{-1}h_N(t)<\infty$.
Thus, by choosing $\eps>0$ small enough we can arrange that the right hand side of \eqref{contr_eq} is bigger than $\log(15/16)$, but on the other hand by using the second inequality from \eqref{for_later_use} we get $\log Z_{h_N(t)}\leq \tfrac{1}{2}\log|t|+C$.
For $t$ sufficiently negative this gives the desired contradiction, and thus proves the theorem.
\end{proof}
 
 \bigskip

\subsection{Conclusion in the noncompact case}

In this subsection, we conclude the argument in remaining case when our flow $\mathcal{M}=\{M_t\subset\mathbb{R}^4\}$ is noncompact. We normalize as usual.

\begin{proposition}[sharp lower bound for cap speed]\label{thm:speed_lower_bound}
We have $\lim_{h\to -\infty} H(P_h)\geq 1$.
\end{proposition}

\begin{proof}
Recall that we have already seen in the proof of Theorem \ref{thm_noncpt_bdd} (noncompact solutions with bounded bubble-sheet scale) that $\liminf_{h\to -\infty} Z_h<\infty$ implies $\lim_{h\to -\infty} H(P_h)=1$.

Suppose now there is some $h_\star>-\infty$, such that $Z_h\geq Z_\star$ for $h\leq h_\star$. Given $N=N(\eps,Z_\star)<\infty$ sufficiently large, for $t\ll 0$ we consider the transition height
\begin{equation}
h_N(t)=\sup\left\{ h\leq h_\star  \, : \, \varsigma(h,t)\leq \tfrac{7}{4}\log (Z_h)+N\right\},
\end{equation}
which thanks to the $\beta$-closeness to the round 3d-bowl from Proposition \ref{prop_quad_scale} (quadratic scale) is attained at some $h>-x_1(p_t,t)$.

Now, suppose towards a contradiction that $\lim_{t\to -\infty}H(p_t,t)<1$. Together with Hamilton's Harnack inequality \cite{Hamilton_Harnack}, this implies
\begin{equation}\label{h_N_over_t}
\limsup_{t\to -\infty}|h_N(t)/t| < 1.
\end{equation}
Therefore, integrating the differential equations for $\mathcal{A}$ and $\mathcal{W}$ from Proposition \ref{thm:width_derivative} (area and width), similarly as in the proof of Theorem \ref{thm:compact_conclusion} (no exotic ovals), we still get \eqref{eq:threshold_level}, namely
\begin{equation}\label{eq:threshold_level_again}
\mathcal{W}(h_N(t),t)^2\leq \big((14+\beta)\log( Z_{h_N(t)})+9N\big)(h_N(t)-t),
\end{equation}
while on the other hand in lieu of \eqref{eq:w_otherhand} we now obtain
\begin{equation}\label{eq:w_otherhand_lieu_of}
\mathcal{W}(h_N(t),t)^2 \geq (8-\beta)|t|\log|t|\left(\frac{h_N(t)-t}{-t}\right)^{1+5\eps}.
\end{equation}

Combining these two inequalities yields
\begin{equation}\label{contr_eq_lieu_of}
\log\left(\frac{(14+\beta)\log(Z_{h_N(t)})+9N}{(8-\beta)\log |t|}\right) \geq 5\varepsilon \log (1-|h_N(t)/t|).
\end{equation}
To conclude, remembering \eqref{h_N_over_t}, by choosing $\eps>0$ small enough we can arrange that the right hand side of \eqref{contr_eq_lieu_of} is bigger than $\log(15/16)$. In light of $\log Z_{h_N(t)}\leq \tfrac{1}{2}\log|t|+1$ this gives the desired contradiction, and thus proves the proposition.
\end{proof}
 
 \begin{theorem}[noncompact solutions]\label{thm:noncompact_conlusion_unbounded}
Every noncompact strictly convex ancient noncollapsed flow in $\mathbb{R}^4$ is selfsimilarly translating.
\end{theorem}

\begin{proof}
By the reduction from Section \ref{sec_not_and_prel} (Notation and preliminaries), we can assume that our flow $\mathcal{M}=\{M_t=\partial K_t\subset\mathbb{R}^4\}$ has a a bubble-sheet tangent at $-\infty$ given by \eqref{bubble-sheet_tangent_intro}, and that its bubble-sheet function satisfies \eqref{DH_asymp}. Also, as usual we can normalize such that $\mathcal{T}(0)=0$ and $a(\mathcal{M})=2^{-1/2}$.

By Proposition \ref{thm:speed_lower_bound} (sharp lower bound for cap speed) it is enough to show that $H(p_t,t)\leq 1$ for $t$ sufficiently negative. In case $\liminf_{h\to \infty} Z_h <\infty$ this has already been accomplished in the proof of Theorem \ref{thm_noncpt_bdd} (noncompact solutions with bounded bubble-sheet scale). Fixing a large enough constant $Z_\star=Z_\star(\mathcal{M})<\infty$, such that in particular $Z_\star > Z_0$, we can thus suppose towards a contradiction that there is some $h_\star \in (0,\infty)$, such that
\begin{equation}\label{eq:assumption_large_scale}
Z_h \geq Z_\star \quad \text{holds for all}\quad h\geq h_\star.
\end{equation}
We begin by establishing the following global curvature bound.

\begin{claim}[global curvature bound]\label{claim_glob_curv_bound} There exists $T_\star>-\infty$, such that for all $t\leq T_\star$ we have
\begin{equation}
\sup_{\{x_1\geq h_\star\}}H(x,t) \leq C|t|^{-1/2}(\log |t|)^{1/2}.
\end{equation}
\end{claim}

\begin{proof}
By Proposition \ref{prop_quad_scale} (quadratic scale) and Corollary \ref{prop:eccentricity.scale} (eccentricity) the diameter of $M_t\cap \{x_1=h\}$ is bounded by $C|t_h|^{1/2}$ in case $\log|t_h|-2\log Z_h\leq C$ and is bounded by $C|t_h|^{1/2}(\log|t_h|)^{1/2}$ in case $\log|t_h|-2\log Z_h > C$ thanks to Lemma \ref{thm:width_zero_level} (central width). Hence, the Harnack estimate \eqref{Harnack_curv} yields
\begin{equation}
\sup_{\{x_1=h\}}H(x,t) \leq C|t_h|^{-1/2}(\log |t_h|)^{1/2}.
\end{equation}
Together with the observation that $|t_h|\geq |t|$ for $h\geq 0$ and $t\leq 0$, this establishes the claim.
\end{proof}

Next, to construct a lower barrier for the translator ratio $\Theta$ at levels $h\geq h_\star$, we consider the solution $\psi$ of the heat equation $\psi_t=\psi_{xx}$ in the halfspace $\mathbb{R}_+$ with Dirichlet boundary condition $\psi(0,t)=0$ and initial condition $\psi(x,0)=1$, namely
\begin{equation}
\psi(x,t)=\frac{1}{\sqrt{4\pi t}}\int_0^\infty \left(e^{-\frac{(x-y)^2}{4t}}-e^{-\frac{(x+y)^2}{4t}}\right)dy.
\end{equation}
Note that $0\leq\psi\leq1$, with $\lim_{x\to \infty}\psi(x,t)=1$ and $\lim_{t\to \infty}\psi(x,t)=0$. Moreover, as checked in \cite[Propositon 6.9]{BC} the function $\psi$ is strictly concave, namely $\psi_{xx}< 0$. In particular, it follows that $\psi_x>0$.
 
\begin{claim}[translator ratio lower bound]\label{prop_slope_lower}
For every $\delta>0$, there exists $T_\delta \in (-\infty,T_\star]$, such that for all $t\leq T_\delta$ we have
\begin{equation}
\inf_{ \{ x_1\geq h_\star\} }\Theta(x,t)\geq 1-\delta.
\end{equation}
\end{claim}
 
\begin{proof}
By Theorem \ref{thm:speed_level} (translator ratio) there is some $T_\delta \in (-\infty,T_\star]$, such that for all $t\leq T_\delta$ we have
\begin{equation}\label{eq_from_lemm}
\inf_{\{|x_1-h_\star|\leq |t|^{1/2}(\log|t|)^2\}  }\Theta(x,t)\geq 1-\delta/2.
\end{equation}
Now, fixing $s< T_\delta$, we consider the function
\begin{equation}
\varphi^s(x,t)=1-\delta-\psi\big(2x_1-2h_\star-|t|^{1/2}(\log |t|)^2,t-s\big).
\end{equation}
Using \eqref{eq_from_lemm} and the above properties of $\psi$, we see that $\Theta> \varphi^s$ holds for $t$ close to $s$. In addition, for each $t>s$ we have $\varphi^s(x,t)<0$ for sufficiently large $x_1$.
Suppose towards a contradiction, that at some time $t\in (s,T_\delta]$ at some point $p\in \Omega_t=\{ 2x_1-2h_\star > |t|^{1/2}(\log|t|)^2\}$ we have
\begin{align}
\Theta=\varphi^s, \quad \nabla \Theta= \nabla\varphi^s, \quad\mathrm{and}\quad   (\partial_t -\Delta)\Theta \leq (\partial_t -\Delta)\varphi^s.
\end{align}
Note that
\begin{equation}
(\partial_t -\Delta)\Theta=2\nabla \log  H \cdot \nabla \Theta.
\end{equation}
Moreover, we compute
\begin{equation}
(\partial_t -\Delta) \varphi^s \leq  -\psi_t-\tfrac{1}{2}\psi_x|t|^{-1/2}\left[(\log |t|)^2+4\log |t|\right]+4|\nabla x_1|^2\psi_{xx}.
\end{equation}
Remembering that $\psi_t=\psi_{xx}<0$ and observing also that $|\nabla x_1|\geq 1/2$, this yields
\begin{equation}
2\nabla \log H \cdot ( - 2\psi_x \nabla x_1) \leq -\tfrac{1}{2}\psi_x|t|^{-1/2}\left[(\log |t|)^2+4\log |t|\right].
\end{equation}
Using the estimate $|\nabla H|\leq CH^2$ from \cite{HaslhoferKleiner_meanconvex} and Claim \ref{claim_glob_curv_bound} (global curvature bound) 
this implies
\begin{equation}
\psi_x|t|^{-1/2}\left[\tfrac12 (\log |t|)^2+2\log |t|-C(\log|t|)^{1/2}\right]\leq 0.
\end{equation}
This contradicts $\psi_x>0$, and thus shows that $\Theta>\varphi^s$ in $\Omega_t$ for all $t \in (s,T_\delta]$.
Passing $s\to -\infty$ we conclude that $\Theta\geq 1-\delta$ in $\Omega_t$ for all $t\leq T_\delta$. This proves the claim.
 \end{proof}

Continuing the proof of the theorem, by Claim \ref{prop_slope_lower} (translator ratio lower bound), possibly after decreasing $T_\star$, we have $H\leq 2|\nu_1|$ in $\{x_1 \geq h_\star\}$ for $t\leq T_\star$. 
 Together with the noncollapsing property it follows that $M_t$ can be expressed as entire graph in $x_1$-direction, namely 
\begin{equation}
M_t=\left\{(f(\xi,t),\xi):\xi\in \mathbb{R}^3\right\}.
\end{equation}
To proceed, note that plugging the vector field $V=-\frac{H}{\nu_1}e_1^\top$ into Hamilton's Harnack expression \eqref{Ham_Harnack} yields that $f_{tt}\geq 0$ everywhere. Hence, given times $t_1\leq t_2$ and any bounded domain $\Omega \subset\mathbb{R}^3$, we have
\begin{equation}
\sup_{\Omega}f_t(\cdot,t_1) \leq \sup_{\partial \Omega}f_t(\cdot,t_2).
\end{equation}
Also recall that by the graphical mean curvature flow equation we have $f_t=H/(-\nu_1)$.
Thus, we can extend the bound from Claim \ref{prop_slope_lower} (translator ratio lower bound) to the region $\{x_1 < h_\star\}$, yielding
\begin{equation}
\sup_{t\leq T_\delta}\sup_{M_t}\frac{H}{-\nu_1}\leq \frac{1}{1-\delta}.
\end{equation}
Finally, by the maximum principle this bound can be propagated forward in time. Since $\delta>0$ was arbitrary, this yields that $H\leq |\nu_1|$ for all $t\leq T_\star$. In particular, this shows that $H(p_t,t)\leq 1$ for all $t\leq T_\star$, so by the rigidity case of Hamilton's Harnack inequality \cite{Hamilton_Harnack} our flow $\mathcal{M}$ is selfsimilarly translating. In particular, $Z_h=Z_0$ for all $h$. Remembering that $Z_\star>Z_0$, this contradicts \eqref{eq:assumption_large_scale}, and thus concludes the proof of the theorem.
\end{proof}

\bigskip

\bibliography{ancient}

\bibliographystyle{alpha}

\vspace{10mm}

{\sc Kyeongsu Choi, School of Mathematics, Korea Institute for Advanced Study, 85 Hoegiro, Dongdaemun-gu, Seoul, 02455, South Korea}\\

{\sc Robert Haslhofer, Department of Mathematics, University of Toronto,  40 St George Street, Toronto, ON M5S 2E4, Canada}\\

\emph{E-mail:} choiks@kias.re.kr, roberth@math.toronto.edu

\end{document}